\documentclass[notitlepage,reqno,11pt]{amsart}
\usepackage{latexsym,amssymb, epsfig, amsmath,amsfonts, subfigure,amsthm}

\usepackage{rotating}
\usepackage[toc,page]{appendix}
\usepackage{color}
\usepackage{multirow}
\usepackage{relsize}
\usepackage{microtype}

\usepackage{bm}

\usepackage[colorlinks, allcolors=blue]{hyperref}

\usepackage[margin=1in]{geometry}

\renewcommand{\S}  {\bar{S}}
\newcommand{\I}  {\bar{I}}
\newcommand{\rR} {\bar{R}}

\numberwithin{equation}{section}
 \newtheorem{assumption}{Assumption}[section]
\newtheorem{lemma}{Lemma}[section]
\newtheorem{theorem}{Theorem}[section]

\newtheorem{coro}{Corollary}[section]
\newtheorem{prop}{Proposition}[section]
\newtheorem{remark}{Remark}[section]

\newlength{\defbaselineskip}
\newcommand{\setlinespacing}[1]%
           {\setlength{\baselineskip}{#1 \defbaselineskip}}

\usepackage[numbers]{natbib}

\newcommand{\RR}{{\mathbb R}}

\newcommand{\NN}{{\mathbb N}}

\def\E{\mathbb{E}}
\def\P{\mathbb{P}}

\newcommand{\sF}{{\mathcal{F}}}

\newcommand{\Var}{\text{\rm Var}}
\newcommand{\Cov}{\text{\rm Cov}}

\newcommand{\R}  {\mathbb{R}}
\newcommand{\N}  {\mathbb{N}}

\newcommand{\mfI}{\mathfrak{I}}
\newcommand{\mfi}{\mathfrak{i}}

\newcommand{\bD}{{\mathbf D}}

\newcommand{\bS}{{\mathbf S}}
\newcommand{\bE}{{\mathbf E}}

\newcommand{\bC}{{\mathbf C}}

\newcommand{\bR}{{\mathbf R}}
\newcommand{\bV}{{\mathbf V}}

\newcommand{\bI}{{\mathbf I}}

\newcommand{\bU}{{\mathbf U}}

\newcommand{\bone}{{\mathbf 1}}

\newcommand{\qforq}{\quad\mbox{for}\quad}

\newcommand{\qasq}{\quad\mbox{as}\quad}
\newcommand{\qinq}{\quad\mbox{in}\quad}

\newcommand{\non}{\nonumber}
\newcommand{\RA}{\Rightarrow}

\newcommand{\ttl}{\Large Recent Advances in Epidemic Modeling: Non--Markov Stochastic \\[5pt] Models and their Scaling Limits}

\begin{document}
	
	\title[]{\ttl}
	
	\author[Rapha\"el \ Forien]{Rapha\"el Forien}
	\address{INRAE, Centre INRAE PACA, Domaine St-Paul - Site Agroparc
		84914 Avignon Cedex
		FRANCE}
	\email{raphael.forien@inrae.fr}
	
\author[Guodong \ Pang]{Guodong Pang}
\address{Department of Computational Applied Mathematics and Operations Research,
George R. Brown College of Engineering,
Rice University, Houston, TX~~77005}
\email{gdpang@rice.edu}
	
	\author[{\'E}tienne \ Pardoux]{{\'E}tienne Pardoux}
	\address{Aix--Marseille Universit{\'e}, CNRS, Centrale Marseille, I2M, UMR \ 7373 \ 13453\ Marseille, France}
	\email{etienne.pardoux@univ.amu.fr}

	
	\newcommand{\raph}[1]{\textcolor{magenta}{#1}}
	
	\maketitle
	
	\allowdisplaybreaks

		\begin{abstract}
	In this survey paper, we review the recent advances in individual based non--Markovian epidemic models. They include epidemic models with a constant infectivity rate, varying infectivity rate or infection-age dependent infectivity, infection-age dependent recovery rate (or equivalently, general law of infectious period), as well as varying susceptibility/immunity. 
	We focus on the scaling limits with a large population,  functional law of large numbers (FLLN) and functional central limit theorems (FCLT), while the large and moderate deviations for some Markovian epidemic models are also reviewed. In the FLLN, the limits are a set of Volterra integral equations, and for the models with infection-age dependent infectivity, the limit becomes a PDE coupled with the Volterra integral equations.
	In the FCLT, the limits are stochastic Volterra integral equations driven by Gaussian processes. We relate our deterministic limits to the results in the seminal papers by Kermack and McKendrick  published in 1927, 1932 and 1933, where the varying infectivity and susceptibility/immunity were already considered. 
	We also discuss some extensions, including models with heterogeneous population, spatial models and control problems, as well as open problems. 
	\end{abstract}

	\section{Introduction}
	
	Had this paper been written in 2019 or before, we would have needed, in order to drive the interest of the reader, to recall the major pandemics of the previous centuries: the black plague pandemic which killed between 30\% and 50\% of Europe's population in the 14th century, the plague epidemic which killed almost half of the population of Marseille and a quarter of the population of Provence in 1720, the so--called Spanish flu which killed between 50 and 100 million people, not forgetting HIV/AIDS, malaria and tuberculosis, which together killed more than 3 million humans in 2011. 
	This long list is not without a silver lining, thanks to the huge success of vaccination, which in particular has reduced the number of deaths due to measles by 94\% and has permitted the eradication of smallpox in the late 20th century.

	However, in 2021, everyone has heard about infectious diseases, the basic reproduction number, herd immunity and the importance of vaccination. This is a result of the Covid--19 pandemic, which started at the end of 2019 in China, and by spring 2020 had hit Europe and North America, filling intensive care units in every country. During the spring of 2020, many countries implemented drastic lockdown measures, which were decided after the leaders of those countries had learned the predictions of mathematical models about the number of deaths that the epidemic was likely to cause, if no such measures were taken. One year later, in the spring of 2021, most wealthy countries are striving to vaccinate a large proportion of their population, in the hope that they can get rid of the epidemic, or at least lower the pressure on hospitals to a manageable level. 
	
	The aim of this survey paper is not primarily to present all those notions, or to review all the efforts aimed at modelling this particular pandemic, but rather to highlight some recent progress in epidemic modelling to which the authors of this paper have contributed.

	The use of mathematics and mathematical models as a tool to understand and control the propagation of infectious diseases has a long history. Around 1760, Daniel Bernoulli, a member of a famous family of mathematicians, who had also been trained as a physician, exploited a mathematical model in order to  convince his contemporaries of the advantage of inoculation (the ancestor of vaccination) against smallpox, discussing already the balance between the  benefit of inoculation  and the associated risk, a question which is much debated these days. The foundations of modern epidemic modelling was mainly the result of efforts of physicians, rather than mathematicians. During the second half of the 19--th century, the Russian physician P. D. En'ko was probably the first scientist who created a chain binomial stochastic model of epidemics, similar to the better known Reed--Frost model, which was formulated in 1928, but published only in the 1950's. 
	
	Most modern mathematical epidemic models are formulated as deterministic compartmental models. Around 1910, Ross introduced the concept of the basic reproduction number $R_0$, and argued that malaria would stop if the proportion of mosquitos to humans were maintained under a certain threshold. His arguments,  which were based on the understanding of the large time behaviour of dynamical systems, were not accepted by many of his contemporaries, who claimed that malaria would continue as long as mosquitos would be present. 
	
	The description of the transmission of communicable diseases via compartmental models was pioneered by Kermack and McKendrick in a series of three papers, published in 1927, 1932 and 1933, see \cite{KMK,KMK32,KMK33}. Their models were very refined. In their 1927 paper, they consider infection--age dependent infectivity (\textit{i.e.,} the infectivity of an infectious individual depends upon the time elapsed since he/she was infected), as well as infection--age dependent recovery rate which, as we shall explain below, corresponds to the fact that the duration of the infectious periods can be very general (any absolutely continuous distribution).  
In one section of the paper, they consider the case of constant rates, \textit{i.e.}, constant infectivity, and constant recovery rate, the latter imposing that the law of the duration of the infectious period be an exponential distribution. In this particular case, the deterministic model is a system of ordinary differential equations (ODEs), instead of the more complex system of  Volterra type  integral equations in the general case.
Let us explain what kind of Volterra integral equations will appear in the present paper. Note that an ODE can be written in integral form as
\[ x(t)=y(t)+\int_0^t f(s,x(s))ds,\]
where $x(t)$ is the solution, and $y(t)$ a given forcing term, while $f(t,x)$ is the coefficient of the ODE, which can be written in differential form as
\[ \frac{dx}{dt}(t)= \frac{dy}{dt}(t)+f(t,x(t)),\quad x(0)=y(0)\,.\]
The class of Volterra integral equations which will appear in this paper is of the general form
\[ x(t)=y(t)+\int_0^t f(t,s,x(s))ds,\]
where now the coefficient $f$ depends upon the upper bound of the integral. Hence there is no simple equation for 
$\frac{dx}{dt}(t)$. More importantly, while, given $x(t)$, the solution of the ODE after time $t$ depends only upon the future increments of $y$ after time $t$, this is no longer the case for the solution of the Volterra integral equation, whose solution does not forget its past. Some specific forms of such equations are called delay equations, or equations with memory.
Note that there is a well--established theory for such equations with in particular results of  existence and uniqueness under appropriate conditions on the coefficients, see e.g. \cite{brunner_volterra_2017}.

It is rather clear that the general model considered in \cite{KMK} can be made much more realistic by a proper choice of the coefficients adapted to each particular situation (both to a specific illness and to a specific society with its interactions) than the particular case of constant rates. However, almost all epidemic models which were considered since 1927 treat the special case of ODE models, and when the seminal paper \cite{KMK} is quoted, most of the time reference is made only to the special case of constant rates. 

Furthermore in their second and third papers \cite{KMK32,KMK33},  Kermack and McKendrick considered the loss of immunity (and study the endemic situations), again with a very realistic point of view: they consider that loss of immunity is not sudden, but progressive. It is hard to find a recent work which adopts that point of view (one exception is \cite{inaba2001kermack}). In their 1932 paper \cite{KMK32}, they also pioneer the use of PDE models for the description of epidemic models with infection-age dependent infectivity and recovery rate and recovery-age dependent level of immunity.

The goal of the present survey paper is twofold. First, we want to draw the attention of the readers to the complex models of Kermack and McKendrick, the more classical ODE models being in our opinion a rather unrealistic approximation of the former, which should be used only when both have a similar behaviour. We shall discuss that point below. Second we want to derive the deterministic models (whether ODEs, integral equations or PDEs) as law of large numbers limits, in the asymptotic of a large population,  of individual based stochastic models. 
This can be seen as an analogue of many recent works which establish certain equations of physics as limits of stochastic particle systems, as the number of particles tends to infinity (see in particular the book by Kipnis and Landim \cite{KL}, and the references therein).
We shall also discuss the difference between the stochastic and the deterministic models, via the central limit theorem, moderate and large deviation principles.

\subsection{Literature review}
The Markov models and their limiting ODE models have long been the standard tools to study epidemics, see the recent survey \cite{brittonpardoux} and monographs \cite{anderson1992infectious,andersson2012stochastic,martcheva2015,bjornstad2018,BCF-2019}. 
We refer the readers to the above for the related literature. The functional law of large numbers (FLLN) and functional central limit theorem (FCLT) results for Markov models can be found in \cite{brittonpardoux}, which use the standard Poisson random measure representations and martingale convergence arguments in \cite{ethier-kurtz}. 
Most relevant results concerning the large deviation and moderate deviation principles in Section~\ref{sec:markov} can be found in the recent works  \cite{pardoux2017large,kratz2018large,pardoux2018large,brittonpardoux,pardouxEJP20}. 

Since the seminal work by Kermack and McKendrick \cite{KMK,KMK32,KMK33}, Volterra integral equations have been developed, without proving an FLLN rigorously, in various epidemic models with general infectious periods, see, e.g., \cite{brauer1975nonlinear,cooke1976epidemic,diekmann1977limiting,hethcote1995sis,van2000simple,feng2007epidemiological,BCF-2019}. Also,  PDE models have been developed,  for various epidemic models, Markovian or non-Markovain, see, e.g., 
\cite{thieme1993may,inaba2004mathematical,magal2013two,hoppensteadt1974age,magal2010lyapunov, ZhangPeng2007,Kaplan2020,Gaubert,reinert1995asymptotic,clemenccon2008stochastic,foutel2020individual}. 

This survey does not focus on these various deterministic Volterra integral equations and PDE models, but on individual based infectious models for which such equations arise as scaling limits, in particular, our recent works in \cite{PP-2020,FPP2020a,FPP2020b,PP-2020b,PP-2020c,PP-2021,FPPZ2021}.

Let us also provide a brief literature review on some other existing methods and results on these general models. 
Sellke \cite{sellke1983asymptotic} developed an approach, the so-called ``Sellke construction", to find the distribution of the number of remaining uninfected individuals in an $\bS\bI\bR$/$\bS\bE\bI\bR$ epidemic model with a large population. 
Ball \cite{ball1986unified} developed a unified approach using a Wald's identity to find the distribution of the total size and the total area under the trajectory of infectious individuals.  
Ball \cite{ball1993final} further developed this approach to study multi-type epidemic models. Barbour  \cite{barbour1975duration} proved limit theorems for the distribution of the duration from the first infection to the last removal in a closed epidemic. 
The LLN and CLT results concerning the final number of infected individuals are  presented in \cite{brittonpardoux}. 
Non--Markovian SIR epidemic models were studied as piecewise Markov deterministic processes using the associated martingales in  \cite{clancy2014sir,gomez2017sir} to analyze the distribution of the number of survivors of the epidemic and  the population transmission number and the infection probability of a given susceptible individual.

We now review the existing literature on functional limit theorems for the non-Markovian epidemic models preceding our recent works. 
Wang \cite{wang1975limit,wang1977central,wang1977gaussian} proved FLLN and FCLT for some age and density dependent stochastic population models, including the SIR model which has an infection rate depending on the number of infectious individuals and allows an initial condition with infection-age dependent infectivity. The limits are deterministic or stochastic Volterra integral equations. 
The strategy of the proof in \cite{wang1975limit,wang1977central,wang1977gaussian} is different from ours, does not make use of Poisson random measures, and also assumes a $C^1$ condition on the distribution of the infectious period for the FCLT.  
Some asymptotic properties of the limiting deterministic integral equations were  studied \cite{wang1978asymptotic}. 
Very few of these articles have proved an FLLN  with a PDE limit. These include \cite{reinert1995asymptotic,clemenccon2008stochastic,foutel2020individual} among the papers mentioned above. 
Reinert \cite{reinert1995asymptotic} used Stein's method with a generalized Sellke construction to prove a LLN for the empirical measure describing the system dynamics of the generalized SIR model with the infection rate dependent upon time and state of infection, and a PDE model can be derived from the limit. 
Cl{\'e}mon\c{c}on et al. \cite{clemenccon2008stochastic} proved an FLLN for the measure-valued process in the SIR model with contact-tracing, from which a PDE model is derived (this paper also establishes an FCLT with a SPDE limit). 
 In developing a PDE model to study the Covid-19 pandemic, \cite{foutel2020individual} establishes the PDE model as a a law of large numbers limit of stochastic individual based models.
 Although not particularly studying epidemic models, there have been studies on population dynamics using measure-valued processes which result in PDE limits, for example, \cite{oelschlager1990limit,bose1995diffusion,meleard2012slow}. Their results do not directly apply to epidemic models, but their methods of proving the convergence of measure-valued processes might be useful for our topic.

The models and main results in our recent works are reviewed in this article, which are briefly described in the next subsection. 

\subsection{Organization of the paper}

The organization of the paper is as follows.
In Section~\ref{sec:markov}, we consider the ``special case'' of constant rates, which leads to a Markov individual-based stochastic model, an ODE model in the FLLN and a diffusion model in the FCLT.  We also discuss the large deviation and moderate deviation results in the Markov models. 
In Section~\ref{sec:nonmarkov}, we consider the case of a constant infectivity, but with a general law of the infectious period, or equivalently an infection age dependent recovery rate, in which case the stochastic model is non-Markov, and the deterministic LLN limiting model is a system of integral equations, \textit{i.e.}, a system of equations with memory. 
The stochastic limiting processes in the FCLT are Gaussian-driven Volterra integral equations.  
This section draws on our first paper in the series \cite{PP-2020}. 
We also describe the discrete spatial model, \textit{i.e.,} a multi-patch non-Markov model with constant infectivity, where individuals may migrate from one patch to another and individuals may be infected locally within each patch or from some distance (which can be thought of as the result of infections taking place during short stays of individuals outside of their current patch).  This part draws upon \cite{PP-2020b}.
In Section~\ref{sec:varinfect}, we add the infection--age infectivity. 
At the level of the stochastic model, we assume that the infectivity functions of the various individuals are i.i.d. copies of a given c\`adl\`ag random function. 
It turns out that only the mean of this random function appears in the limiting LLN deterministic model, which, as we shall see, is precisely the model introduced in \cite{KMK}. 
We then present the limiting Gaussian-driven stochastic integral equations obtained in the FCLT. 
We also discuss the use of these models to model the Covid-19 epidemic.  
These results are proved in \cite{FPP2020b,PP-2020c}. 
We shall then discuss the PDE point of view of the same model in Section~\ref{sec:PDE}. 
In addition to the total infectivity process, we use a stochastic process that tracks the number of infected individuals at each time that have been infected for a certain amount of time. 
The LLN limit is again a system of integral equations, while the density of the more detailed process with respect to the elapsed infectious time have a PDE representation that coincides with the equations of \cite{KMK32} if the distribution of the infectious period is absolutely continuous. In addition, we obtain a PDE for the models with a deterministic infectious period. These PDE results are established in \cite{PP-2021}. 
In Section~\ref{sec:varinfect-varimmun} we discuss the case where the infectivity depends upon the age of infection, the duration of the infectious period has a general distribution, and the loss of immunity is a random function of the time elapsed since recovery. Again, those varying immunities of the various individuals are i.i.d. copies of a given random function. It turns out that in this case, the limiting deterministic model involves the whole distribution of this varying immunity function (and not just its expectation), and in general our LLN model is very different from the model introduced in \cite{KMK32}, unless the loss of immunity is described by the same deterministic function of the time elapsed since the random recovery time for all individuals.	 This result is proved in \cite{FPPZ2021}. 
Finally, in Section~\ref{sec:open}, we discuss various extensions, including models with heterogeneous population, spatial models and control problems.   We then discuss open problems.

Starting with Section~\ref{sec:nonmarkov}, this paper describes recent results obtained by the authors, as the output of a research effort which started in January 2020. Needless to say, our program has not yet been completed. In particular, the moderate and large deviations results have so far been obtained only in the Markov case, and little has been done until now on spatial model outside the Markov case. We nevertheless believe that it is now a good time to put together a series of results which both relate stochastic and deterministic results, and insist upon the rich and complex models which Kermack and McKendrick introduced almost a century ago, and have been unfortunately largely neglected and/or forgotten.

 Let us add some comments on the present paper. There are essentially two large classes of epidemic models.
\begin{enumerate}
\item Those models where the number of susceptible individuals only decreases during the epidemic. The individuals who get infected recover sooner or later from the illness, and they have a permanent immunity. Moreover, there is no birth or immigration of new susceptibles. These are the $\bS$ $\bI$ $\bR$ and 
$\bS$ $\bE$ $\bI$ $\bR$ models without demography. In those models, the epidemic clearly cannot last forever.
\item Those models with a permanent flux of new susceptibles, either by birth or immigration, or through the fact that individuals who recover from the illness loose their immunity after some time, and become susceptible again. In that case, under certain conditions the epidemic can in principle last for ever. This is what is called an endemic situation. 
\end{enumerate}
Sections \ref{sec:markov} to \ref{sec:PDE} essentially treat $\bS$ $\bI$ $\bR$ models, with no flux of susceptibles, except for sections \ref{M-endemic}, \ref{M-LD} and \ref{sec:SIS-PDE}.

Section~\ref{sec:varinfect-varimmun} presents our latest results on a $\bS$ $\bI$ $\bR$ $\bS$ type of model where, following Kermack and McKendrick, both the infectivity and the susceptibility depend upon the age of the last infection. The type of model considered in Section~\ref{sec:varinfect-varimmun} is very different from all other models in this paper, and is in fact completely new in the literature of epidemic models.

As already indicated, Section~\ref{sec:markov} studies Markov models and their LLN ODE limits. Sections~\ref{sec:nonmarkov} to~\ref{sec:PDE} present our recent results on 
$\bS$ $\bI$ $\bR$ non--Markov models, and their limits. Sections~\ref{sec:nonmarkov} presents non--Markov models with fixed infectivity, but during a duration which is not assumed to be exponential (this is why the model is non--Markov), while Section~\ref{sec:varinfect} presents the more general situation where the infectivity depends upon the age of infection. Those two sections discuss the convergence to Volterra integral equations. In Section~\ref{sec:PDE}, we present a different point of view, where the limit is a first order partial differential equation. The intuition behind this is that the non--Markov model can be made Markov in infinite dimension, and the limit, instead of being an integral equation with memory, becomes a PDE (i.e., an equation without memory, but in infinite dimension). We have separated the integral equation and the PDE approaches, in order to be more understandable. We have devoted more space to the integral equation approach than to the PDE approach. This is a matter of taste. In the recent years, we have devoted more effort to the integral equation point of view than to the PDE point of view. Other authors prefer the PDE approach. The reader can make his/her own choice.

\subsection{Basic vocabulary of epidemic models}

\paragraph{\bf Compartments}
In a compartmental model, each individual belongs to one of the following compartments:
\begin{itemize}
	\item[-] $\bS$ denotes the compartment of susceptible individuals: those who are not infected, but are susceptible to the disease, which means that they might get infected if they meet an infectious individual.
	
	\item[-] $\bE$ denotes the compartment of exposed individuals, who are infected, but not yet infectious. 
	
	\item[-] $\bI$ denotes the compartment of infectious individuals, who are infected, and able to transmit the disease to susceptible individuals.  Note that when considering the models with varying infectivity, we shall denote by $\bI$ the compartment of infected individuals, whether they are exposed or infectious. Their infectivity is $\ge0$, they are infectious when it is $>0$.
	
	\item[-] $\bR$ denotes the compartment of removed or recovered individuals. Those who have been infected, and have recovered from the disease. They are neither infected, nor susceptible, and are immune to the disease. Often one includes in that compartment those who died from the disease. In the model with varying immunity/susceptibility which we shall consider in Section~\ref{sec:varinfect-varimmun}, we will merge the $\bS$ and $\bR$ compartments into the $\bS$ compartment, each individual having a varying susceptibility. Whenever that susceptibility is $0$, he/she cannot be infected, as when he/she is in the $\bR$ compartment.
\end{itemize}

The above compartments are the most commonly used ones, but certain models consider other compartments,  e.g., $\bV$ for vaccinated. Also, in particular concerning the COVID--19, some authors have decomposed the $\bI$ compartment into several ones, distinguishing, e.g., the symptomatic and the asymptomatic infectious individuals, creating a compartment for those who are hospitalized, and another one for those in Intensive Care Units.
\bigskip

\paragraph{\bf Various types of models}

The most classical model is the $\bS\bI\bR$ model. 
In this model, when a susceptible individual is infected, he/she leaves the $\bS$ compartment and enters the $\bI$ compartment. 
While in that compartment, he/she is infectious and can infect susceptible individuals. 
After some time, the infectious individual recovers, moves to the $\bR$ compartment, and stays there for ever. 
This means that he/she is immune and cannot be infected a second time.  

A variant of the $\bS\bI\bR$ model is the $\bS\bE\bI\bR$ model, in which infected individual first enter the $\bE$ compartment. While in that compartment, the individual is not infectious. He/she becomes infectious when entering the $\bI$ compartment. As in the $\bS\bI\bR$ model, the individual eventually 
recovers when entering the $\bR$ compartment.

Next we have the models where the recovered individuals lose their immunity after some time, and become susceptible again. 
The simplest of those is the $\bS\bI\bS$ (or the $\bS\bE\bI\bS$) model, where 
upon recovery the individual becomes susceptible again, \textit{i.e.,} there is no period of immunity.
Another class of such models is the $\bS\bI\bR\bS$ (or the $\bS\bE\bI\bR\bS$) model, where upon recovery the individual first stays for some time in the $\bR$ compartment, where he/she is immune and cannot be infected. 
Later he/she loses immunity, and becomes susceptible again.

In the simplest models, the size of the population is fixed. This makes sense if the epidemic is considered over a time interval during which there are not many births and deaths, and the deaths due to the epidemic are possibly included in the $\bR$ compartment. Quite a few models however include the demography, and the total size of the population is allowed to fluctuate. The latter are called $ \bS(\bE)\bI\bR(\bS) $ models with demography.

As already explained, contrary to the traditional approach, if we consider an infection--age dependent infectivity,   
we need not distinguish the compartments $\bE$ and $\bI$ (while in $\bE$, the infectivity is zero), and if we consider a recovery age dependent susceptibility, we need not distinguish the compartments $\bR$ and $\bS$ (while in $\bR$, the susceptibility is zero). We shall also see that the non--Markov models / models with memory allow to have a precise description of the propagation of the disease, without increasing the number of compartments, as is commonly done with Markov / ODE models.
\bigskip

\paragraph{\bf $\bR_0$.} A fundamental concept of infectious disease modeling is the so--called {\it basic reproduction number}, denoted $R_0$, which is the mean number of susceptible individuals whom an infectious individual infects during its infectious period, at the beginning of the epidemic (i.e., while essentially all members of the population are susceptible). 

Note that when a significant fraction of the population has been hit by the disease, the mean number of susceptible individuals whom an infectious individual infects during its infectious period will be different, and is sometimes called the effective reproduction number, which depends upon time, since it depends upon the evolution of the epidemic.
More precisely, if $\bar{S}(t)$ denotes the proportion of susceptible individuals in the population, 
$R_{eff}(t)=R_0\times \bar{S}(t)$. If $R_{eff}(t)\le1$, the epidemic regresses and eventually goes extinct. For that to occur, we need to have $\bar{S}(t)\le R_0^{-1}$, i.e., the proportion of immune individuals should be greater than $1-R_0^{-1}$. In such a situation, one says that {\it herd immunity} has been achieved.

\section{Markov and ODE models}\label{sec:markov}

In this section, we shall discuss stochastic Markov models, and their limiting ODE models. In this case, the proof of the LLN is rather easy, and has been known for a long time. Also, the fluctuations of the stochastic model around its law of large numbers limit has been fully studied. Not only do we have a Central Limit Theorem, but also moderate and large deviations have been studied in this simpler case. 

\subsection{The $\bS\bI\bR$ Markov model and its LLN limit}\label{subsec:SIRLLN}

We shall follow the recent presentation in Section 2 of \cite{brittonpardoux} (see also the original paper by Kurtz \cite{Kurtz70}, Barbour \cite{barbour1974functional}
and the book \cite{andersson2012stochastic}). 
Let us first describe the stochastic individual based model. 
We shall use the notions of a Poisson process and a Poisson Random Measure, two notions which are introduced in Subsection~\ref{subsec:Poisson} of the Appendix below.

Suppose that we have a population of fixed size $N$, which is distributed in the three compartments $\bS$, $\bI$ and $\bR$. Let $S^N(t)$ (resp. $I^N(t)$, resp. $R^N(t)$) denote the number of susceptible (resp. infectious, resp. recovered) individuals at time $t$. We have $S^N(t)+I^N(t)+R^N(t)=N$. We assume that each infectious individual meets others at rate $\beta$. If the encountered individual is susceptible, which at time $t$ happens with probability $S^N(t)/N$ (since we make the homogeneity assumption that the individual who is met is chosen uniformly in the population), then the encounter results in a new infection with probability $p$. Hence, if we use the notation $\lambda=\beta\times p$, the rate at which one particular infectious individual infects susceptibles is  $\lambda S^N(t)/N$, and the total rate of new infections in the population at time $t$ is
\[ \Upsilon^N(t)=\lambda I^N(t)\frac{S^N(t)}{N}\,.\]
This means that the number of new infections on the time interval $[0,t]$ takes the form
\[ P_{inf}\left(\int_0^t \Upsilon^N(s)ds\right),\]
where $P_{inf}(t)$ is a standard Poisson process. 
The fact that this process at time $t$ involves $P_{inf}$ up to a random time creates sometimes technical problem. Therefore we shall use in further sections an alternative description which we now introduce. 
Given a standard Poisson random measure $Q$ on $\R_+^2$, an alternative equivalent description of the counting process of infections is
\[ \int_0^t\int_0^\infty{\bf1}_{u\le  \Upsilon^N(s^-)}Q(ds,du)\,. \]

The crucial point now is that in the Markov model considered in this section, we assume that the duration of the infection period is $\mathcal{E}\text{xp}(\gamma)$ (the durations for various individuals are independent, and independent of the rest of the population), where $\gamma^{-1}$ is the mean duration of this infectious period. The end of the infectious period of a given individual is thus the first jump time of a rate $\gamma$ Poisson process.  Since the sum of mutually independent Poisson processes is a Poisson process with rate the sum of the rates, the counting process of the number of recoveries on the interval $[0,t]$ (\textit{\textit{i.e.}}, the number of jumps from the $\bI$ to the $\bR$ compartment) is
\[ P_{rec}\left(\gamma\int_0^t I^N(s)ds\right)\,,\]
where $P_{rec}(t)$ is a standard Poisson process, independent of $P_{inf}(t)$.
Finally, we obtain the following system of stochastic differential equations for the evolution of the numbers of susceptible, infectious and recovered individuals: 
\begin{equation}\label{eq:MSIR}
\left\{
\begin{aligned}
S^N(t)&=S^N(0)-P_{inf}\left(\int_0^t \Upsilon^N(s)ds\right),\\
I^N(t)&=I^N(0)+P_{inf}\left(\int_0^t \Upsilon^N(s)ds\right)-P_{rec}\left(\gamma\int_0^t I^N(s)ds\right),\\
R^N(t)&=R^N(0)+P_{rec}\left(\gamma\int_0^t I^N(s)ds\right)\,.
\end{aligned}
\right.
\end{equation}

\begin{remark}{\bf Value of $R_0$}
It is easy to compute $R_0$ in the present model. ``At the beginning of the epidemic'', means ``while $N^{-1}S^N(t)=1$''.  In that case, each infectious individual infects a mean $\lambda$ susceptible individuals per time unit. The mean duration of the infectious period is $1/\gamma$. Hence $R_0=\lambda/\gamma$.
\end{remark}

We next define $(\bar{S}^N(t),\bar{I}^N(t),\bar{R}^N(t)):=N^{-1}(S^N(t),I^N(t),R^N(t))$. We need to formulate one assumption.

\begin{assumption} \label{AS-SIR-0} We assume that as $N\to\infty$,
\[ (\bar{S}^N(0),\bar{I}^N(0),\bar{R}^N(0))\to(\bar{S}(0),\bar{I}(0),\bar{R}(0))\quad\text{in probability,}\]
where $(\bar{S}(0),\bar{I}(0),\bar{R}(0))\in[0,1]^3$ is such that $\bar{S}(0)+\bar{I}(0)+\bar{R}(0)=1$, 
$\bar{S}(0)>0$ and $\bar{I}(0)>0$.
\end{assumption} 

\begin{remark} If either $\bar{S}(0)=0$ or $\bar{I}(0)=0$, there would be no epidemic in the limiting LLN deterministic model.
Typically an epidemic starts with a small number of initially infectious individuals, which is not of the order of $N$. The description of the first phase of the epidemic, until the number of infectious reaches a ``positive fraction of $N$'' must be done by a stochastic model, the deterministic model becomes valid once a significant fraction of the total population is infectious. The stochastic model at the start of the epidemic can be well approximated by a branching process. We shall explain this below in Section~\ref{sec:varinfect}.
\end{remark}

In the next LLN result, we shall denote by $\bD$ the Skorokhod space of c\`adl\`ag real-valued functions defined on $ \R_+ $, endowed with the Skorokhod $J_1$ topology. The reader is referred to Subsection~\ref{subsec:D} in the Appendix below for its definition and properties.
\begin{theorem}\label{th:MSIR-LLN}
Under Assumption~\ref{AS-SIR-0}, 
\[ (\bar{S}^N,\bar{I}^N,\bar{R}^N)\to(\bar{S},\bar{I},\bar{R}) \qinq \bD^3 \qasq N\to\infty\]
in probability, where $(\bar{S},\bar{I},\bar{R})$ is the unique solution of the system of ODEs
\begin{equation}\label{eq:SIRODE}
\left\{
\begin{aligned}
\frac{d \bar{S}(t)}{dt}&=-\lambda \bar{S}(t)\bar{I}(t),\\
\frac{d \bar{I}(t)}{dt}&=\lambda \bar{S}(t)\bar{I}(t)-\gamma \bar{I}(t),\\
\frac{d \bar{R}(t)}{dt}&=\gamma \bar{I}(t),
\end{aligned}
\right.
\end{equation}
with the initial condition specified by Assumption~\ref{AS-SIR-0}.
\end{theorem}
Note that $(\bar{S}(t),\bar{I}(t),\bar{R}(t))$ are, respectively, the proportions of susceptible, infectious and recovered individuals in the limit as $N\to\infty$.
\begin{proof}
Let us consider the proportions in the three compartments, \textit{\textit{i.e.,}} we divide equation~\eqref{eq:MSIR} by $N$, and define $M_{inf}(t):=P_{inf}(t)-t$, 
$M_{rec}(t):=P_{rec}(t)-t$. We obtain, with the notation $\bar{\Upsilon}^N(s)=N^{-1}\Upsilon^N(s)$,
\begin{equation}\label{eq:MSIR/N}
\left\{
\begin{aligned}
\bar{S}^N(t)&=\bar{S}^N(0)-\int_0^t\bar{\Upsilon}^N(s) ds-N^{-1}M_{inf}\left({N}\int_0^t \bar{\Upsilon}^N(s)ds\right),\\
I^N(t)&=I^N(0)+ \int_0^t\bar{\Upsilon}^N(s) ds-\gamma\int_0^t\bar{I}^N(s)ds\\
&\quad+N^{-1}M_{inf}\left({N}\int_0^t \bar{\Upsilon}^N(s)ds\right)-N^{-1}M_{rec}\left(\gamma N\int_0^t \bar{I}^N(s)ds\right),\\
R^N(t)&=R^N(0)+\gamma\int_0^t\bar{I}^N(s)ds+N^{-1}M_{rec}\left(\gamma N\int_0^t \bar{I}^N(s)ds\right)\,.
\end{aligned}
\right.
\end{equation}
It is not hard to show that $N^{-1}M_{inf}(Nt)\to0$ and $N^{-1}M_{rec}(Nt)\to0$ a.s., uniformly for $t\in[0,T]$.
Indeed, the pointwise convergence follows directly from the classical law of large numbers, and then the uniform convergence from the fact that the function $t\mapsto N^{-1}P(Nt)$ is increasing, and converges to the continuous function $t\mapsto t$, thanks to the second Dini Theorem. This, combined with Assumption~\ref{AS-SIR-0}, allows one to take the limit in \eqref{eq:MSIR/N}, and deduce \eqref{eq:SIRODE}. More details can be found in Section 2.2 of
\cite{brittonpardoux}.
\end{proof}

\subsection{The Markov $\bS\bI\bR$  model : Central Limit Theorem}

Let us now rescale the differences between the proportions in the $N$ model and the limiting proportions. We define 
\[ (\hat{S}^N(t),\hat{I}^N(t),\hat{R}^N(t)):=\sqrt{N}(\bar{S}^N(t)-\bar{S}(t),\bar{I}^N(t)-\bar{I}(t),\bar{R}^N(t)-\bar{R}(t))\,.\]

In order to obtain a limit of the above processes, we need to formulate an assumption concerning 
$(\hat{S}^N(0),\hat{I}^N(0),\hat{R}^N(0))$.

\begin{assumption}\label{AS-SIR-02}
There exists a random vector $(\hat{S}(0),\hat{I}(0),\hat{R}(0))$ such that
\[ (\hat{S}^N(0),\hat{I}^N(0),\hat{R}^N(0))\Rightarrow(\hat{S}(0),\hat{I}(0),\hat{R}(0)) \qinq \RR^3 \qasq N \to \infty\,.\]
\end{assumption}

We have the following FCLT.
\begin{theorem}\label{th:MSIR-CLT}
Under Assumption~\ref{AS-SIR-02}, 
\[ (\hat{S}^N,\hat{I}^N,\hat{R}^N)\Rightarrow(\hat{S},\hat{I},\hat{R})\qinq \bD^3 \qasq N \to \infty,\]
where $(\hat{S}(t),\hat{I}(t),\hat{R}(t))$ is the unique solution of the following linear SDE:
\begin{equation}\label{eq:SIRCLT}
\left\{
\begin{aligned}
\hat{S}(t)&=\hat{S}(0)-\lambda\int_0^t(\bar{S}(s)\hat{I}(s)+\hat{S}(s)\bar{I}(s))ds
-\int_0^t\sqrt{\lambda\bar{S}(s)\bar{I}(s)}dB_{inf}(s),\\
\hat{I}(t)&=\hat{I}(0)+\lambda\int_0^t(\bar{S}(s)\hat{I}(s)+\hat{S}(s)\bar{I}(s))ds-\gamma\int_0^t\hat{I}(s)ds\\&\quad
+\int_0^t\sqrt{\lambda\bar{S}(s)\bar{I}(s)}dB_{inf}(s)-\int_0^t\sqrt{\gamma\bar{I}(s)}dB_{rec}(s),\\
\hat{R}(t)&=\hat{R}(0)+\gamma\int_0^t\hat{I}(s)ds+\int_0^t\sqrt{\gamma\bar{I}(s)}dB_{rec}(s),
\end{aligned}
\right.
\end{equation}
where $B_{inf}(t)$ and $B_{rec}(t)$ are two mutually independent standard Brownian motions, which are globally independent of $(\hat{S}(0),\hat{I}(0),\hat{R}(0))$. If $(\hat{S}(0),\hat{I}(0),\hat{R}(0))$ is Gaussian, then 
$(\hat{S}(t),\hat{I}(t),\hat{R}(t))$ is a Gaussian process.
\end{theorem}
Note that the notion of a Brownian motion is defined in Section \ref{sec:BM-WN} in the Appendix below.
\begin{proof}
We take the difference between \eqref{eq:MSIR/N} and \eqref{eq:SIRODE}, and multiply by $\sqrt{N}$. Hence
\begin{equation*}
\left\{
\begin{aligned}
\hat{S}^N(t)&= \hat{S}^N(0)-\lambda\int_0^t(\hat{S}^N(s)\bar{I}^N(s)+\bar{S}(s)\hat{I}^N(s))ds
-N^{-1/2}M_{inf}\left(N\int_0^t \bar{\Upsilon}^N(s)ds\right),\\
\hat{I}^N(t)&=\hat{I}^N(0)+\lambda\int_0^t(\hat{S}^N(s)\bar{I}^N(s)+\bar{S}(s)\hat{I}^N(s))ds
-\gamma\int_0^t\hat{I}^N(s)ds\\&\quad+N^{-1/2}M_{inf}\left(N\int_0^t \bar{\Upsilon}^N(s) ds\right)-N^{-1/2}M_{rec}\left(\gamma N\int_0^t \bar{I}^N(s)ds\right),\\
\hat{R}^N(t)&=\hat{R}^N(0)+\gamma\int_0^t\hat{I}^N(s)ds+N^{-1/2}M_{rec}\left(\gamma N\int_0^t \bar{I}^N(s)ds\right)\,.
\end{aligned}
\right.
\end{equation*}
The result now follows from Theorem~\ref{th:MSIR-LLN}, the next Lemma and rather standard arguments, see for the details Section 2.3 in \cite{brittonpardoux}.
\end{proof}

\begin{lemma}\label{le:cltPoisson}
Let $P(t)$ be a standard Poisson process, $M(t):=P(t)-t$. Then
\[ N^{-1/2}M(N\cdot)\Rightarrow B \qinq \bD \qasq N \to \infty,\]
where $B(t)$ is a standard Brownian motion.
\end{lemma}
\begin{proof}
Note that $\mathcal{M}(t):= N^{-1/2}M(Nt)$ is a square integrable martingale, whose associated increasing process
is given as $\langle\mathcal{M}\rangle_t=t$. Hence tightness in $\bD$ follows readily by the criterion from Proposition~\ref{pro:tightD} in the Appendix below. It thus suffices to show that for any $n\ge1$, any $0=t_0<t_1<\cdots<t_n$,
\[(\mathcal{M}(t_1),\mathcal{M}(t_2),\ldots,\mathcal{M}(t_n)\Rightarrow(B(t_1),B(t_2),\ldots,B(t_n)).\]
By independence of the increments of both $\mathcal{M}$ and $B$, it suffices to show that for any $t>0$,
$\mathcal{M}(t)\Rightarrow B(t)$. This is easily verified by a characteristic function computation. Indeed, for any $u\in\R$,
\begin{align*}
\E\left(\exp\left[iuN^{-1/2}M(Nt)\right]\right)=\exp\left(Nt\left[e^{i\frac{u}{\sqrt{N}}}-1-i\frac{u}{\sqrt{N}}\right]\right)
\to\exp\left(-t\frac{u^2}{2}\right)\,.
\end{align*}
\end{proof}

\begin{remark}
The last lemma is one of the simplest examples of convergence of discontinuous martingales towards Brownian motion. Such results have a long history, see e.g. Theorem 2 in \cite{rebolledo}.
\end{remark}

\subsection {Markovian $\bS\bI\bS$ and $\bS\bI\bR\bS$ models, and $\bS\bI\bR$ model with demography}\label{M-endemic}
In the SIR model, the number of susceptibles who can be infected is limited, and therefore, the epidemic goes soon or later to an end. However, there are several models where there is a constant flux of susceptibles, which allow the establishment of an endemic disease. Let us describe three such models.
\subsubsection{The $\bS\bI\bS$ model}\label{sec:sism}
In this model, contrary to the SIR model, when an infectious individual recovers, he/she becomes susceptible again. There is no immunity. The stochastic model reads:
\begin{equation*}
\left\{
\begin{aligned}
S^N(t)&=S^N(0)-P_{inf}\left(\int_0^t \Upsilon^N(s) ds\right)+P_{rec}\left(\gamma\int_0^t I^N(s)ds\right),\\
I^N(t)&=I^N(0)+P_{inf}\left(\int_0^t \Upsilon^N(s) ds\right)-P_{rec}\left(\gamma\int_0^t I^N(s)ds\right),
\end{aligned}
\right.
\end{equation*}
and the LLN limiting deterministic model reads:
\begin{equation*}
\left\{
\begin{aligned}
\frac{d \bar{S}(t)}{dt}&=-\lambda \bar{S}(t)\bar{I}(t)+\gamma \bar{I}(t),\\
\frac{d \bar{I}(t)}{dt}&=\lambda \bar{S}(t)\bar{I}(t)-\gamma \bar{I}(t)\,.
\end{aligned}
\right.
\end{equation*}
Note that exploiting the identities $S^N(t)+I^N(t)=N$, $\bar{S}(t)+\bar{I}(t)=1$, we can write in fact equations for 
$I^N(t)$ and $\bar{I}(t)$ only, which read:
\begin{align*}
I^N(t)&=I^N(0)+P_{inf}\left(\frac{\lambda}{N}\int_0^t (N-I^N(s))I^N(s)ds\right)-P_{rec}\left(\gamma\int_0^t I^N(s)ds\right),
\end{align*}
and
\begin{align*}
\frac{d \bar{I}(t)}{dt}=\lambda (1-\bar{I}(t))\bar{I}(t)-\gamma \bar{I}(t)\,.
\end{align*}
Again in this model $R_0=\lambda/\gamma$. If $R_0\le1$, the last equation has the unique equilibrium $I^\ast=0$, while if $R_0>1$, this disease--free equilibrium is unstable, and there is a stable endemic equilibrium
$I^\ast=1-\frac{\gamma}{\lambda}=1-R_0^{-1}$.
\subsubsection{The $\bS\bI\bR\bS$ model}
In this model, an individual is first removed (\textit{i.e.,} immune) when he/she recovers, but he/she loses its immunity at a given rate $\rho$. This gives the following stochastic model (which we write for the two quantities $S^N(t)$ and $I^N(t)$:
\begin{equation*}
\left\{
\begin{aligned}
S^N(t)&=S^N(0)-P_{inf}\left(\int_0^t \Upsilon^N(s)ds\right)+P_{loim}\left(\rho\int_0^t(N-S^N(s)-I^N(s)) ds\right),\\
I^N(t)&=I^N(0)+P_{inf}\left(\int_0^t \Upsilon^N(s) ds\right)-P_{rec}\left(\gamma\int_0^t I^N(s)ds\right),
\end{aligned}
\right.
\end{equation*}
where ``loim'' is an abbreviation for ``loss of immunity'',
and the following deterministic model:
\begin{equation*}
\left\{
\begin{aligned}
\frac{d \bar{S}(t)}{dt}&=-\lambda \bar{S}(t)\bar{I}(t)+\rho(1-\bar{S}(t)- \bar{I}(t)),\\
\frac{d \bar{I}(t)}{dt}&=\lambda \bar{S}(t)\bar{I}(t)-\gamma \bar{I}(t)\,.
\end{aligned}
\right.
\end{equation*}
Again $R_0=\lambda/\gamma$, if $R_0\le 1$, the only equilibrium is $(1,0)$, while if $R_0>1$, we have an endemic equilibrium $(\gamma/\lambda,(1-R_0^{-1})(\gamma+\rho)^{-1}\rho)$.
\subsubsection{The $\bS\bI\bR$ model with demography}
In this model, the recovered individuals do not lose their immunity, but births produce a constant flux of susceptible individuals. The stochastic model reads:
\begin{equation*}
\left\{
\begin{aligned}
S^N(t)&=S^N(0)-P_{inf}\left(\int_0^t \Upsilon^N(s) ds\right)+P_{birth}\left(\rho Nt \right)
-P_{death-sus}\left(\mu N\int_0^tS^N(s)ds\right),\\
I^N(t)&=I^N(0)+P_{inf}\left(\int_0^t \Upsilon^N(s) ds\right)-P_{rec}\left(\gamma\int_0^t I^N(s)ds\right)
-P_{death-inf}\left(\mu N\int_0^tI^N(s)ds\right),
\end{aligned}
\right.
\end{equation*}
and the deterministic model reads
\begin{equation*}
\left\{
\begin{aligned}
\frac{d \bar{S}(t)}{dt}&=\mu-\lambda \bar{S}(t)\bar{I}(t)-\mu\bar{S}(t),\\
\frac{d \bar{I}(t)}{dt}&=\lambda \bar{S}(t)\bar{I}(t)-(\gamma+\mu) \bar{I}(t)\,.
\end{aligned}
\right.
\end{equation*}
\begin{remark} We model the birth as a constant flux at rate $\mu N$, instead of $\mu$ times the number of individuals in the population, in order to avoid the pitfall of critical branching processes, which go extinct in finite time a.s. As a result, if the individuals in the $\bR$ compartment die at rate $\mu$ as well, the total population in the stochastic model remains close to $N$. Therefore we approximate the proportion of susceptibles in the population by $S^N(t)/N$.
\end{remark}

This time, $R_0=\frac{\lambda}{\gamma+\mu}$. If $R_0>1$, the endemic equilibrium reads
$((\gamma+\mu)/\lambda, (1-R_0^{-1})(\gamma+\mu)^{-1}\mu)$.

\subsection{Deviations from the law of large numbers and extinction of an endemic disease}\label{M-LD}
In the above three models, the endemic equilibrium is stable whenever $R_0>1$. This means that the LLN deterministic model, starting from a positive value of $\bar{I}(0)$, will never go extinct. However, if we consider the stochastic model, it is easily seen that for any $N\ge1$, disease free states are accessible with positive probability. Since moreover they are absorbing, in the stochastic model the epidemic stops soon or later. We would like to know how  long we have to wait for this to happen.  One approach is to evaluate the time needed for the stochastic system to diverge enough from its deterministic limit, so that $I^N(t)=0$. For that purpose, we shall exploit three tools that Probability theory gives us, in order to estimate the difference between a stochastic process and its law of large numbers limit, namely the Central Limit Theorem, Moderate Deviations and Large Deviations.

Let us first see what the CLT tells us. Consider first the $\bS\bI\bR$ model with demography. In that model, 
at the endemic equilibrium, $\bar{I}^\ast=(1-R_0^{-1})\frac{\mu}{\gamma+\mu}$, where $R_0=\frac{\lambda}{\gamma+\mu}$. It is clear that $\gamma$ is much larger than $\mu$ (in inverse of years, compare 52 to 1/75, since $\gamma^{-1}$ is of the order of 1 week, and $\mu^{-1}$ is of the order of 75 years). Consequently we can 
consider that $R_0\sim\frac{\lambda}{\gamma}$ and $\varepsilon:=\frac{\mu}{\gamma+\mu}\sim\frac{\mu}{\gamma}$. Now the CLT tells us that $I^N(t)$ is approximately Gaussian, with mean $N\varepsilon(1-R_0^{-1})$ and standard deviation$\sqrt{N/R_0}$ (the asymptotic variance of  $\hat{I}(t)$ is close to $R_0^{-1}$, see \cite{brittonpardoux} page 62). If $N$ is such that the standard deviation is at least the mean divided by 3, then it is likely that $I^N(t)$ will hit zero in time of order 1. This leads to the idea of a critical population size $N_c$ given by 
 $N_c\varepsilon(1-R_0^{-1})=3\sqrt{N_c/R_0}$, that is 
 \[ N_c=\frac{9}{\varepsilon^2(1-R_0^{-1})^2R_0}.\]
 In the case of measles, $R_0=15$. With the above approximations for $\mu$ and $\gamma$, we arrive at $N_c$ in the order of a few million.  If $N\le N_c$, then the CLT predicts that extinction should occur in time of order $1$. If $N>N_c$, extinction is likely to happen in time which is large with $N$, and this is predicted by moderate or large deviations, as we shall explain next.  This confirms the empirical observation that, prior to vaccination, measles was 
 continuously endemic in countries like UK, and died out quickly in Iceland (and was then reintroduced by infected visitors).
 
 \begin{remark} Would we consider the $\bS\bI\bS$ model instead of the $\bS\bI\bR$ model with demography,  then we would find a much smaller $N_c$. Indeed, the asymptotic variance in the CLT is about the same, but $\bar{I}^\ast$ is much larger. Indeed, if everyone who recovers becomes susceptible again, we are likely to have a much larger proportion of infectious at equilibrium. In the $\bS\bI\bR$ model with demography, contrary to the situation in the $\bS\bI\bS$ model, those who get infected are infected only once in their life. The ratio between the two 
 $\bar{I}^\ast$'s is the above $\varepsilon$.
 \end{remark}
 
 As explained above, sooner or later the stochastic process $I^N(t)$ will hit zero (and then stay there for ever). The CLT allows us to guess for which population sizes extinction is likely to happen in time of order 1. 
 We will now discuss what large deviations tell us on this problem. In the three above examples, we have an $\R^d$--valued ODE of the form
 \begin{align}\label{limitODE}
 \frac{dz(t)}{dt}=b(z(t)),\quad z(0)=x\,,
 \end{align}
 which is the LLN limit of a sequence of SDEs of the form
 \begin{align}\label{prelimitSDE}
 Z^N(t)=x_N+\sum_{j=1}^k\frac{h_j}{N}P_j\left(N\int_0^t\beta_j(Z^N(s))ds\right),
 \end{align}
 and we have  $b(z)=\sum_{j=1}^k\beta_j(z)h_j$. Recall that $P_1,\ldots,P_k$ are mutually independent standard Poisson processes. Note that if we rewrite the SDE \eqref{prelimitSDE} in the form (with $M_j(t)=P_j(t)-t$)
 \[ Z^N(t)=x_N+\int_0^tb(Z^N(s))ds+\sum_{j=1}^k\frac{h_j}{N}M_j\left(N\int_0^t\beta_j(Z^N(s))ds\right),\]
 we may regard
 the SDE \eqref{prelimitSDE} for large $N$ as a ``small random perturbation of the dynamic system \eqref{limitODE}''. Freidlin and Wentzell have studied such perturbations as an application of large deviations theory, see \cite{freidWentz}.
 
 Let us first state what kind of information large deviations give us, concerning the convergence of $Z^N$ toward $z$. We shall state the results without giving the precise technical conditions under which one can establish them, referring the reader to Section 4.2 of \cite{brittonpardoux} for all the technical details and proofs. Consider the function $\ell:\R^{2d}\mapsto\R$ defined by
 \[ \ell(x,y,\theta)=\langle y,\theta\rangle-\sum_{j=1}^k\beta_j(x)\left(e^{\langle h_j,\theta\rangle}-1\right)\,.\]
 We let
 \[I_T(\phi):=\int_0^T L(\phi(t),\dot{\phi(t)})dt,\quad \text{where } L(x,y)=\sup_{\theta\in\R^d}\ell(x,y,\theta)\,.\]
 One essential property of this functional  is that $I_T\ge0$ and $I_T=0$ iff $\phi$ solves the ODE \eqref{limitODE}.
 One may think of $I_T(\phi)$ as a sort of measure of how much $\phi$ differs from being a solution to \eqref{limitODE}. 
 
 Large deviations theory gives us both (in the statements below, $x_N$ stands for the vector whose $i$--th coordinates 
 is $[Nx^i ]/N$)
 \begin{itemize}
 \item a lower bound:
 for any open subset $O\subset \bD([0,T];\R^d)$, 
 \[ \liminf_N\frac{1}{N}\log\P\left(Z^{N,x_N}\in O\right)\ge-I_{T,x}(O),\]
where $Z^{N,x_N}$ denotes the solution of \eqref{prelimitSDE} starting from $Z^{N,x_N}(0)=x_N$, and
$I_{T,x}(O):=\inf_{\phi\in O,\phi(0)=x}I_T(\phi)$; and 
\item an upper bound:
for any closed subset $F\subset \bD([0,T];\R^d)$,
\[\limsup_N\frac{1}{N}\log\P\left(Z^{N,x_N}\in F\right)\le I_{T,x}(F)\,.\]
\end{itemize}
Note that, as is to be expected, those results give us information only in case $O$ (resp. $F$) does not contain the solution of \eqref{limitODE} starting from $x$.

Following the ideas of Freidlin and Wentzell, one can deduce from those two statements a rather precise statement about the time taken by the random perturbations to drive the process $Z^N$ to a disease free situation. Denote by $A$ the subset of points of $\R^d_+$ which are accessible by our system. Suppose that $I^N$ is the first component of $Z^N$. We are interested in the time needed for $Z^N(t)$ to reach the subset of $\R^d_+$ where its first component is $0$. Let us denote by $z^\ast\in A$ the endemic equilibrium, i.e., the point in $A$ such that $b(z^\ast)=0$ and $z^\ast_1>0$ 
(where 
$z^\ast_1$ denotes the first coordinate of $z^\ast$), which we assume to be unique. We now define the ``quasi--potential'' ($\bD_{T,A}$ stands for $\bD([0,T];A)$): 
\begin{align*}
V(z,z',T)&=\inf_{\phi\in D_{T,A},\ \phi(0)=z,\ \phi(T)=z'}I_T(\phi),\\
V(z,z')&=\inf_{T>0}V(z,z',T),\\
\overline{V}&=\inf_{z\in A,\ z_1=0}V(z^\ast,z)\,.
\end{align*}

For $z\in A$, we define the extinction time of the process $Z^N$ as
\[ T^{N,z}_{\text{ext}}:=\inf\{t>0,\ Z^{N,[Nz]/N}_1(t)=0\}\,.\]
We have the result (see Theorem 4.2.17 in \cite{brittonpardoux})
\begin{theorem}
For any $\eta>0$ and $z\in A$,
\begin{align*}
\lim_N\P\left(\exp[N(\overline{V}-\eta)]<T^{N,z}_{\text{ext}}<\exp[N(\overline{V}+\eta)]\right)&=1,\\
\text{and }\lim_N\frac{1}{N}\log\E[T^{N,z}_{\text{ext}}]&=\overline{V}\,.
\end{align*}
\end{theorem}

Note that $\overline{V}$ is the value function of an optimal control problem. It can be computed explicitly in case of the $\bS\bI\bS$ model, in which case $\overline{V}=\log R_0-1+R_0^{-1}$.  Unless $\overline{V}$ is quite small, we expect $\exp(\overline{V}N)$ to be very large.

  In between the CLT and Large Deviations, we have the theory of ``moderate deviations''.
  Let us explain what we can learn from this theory, concerning our problem. We will only give a brief sketch of the ideas, referring the reader to \cite{pardouxEJP20} for the details. 
  
  Large Deviations discusses the probability of observing deviations from the LLN of the order of 1, as well as the time we have to wait for observing such deviations. The CLT predicts deviations from the LLN of the order of $N^{-1/2}$. Moderate Deviations discusses deviations of the order of $N^{-\alpha}$, for some $0<\alpha<1/2$: both the probability of observing such deviations, and the time we have to wait to see such deviations. This should allow us to predict extinction in less time than what Large Deviations predicts, with a critical population size larger than that associated to the CLT. 
  
  Since we want to discuss deviations from the endemic equilibrium $z^\ast$ of the order of $N^{-\alpha}$, let us consider the process $Z^N_z(t)$, starting from $Z^N_z(0)=z^\ast+N^{-\alpha}z$, where $z\in\R^d$ is arbitrary. 
  We want to study the Moderate Deviations of that process, which amounts to studying the Large Deviations  of
 $Z^{N,\alpha}_z(t):=N^\alpha(Z^N_z(t)-z^\ast)$.    For some $a>0$, the above Freidlin--Wentzell result tells us that if we define ($Z^N_{z,1}(t)$ stands for the first coordinate of $Z^N_z(t)$)
 \[ T^N_{z,a}=\inf\{t>0,\ Z^{N,\alpha}_{z,1}(t)\le -a\}\,,\]
 we obtain, for a certain $\overline{V}_a$, any $\eta>0$,
\begin{align*}
\lim_N\P\left(\exp[N^{1-2\alpha}(\overline{V}_a-\eta)]<T^{N}_{z,a}<\exp[N^{1-2\alpha}(\overline{V}_a+\eta)]\right)&=1,
\\ 
\lim_N N^{2\alpha-1}\log\E[T^{N}_{z,a}]&=\overline{V}_a\,.
\end{align*}
The case $\alpha=1/2$ is covered by the CLT, the case $\alpha=0$ by Large Deviations. Moderate Deviations, fills  the gap between those two regimes. If the population size $N$ is such that $z^\ast_1$, the first coordinate of $z^\ast$,
 is of the order of $N^{-\alpha}$, for some $0<\alpha<1/2$, Moderate Deviations will predict extinction in time of the order of 
$\exp[N^{1-2\alpha}\overline{V}_{a}]$, with $a=N^\alpha z^\ast_1$.  Of course, the value of $z^\ast_1$
is independent of $N$. The previous sentence should be understood as follows: if the population size $N$ is such that 
for some $0<\alpha<1/2$, the quantity $N^{-\alpha}$  of the order of $z^\ast_1$, then Moderate Deviations will predict extinction in time of the order of 
$\exp[N^{1-2\alpha}\overline{V}_{a}]$, with $a=N^\alpha z^\ast_1$. Note that in case $\alpha=0$, this means no restriction on $N$.

\section{Non--Markov and  integral equation models}\label{sec:nonmarkov}

\subsection{The $\bS\bI\bR$ model}\label{sec:SIRnonM}
In this section, we will still assume that the infectivity is constant, but we shall let the infectious period have a general probability distribution. The way Kermack and McKendrick assumed a general distribution for the infectious period in their 1927 paper \cite{KMK} was to choose an infection age dependent recovery rate. Let us first show that this formulation covers all absolutely continuous distributions for the duration of the infectious period.

To an $\R_+$--valued random variable $X$, we associate its cumulative distribution function (cdf) $F(t)=\P(X\le t)$ and its survival probability $F^c(t)=1-F(t)=\P(X>t)$.
 If $F$ has a density $f(t)$ (\textit{i.e.,} $f(t)=F'(t)$), then we define its hazard function as the quantity $\gamma(t):=f(t)/F^c(t)$. Let $P(t)$ be a standard Poisson process. Then the law of $X$ coincides with that of the first jump of the counting process $P\left(\int_0^t\gamma(s)ds\right)$, which follows from the following computations.
\begin{align*}
\P\left[P\left(\int_0^t\gamma(s)ds\right)=0\right]&=\exp\left(-\int_0^t\gamma(s)ds\right)\\
&=\exp\left(\int_0^t\frac{d}{ds}\left[\log F^c(s)\right]ds\right)\\
&=F^c(t)=\P(X>t)\,.
\end{align*}
In other words, by choosing an infection age dependent recovery rate, Kermack and McKendrick allowed a general absolutely continuous distribution for the duration of the infectious period. We shall use a different formulation, and allow a completely arbitrary distribution for the infectious period. 

Concerning the infection process, we have the same infection rate as in Section~\ref{subsec:SIRLLN}, namely
\begin{equation}\label{Upsilon}
\Upsilon^N(t)=\lambda I^N(t)\frac{S^N(t)}{N}\,.
\end{equation}
Let $A^N(t)$ denote the cumulative counting process of newly infected individuals on the time interval $(0,t]$. We have, as above,
\begin{equation}\label{eq:AN}
A^N(t)=\int_0^t\int_0^\infty {\bf1}_{u\le\Upsilon^N(s)} Q(ds,du),\ t\ge0\,.
\end{equation}
Clearly the following balance equations hold
\begin{align} \label{eqn-SIR-balance}
S^N(t)+I^N(t)+R^N(t)&=N, \non\\
S^N(t)&=S^N(0)-A^N(t),\\
I^N(t)&=I^N(0)+A^N(t)-R^N(t)\,. \non
\end{align}

To each newly infected individual $i\in \NN$, we associate a random variable $\eta_i$ to represent its infectious duration. We assume that the $\eta_i$'s are i.i.d. with a cumulative distribution function (c.d.f.) $F$, and let $F^c:=1-F$.  For each initially infectious individual $j=1,\dots, I^N(0)$, let $\eta_j^0$ be its remaining infectious period. We also assume that the $\eta_j^0$'s are i.i.d. with a c.d.f. $F_0$, and let $F^c_0:=1-F_0$. 

\begin{remark}
An initially infected individual is thought of as having been infected at some time $\tau_0<0$.
Assuming that the law of the duration of the infectious period of this individual is $F$,
the probability that he/she is still infectious at some time $t>0$, given the time of infection,  equals $F^c(t-\tau_0)/F^c(-\tau_0)$, the conditional law of still being infectious at time $t$, given that he/she was still infectious at time $0$. In the case where $F$ is exponential, this is exactly $F^c(t)$, hence there is no reason to choose an $F_0$ different from $F$. But if $F$ is not the exponential distribution,  then $F_0\not=F$.
\end{remark}

\begin{remark}
In this model it is clear that $R_0=\lambda\int_0^\infty F^c(t)dt$, since 
$\int_0^\infty F^c(t)dt=\E[\eta]$.
\end{remark}

Denoting by $\tau^N_i$, $i\ge1$ the successive jumps times of $A^N$, the dynamics of $I^N(t)$ can be described by
\begin{align}\label{eqn-SIR-In}
I^N(t) = \sum_{j=1}^{I^N(0)} \bone_{\eta^0_j > t}+ \sum_{i=1}^{A^N(t)} \bone_{\tau^N_i + \eta_i >t}, \quad t\ge 0.
\end{align}
and the dynamics of $R^N(t)$:
\begin{align}\label{eqn-SIR-Rn}
  R^N(t) &\,=\,R^N(0)+ \sum_{j=1}^{I^N(0)}{\bf1}_{\eta^0_j\le t}+\sum_{i=1}^{A^N(t)}{\bf1}_{\tau^N_i+\eta_i\le t}\,. 
\end{align}

Define the fluid-scaled process $\bar{X}^N:= N^{-1} X^N$ for any process $X^N$. We make the following assumption on the initial quantities. 
\begin{assumption} \label{AS-SIR-1}
There exists a deterministic constant $(\bar{S}(0),\bar{I}(0),\bar{R}(0))\in (0,1)^3$ such that 
$\bar{S}(0)+\bar{I}(0)+\bar{R}(0)=1$, and  as $N\to\infty$,
 $(\bar{S}^N(0),\bar{I}^N(0),\bar{R}^N(0)) \to (\bar{S}(0),\bar{I}(0),\bar{R}(0))$ in probability. 
\end{assumption}

\begin{theorem}\label{thm-FLLN-SIR} 
{\bf (Functional Law of Large Numbers)}
Under Assumption~\ref{AS-SIR-1}, 
$$
(\bar{S}^N, \bar{I}^N, \bar{R}^N) \to ( \bar{S},  \bar{I}, \bar{R})\qinq \bD^3$$
in probability 
as $N\to\infty$, 
where the limit process $(\bar{S},\bar{I}, \bar{R})$ is the unique solution to the system of deterministic Volterra integral equations
\begin{align}
\bar{S}(t) &= \bar{S}(0) - \int_0^t\bar\Upsilon(s) ds, \label{SIR-barS}\\
\bar{I}(t)  & =  \bar{I}(0) F_0^c(t) +  \int_0^t F^c(t-s) \bar\Upsilon(s) ds,  \label{SIR-barI}\\
\bar{R}(t) &= \bar{R}(0)+\bar{I}(0) F_0(t) + \int_0^t F(t-s) \bar\Upsilon(s) ds, \label{SIR-barR}
\end{align}
with
\begin{equation} \label{eqn-barUpsilon-constlambda}
\bar\Upsilon(t) = \lambda \bar{S}(t) \bar{I}(t)
\end{equation}
for $t\ge 0$.  $\bar{S}$ is in $\bC$.  If $F_0$ is continuous, then $\bar{I}$ and $\bar{R}$ are in $\bC$; otherwise they are in $\bD$. 
\end{theorem}

The limits $(\bar{S},\bar{I})$ are uniquely determined by the two equations \eqref{SIR-barS} and \eqref{SIR-barI}. Existence and uniqueness for such a system is well--known. Given the solution $(\bar{S},\bar{I})$, the limit $\bar{R}$ is given by \eqref{SIR-barR}. 
Assuming that $F_0$ and $F$ have densities $f_0$ and $f$, by taking derivatives in  \eqref{SIR-barS}, \eqref{SIR-barI} and \eqref{SIR-barR}, we obtain
\begin{equation}\label{eqn-SIR-deriv}
\left\{
\begin{aligned} 
\bar{S}'(t) &= - \bar\Upsilon(t),  \\
\bar{I}'(t)  & = -  \bar{I}(0) f_0(t) +  \Big(\bar\Upsilon(t) - \int_0^t f(t-s) \bar\Upsilon(s)  ds \Big),\\
\bar{R}'(t)  & =  \bar{I}(0) f_0(t) +  \int_0^t f(t-s)\bar\Upsilon(s) ds\,.
\end{aligned}
\right.
\end{equation}
Note that, as expected, $\bar{S}'(t)+\bar{I}'(t)+\bar{R}'(t)=0$.

Let us now sketch the proof of Theorem~\ref{thm-FLLN-SIR}.

\begin{proof}
Define $\bar{Q}(ds,du)=Q(ds,du)-dsdu$ the compensated measure. We have
\begin{align}\label{decompAN} 
\bar{A}^N(t)=\int_0^t\bar{\Upsilon}^N(s)ds+\bar{M}_A^N(t),
\end{align}
where
\[ \bar{M}_A^N(t)=\frac{1}{N}\int_0^t\int_0^\infty{\bf1}_{u\le\Upsilon^N(s)} \bar{Q}(ds,du)\,.\]
Since $0\le\bar{\Upsilon}^N(s)\le\lambda$, the first term on the right of \eqref{decompAN} is an increasing process
which is Lipchitz continuous, with a Lipschitz constant bounded by $\lambda$. Hence that sequence is equi--continuous, it is tight in $\bC$, and also in $\bD$. The second term is a martingale, which satisfies
\[ \E\left[\left(\bar{M}_A^N(t)\right)^2\right]\le\frac{\lambda}{N}t,\]
hence from Doob's maximal inequality, it converges to $0$ in mean square, locally uniformly in $t$. As a consequence, along a subsequence, $\bar{A}^N\Rightarrow\bar{A}$ in $\bD$, where for $0\le s<t$, 
$0\le \bar{A}(t)-\bar{A}(s)\le \lambda(t-s)$, and along the same subsequence, 
$S^N(\cdot)\Rightarrow \bar{S}(0)-\bar{A}(\cdot)$ in $\bD$.

Let now $I^N(t)=I^N_0(t)+I^N_1(t)$, where
\[ I^N_0(t)=\sum_{j=1}^{I^N(0)}{\bf1}_{\eta^0_j>t},\quad I^N_1(t)=\sum_{i=1}^{A^N(t)}{\bf1}_{\tau^N_i+\eta_i>t}\,.\]
Consider first $I^N_0(t)$. Define
\[ \breve{I}^N_0(t):=\sum_{j=1}^{N\bar{I}(0)} {\bf1}_{\eta^0_j>t}\,.\]
It is not too hard to deduce from the law of large numbers, see Theorem 14.3 in \cite{billingsley} that as $N\to\infty$, $\breve{I}^N_0\to\bar{I}(0)F^c$ in $\bD$ in probability. Next, as $N\to\infty$,
\begin{align*}
|\bar{I}^N_0(t)-\breve{I}^N_0(t)|\le |\bar{I}^N(0)-\bar{I}(0)|\to0
\end{align*}
in probability, thanks to Assumption~\ref{AS-SIR-1}. 

We finally consider $I^N_1(t)$. Let $\sF^N_t:=\sigma\{\Upsilon^N(s), s\le t; \tau^N_i, i\le A^N(t)\}$ and define
\begin{align*}
 \breve{I}^N_1(t):&= \E[\bar{I}^N_1(t)|\sF^N_t]\\
 &= \frac{1}{N}\sum_{i=1}^{A^N(t)}F^c(t-\tau^N_i)\\
 &=\int_0^t F^c(t-s)d\bar{A}^N(s)\,.
 \end{align*}
 It is not hard to deduce from a slight extension of the Portmanteau theorem (see e.g. Lemma 4.4 in \cite{FPP2020b}) that along a subsequence along which $\bar{A}^N\Rightarrow \bar{A}$, for each $t>0$, 
 \[ \breve{I}^N_1(t)\Rightarrow\int_0^tF(t-s)d\bar{A}(s)\,.\]
  From a tightness argument, we deduce that this convergence holds in fact in $\bD$. Finally we consider $V^N(t):=\bar{I}^N_1- \breve{I}^N_1$. 
  We have
  \begin{align*}
  V^N(t)&=\frac{1}{N}\sum_{i=1}^{A^N(t)}\kappa^N_i(t),\ \text{ where}\\
  \kappa^N_i(t)&={\bf1}_{\tau^N_i+\eta_i>t}-F^c(t-\tau^N_i)\,.
  \end{align*}
  It is easy to check that
  \[\E[\kappa^N_i(t)|\sF^N_t]=0,\quad \text{and for }i\not=j, \  \E[\kappa^N_i(t)\kappa^N_j(t)|\sF^N_t]=0\,.\]
  Thus
  \begin{align*}
  \E[V^N(t)^2|\sF^N_t]&=\frac{1}{N^2}\sum_{i=1}^{A^N(t)}\E[\kappa^N_i(t)^2|\sF^N_t]\\
  &\le\frac{\bar{A}^N(t)}{N},\\
  \E[V^N(t)^2]&\le\frac{\lambda}{N}t\,.
  \end{align*}
  With some additional effort, one can show that in fact  $V^N(t)\to0$ in probability, locally uniformly in $t$, see the proof of Lemma 5.2 in \cite{PP-2020}. Moreover one can show by similar arguments that $\bar{R}^N\Rightarrow\bar{R}$, and that the limiting equations have a unique deterministic solution, hence the whole sequence converges, and the convergence is in probability.
  \end{proof}
  
 We next turn to the central limit theorem. For that sake, we need to state an appropriate assumption concerning the initial quantities.
 \begin{assumption}\label{AS-SIR-2}
There exists a random vector $(\hat{S}(0),\hat{I}(0),\hat{R}(0))$ such that
\[ (\hat{S}^N(0),\hat{I}^N(0),\hat{R}^N(0))\Rightarrow(\hat{S}(0),\hat{I}(0),\hat{R}(0))\,.\]
In addition, we assume that $\sup_N\E\left(|\hat{S}^N(0)|^2+|\hat{I}^N(0)|^2+|\hat{R}^N(0)|^2\right)<\infty$.
\end{assumption}  

We can now state the following result.
 \begin{theorem} \label{thm-FCLT-SIR}
{\bf (Functional Central Limit Theorem)} Under Assumption~\ref{AS-SIR-2}, 
 \begin{equation*} 
(\hat{S}^N, \hat{I}^N, \hat{R}^N) \RA (\hat{S}, \hat{I}, \hat{R}) \qinq \bD^3 \qasq N \to\infty,
\end{equation*}
where the limit $(\hat{S}, \hat{I}, \hat{R})$ is the unique solution to the following set of linear stochastic Volterra integral equations driven by Gaussian processes: 
\begin{align}\label{SIR-Shat}
\hat{S}(t) &= \hat{S}(0)  -    \int_0^t \hat{\Upsilon}(s) ds -  \hat{M}_A(t), 
\\
\label{SIR-Ihat}
\hat{I}(t) &=  \hat{I}(0) F^c_0(t)   +  \int_0^t F^c(t-s)  \hat{\Upsilon}(s) ds +  \hat{I}_{0}(t) +  \hat{I}_{1}(t),  
\\
\label{SIR-Rhat}
\hat{R}(t) &= \hat{R}(0)+\hat{I}(0) F_0(t)  + \lambda \int_0^t F(t-s) \hat{\Upsilon}(s)ds + \hat{R}_{0}(t) +  \hat{R}_{1}(t), 
\end{align}
and
\begin{equation} \label{eqn-hatUpsilon-constlambda}
\hat{\Upsilon}(t) =  \lambda (\hat{S}(t) \bar{I}(t) +\bar{S}(t)\hat{I}(t)),
\end{equation}
with $\bar{S}(t)$ and $\bar{I}(t)$ given in Theorem~\ref{thm-FLLN-SIR}.
Here $ (\hat{I}_{0},\hat{R}_0)$, independent of $(\hat{S}(0),\hat{I}(0),\hat{R}(0))$, is a mean-zero two-dimensional Gaussian process with the covariance functions: for $t, t'\ge 0$,
 \begin{align} \label{eqn-cov-I0R0-SIR}
\Cov(\hat{I}_{0}(t), \hat{I}_{0}(t')) &= \bar{I}(0) (F_0^c(t\vee t') - F_0^c(t) F_0^c(t')), \non \\
\Cov(\hat{R}_{0}(t), \hat{R}_{0}(t')) &= \bar{I}(0) (F_0(t\wedge t') - F_0(t) F_0(t')), \\
\Cov(\hat{I}_0(t), \hat{R}_0(t') ) & = \bar{I}(0)\big[ (F_0(t') - F_0(t)) \bone(t'\ge t) - F_0^c(t) F_0(t') \big].  \non
\end{align}
If $F_0$ is continuous, then $ \hat{I}_{0}$ and $\hat{R}_0$ are continuous. 
The limit process
$ (\hat{M}_A, \hat{I}_1,\hat{R}_1)$,   is a continuous three-dimensional Gaussian process, independent of 
$(\hat{S}(0),\hat{I}(0),\hat{R}(0),\hat{I}_0, \hat{R}_0)$,  and has the representation
\begin{align}
\hat{M}_A(t) = W_F([0,t]\times[0,\infty)), \quad
 \hat{I}_1(t) = W_F([0,t]\times[t,\infty)), \quad
\hat{R}_1(t) = W_F([0,t]\times[0,t]),  \non
\end{align}
where  $W_F$ is
a Gaussian white noise process on $\RR_+^2$ with mean zero and 
$$\E \left[ W_F((a,b]\times (c,d])^2\right] = \int_a^b (F(d-s)-F(c-s)) \bar\Upsilon(s) ds,
$$
for $0 \le a \le b$ and $0 \le c \le d$. 
 The limit process $\hat{S}$ has continuous sample paths and $\hat{I}$ and $\hat{R}$ have  c{\`a}dl{\`a}g sample paths.  If the c.d.f. $F_0$ is continuous, then $\hat{I}$ and $\hat{R}$ have continuous sample paths. 
 If $(\hat{S}(0),\hat{I}(0),\hat{R}(0))$ is a Gaussian random vector, then  $(\hat{S},\hat{I},\hat{R})$ is a Gaussian process. 
 \end{theorem}
 Note that the notion of white noise in defined in Section \ref{sec:BM-WN} in the Appendix below.
 
 From the representation of the limit processes $ (\hat{M}_A, \hat{I}_1,\hat{R}_1)$ using the white noise $W_F$, we easily obtain 
their covariance functions: for $t, t'\ge 0$, 
\begin{align*}
\Cov(\hat{M}_A(t), \hat{M}_A(t'))  &=  \int_0^{t\wedge t'}  \bar\Upsilon(s) ds, \\
\Cov(\hat{I}_1(t), \hat{I}_1(t')) &=  \int_0^{t\wedge t'} F^c(t\vee t'-s)  \bar\Upsilon(s) ds, \\
\Cov(\hat{R}_1(t), \hat{R}_1(t')) &=  \int_0^{t\wedge t'} F(t\wedge t'-s)  \bar\Upsilon(s) ds, \\
\Cov(\hat{M}_A(t), \hat{I}_1(t')) & =  \int_0^{t\wedge t'} F^c(t'-s)  \bar\Upsilon(s) ds,\\
\Cov(\hat{M}_A(t), \hat{R}_1(t')) &=  \int_0^{t\wedge t'} F(t'-s)  \bar\Upsilon(s) ds, \\
\Cov(\hat{I}_1(t), \hat{R}_1(t')) &=  \int_0^{t} (F(t'-s) - F(t-s)) \bone(t'>t)   \bar\Upsilon(s) ds.  
\end{align*}

In the FCLT, the limits $(\hat{S},\hat{I})$ are the unique solution of the system of stochastic Volterra integral equations \eqref{SIR-Shat} and \eqref{SIR-Ihat}.  Once $\hat{S}$ and $\hat{I}$ are specified,  $\hat{R}$ is given by the formula \eqref{SIR-Rhat}. 

For the proof of Theorem~\ref{thm-FCLT-SIR}, we refer the reader to Section 6 of \cite{PP-2020}.

\begin{remark}
When the infectious periods are deterministic, that is, $\eta_i$ is equal to a positive constant $\eta$ with probability one, the dynamics of $I^N(t)$ can be written as
\begin{align*} 
I^N(t) = \sum_{j=1}^{I^N(0)} \bone(\eta^0_j > t)+ A^N(t) - A^N((t-\eta)^+), \quad t\ge 0. 
\end{align*}
We assume that $\eta^0_j \sim U[0,\eta]$, that is, $F_0(t)= t/\eta$ for $t \in [0,\eta]$,  which is  the equilibrium (stationary excess) distribution of 
$F(t) = \bone_{t \ge \eta}$, $t\ge 0$. 
The fluid equation $\bar{I}(t)$ becomes 
$$
\bar{I}(t) = \bar{I}(0) (1-t/\eta)^+ + \int_{(t-\eta)^+}^t \bar\Upsilon(s) ds
$$
which gives 
$$
\bar{I}'(t) = - \frac{1}{\eta}\bar{I}(0) \bone_{t < \eta}+ \bar\Upsilon(t) - \bone_{t\ge\eta} \bar\Upsilon(t-\eta).  
$$
In the FCLT, we have
\begin{align}\label{hatI-SIR-det}
\hat{I}(t) &=  \hat{I}(0) (1-t/\eta)^+    +  \int_{(t-\eta)^+}^t \hat\Upsilon(s)   ds  +  \hat{I}_{0}(t) +  \hat{I}_{1}(t), \quad t \ge 0, 
\end{align}
where 
 $ \hat{I}_{0}(t)$, $t \in [0,\eta]$,  is a continuous mean-zero Gaussian process with the covariance function 
\begin{equation*}
\Cov(\hat{I}_{0}(t), \hat{I}_{0}(t')) = \bar{I}(0) (1-(t\wedge t')/\eta - (1-t/\eta) (1-t'/\eta)), \quad t, t' \in [0, \eta], 
\end{equation*}
and  $\hat{I}_1(t)$, $t\ge 0$,  is a continuous mean-zero Gaussian process with the covariance function 
\begin{equation*}
\Cov(\hat{I}_1(t), \hat{I}_1(t')) =  \int_0^{t\wedge t'} \bone_{t \vee t'-s< \eta} \bar\Upsilon(s) ds, \quad t, t' \ge 0. 
\end{equation*}
Note that the effect of the initial quantities vanish after time $\eta$, that is, in the stochastic integral equation \eqref{hatI-SIR-det} of $\hat{I}(t)$, the components $\hat{I}_0(t)$ and $ \hat{I}(0) (1-t/\eta)^+  $ vanish after $\eta$. 
\end{remark}

\subsection{An alternative initial condition}
In the above model, we have assumed that the remaining infectious periods have a different distribution $F_0$. That modeling approach is mostly due to a lack of information concerning the times infection for these initially infected individuals. 
An alternative modeling approach is to assume that the infection times of the initially infected individuals are known.
We assume the laws of the infectious durations of all individuals are the same, given by the c.d.f. $F$.
This is reasonable since the model is for the same disease. 
Of course, there is no difference in the two modeling approaches in the Markovian setting due to the lack of memory property of exponential distributions.

Suppose that the initially infected individuals are infected at times $\tau^N_{j,0}$, $j=1,\dots,I^N(0)$.
Then $\tilde\tau^N_{j,0} = -\tau^N_{j,0}$, $j=1,\dots, I^N(0)$, represent the amount of time that an initially infected individual has been infected by time 0, that is, the age of infection at time 0. 
WLOG, we can assume that $0 > \tau^N_{1,0}> \cdots>\tau^N_{I^N(0),0}$. 
We use the same notation $\eta^0_j$ to denote the remaining infectious duration for $j=1,\dots, I^N(0)$. Then the distribution of $\eta^0_j$ will naturally depend on the elapsed infectious time $\tilde\tau^N_{j,0}$. In particular, the conditional distribution of $\eta^0_j$ given that 
$\tilde{\tau}_{j,0}^N=s>0$ is given by
\begin{align} \label{eqn-eta0-age}
\P(\eta^0_j > t | \tilde{\tau}_{j,0}^N=s) = \frac{F^c(t+s)}{F^c(s)}, \qforq t, s >0. 
\end{align}
Note that the $\eta^0_j$'s are independent but not identically distributed. 
Set $\tilde{\tau}_{0,0}^N=0$. Let $I^N(0,x) = \max\{j\ge 0: \tilde{\tau}_{j,0}^N \le x \}$. Assume that there exists $\bar{x}>0$ such that  $I^N(0) = I^N(0,\bar{x})$. 

We will have the same description of the dynamics of $I^N(t)$ in \eqref{eqn-SIR-In}, however, the variables $\eta^0_j$ implicitly depend on $\tilde{\tau}_{j,0}^N$. One can explicitly write  $I^N(t)$ as 
\begin{align}\label{eqn-SIR-In-I0x}
I^N(t) = \sum_{j=1}^{I^N(0)} \bone_{\eta^0_j > t}\bone_{\tilde{\tau}_{j,0}^N \le \bar{x}}+ \sum_{i=1}^{A^N(t)} \bone_{\tau^N_i + \eta_i >t}, \quad t\ge 0.
\end{align}
Instead of  Assumption~\ref{AS-SIR-1}, we assume that the following holds.


\begin{assumption} \label{AS-SIR-initial2-LLN}
 There exists a deterministic function $\bar{I}(0,\cdot)\in (0,1)$ such that 
 $\bar{I}^N(0,\cdot) \to \bar{I}(0,\cdot)$ in $\bD$ in probability as $N\to\infty$. Then $\bar{I}^N(0)\to \bar{I}(0):=\bar{I}(0,\bar{x}) \in (0,1)$ in probability. In addition, we assume that $(\bar{S}^N(0), \bar{R}^N(0)) \to (\bar{S}(0), \bar{R}(0))$ such that $\bar{S}(0)+ \bar{I}(0)+ \bar{R}(0)=1$.
 \end{assumption}
Then it can be shown that  the FLLN in Theorem~\ref{thm-FLLN-SIR} holds 
with the same $\bar{S}(t)$ in \eqref{SIR-barS}, and 
\begin{align}
\bar{I}(t)  & =  \int_0^{\bar{x}} \frac{F^c(t+y)}{F^c(y)}  \bar{I}(0, dy) +  \int_0^t F^c(t-s)  \bar\Upsilon(s) ds,  \label{SIR-barI-I0x}\\
\bar{R}(t) &=  \bar{R}(0)+ \int_0^{\bar{x}}\left(1- \frac{F^c(t+y)}{F^c(y)}\right)  \bar{I}(0, d y)   +  \int_0^t F(t-s) \bar\Upsilon(s) ds.  \label{SIR-barR-I0x}
\end{align}
This recovers the result in \cite{wang1975limit}, specialized to the SIR model. 
In addition to the FLLN limits, one can also establish the FCLT and obtain the Gaussian limits. 
Instead of Assumption~\ref{AS-SIR-2}, we assume the following holds.

\begin{assumption} \label{AS-SIR-initial2-CLT}
 There exist  a deterministic function $\bar{I}(0,\cdot)\in (0,1)$, a stochastic process $\hat{I}(0,\cdot)$, a constant vector $(\bar{S}(0), \bar{R}(0))\in [0,1]^2$ and a random vector $(\hat{S}(0), \hat{R}(0))$ such that 
$(\hat{S}^N(0),\hat{I}^N(0, \cdot), \hat{R}^N(0))\RA (\hat{S}(0),\hat{I}(0, \cdot), \hat{R}^N(0)) $ in $\RR\times \bD\times R$ as $N \to \infty$, where 
 $\hat{I}^N(0, \cdot):=\sqrt{N} (\bar{I}^N(0,\cdot) - \bar{I}(0,\cdot))$.
 \end{assumption}
Then it can be shown that the FCLT in Theorem~\ref{thm-FCLT-SIR} holds with the same $\hat{S}(t)$ in
\eqref{SIR-Shat}, and 
\begin{equation}\label{SIR-Ihat-I0x}
\hat{I}(t) =  \int_0^{\bar{x}} \frac{F^c(t+y)}{F^c(y)}  \hat{I}(0, dy)   +  \int_0^t F^c(t-s) \hat\Upsilon(s) ds +  \hat{I}_{0}(t) +  \hat{I}_{1}(t),  
\end{equation}
\begin{equation}\label{SIR-Rhat-I0x}
\hat{R}(t) = \hat{R}(0)+  \int_0^{\bar{x}} \left(1-\frac{F^c(t+y)}{F^c(y)}\right)  \hat{I}(0, dy) +  \int_0^t F(t-s) \hat\Upsilon(s) ds + \hat{R}_{0}(t) +  \hat{R}_{1}(t), 
\end{equation}
with  $\bar{I}(t)$ is given in \eqref{SIR-barI-I0x}, and the limits $\hat{I}_{0}(t)$ and $\hat{R}_{0}(t)$ are as given in Theorem~\ref{thm-FCLT-SIR}. 
However, the limits $(\hat{I}_{0},\hat{R}_{0})$ are continuous Gaussian processes with covariance functions: for $t,t'\ge 0$, 
\begin{align}
\Cov (\hat{I}_{0}(t), \hat{I}_{0}(t')) &=\int_0^{\bar{x}} \left(\frac{F^c(t\vee t'+y)}{F^c(y)} -\frac{F^c(t+y)}{F^c(y)}\frac{F^c(t'+y)}{F^c(y)}\right)  \bar{I}(0, dy),  \label{eqn-cov-hatI0-a} \\
 \Cov (\hat{R}_{0}(t), \hat{R}_{0}(t')) &= \int_0^{\bar{x}} \left[ \left(1-\frac{F^c(t \wedge t'+y)}{F^c(y)}\right) - \left(1-\frac{F^c(t+y)}{F^c(y)}\right) \left(1-\frac{F^c(t'+y)}{F^c(y)}\right) \right]  \bar{I}(0, dy),  \label{eqn-cov-hatR0-a}\\
 \Cov (\hat{I}_{0}(t), \hat{R}_{0}(t')) &= \int_0^{\bar{x}} \left[ \left(\frac{F^c(t+y)}{F^c(y)} -\frac{F^c(t'+y)}{F^c(y)}\right) \bone_{t \le t'}  - \frac{F^c(t+y)}{F^c(y)} \left(1-\frac{F^c(t'+y)}{F^c(y)}\right) \right] \bar{I}(0, dy).  \label{eqn-cov-hatI0R0-a}
\end{align}
This recovers the result in \cite{wang1977gaussian}, specialized to the SIR model.

\subsection{The SEIR model}
In the $\bS\bE\bI\bR$ model, the population is split into four groups of individuals: Susceptible, Exposed, Infectious and Recovered/Immune.  
Let $S^N(t), E^N(t), I^N(t)$ and $R^N(t)$ be the numbers of susceptible, exposed and infectious and removed/immune individuals. 
Infections happen in the same way as the SIR model, that is, through contacts of infectious individuals and susceptible ones according to a Poisson process with rate $\lambda$. Thus, the instantaneous infection rate $\Upsilon^N(t)$ is given as in \eqref{Upsilon}. Then the cumulative process of newly infected individuals in $(0,t]$, $A^N(t)$, has the same expression as in \eqref{eq:AN}. Let $L^N(t)$ be the number of individuals that have become infectious after being exposed by time $t$. We have the following balance equations: for each $t\ge 0$, 
\begin{align}
N&=S^N(t) + E^N(t) +  I^N(t) + R^N(t) , \non\\
E^N(t) &=  E^N(0) + A^N(t) - L^N(t), \non\\
 I^N(t) &= I^N(0) + L^N(t) - R^N(t). \non
\end{align}

Each newly infected individual $i$ is associated with the time epoch of being exposed $\tau^N_i$, exposing duration $\xi_i$ and infectious duration $\eta_i$. Each initially infectious individual $j=1,\dots, I^N(0)$, is associated with the remaining infectious period $\eta_j^0$. Each initially exposed individual $k=1,\dots, E^N(0)$, is associated with the remaining exposing time $\xi^0_k$ and the infectious duration $\eta^0_k$. 

Then, we can represent the dynamics of $(S^N, E^N, I^N, R^N)$ as follows: for $t\ge 0$, 
\begin{align*}
S^N(t) &= S^N(0) - A^N(t) \\
E^N(t) &= \sum_{k=1}^{E^N(0)} \bone_{\xi_k^0 >t} +  \sum_{i=1}^{A^N(t)} \bone_{\tau_i^N + \xi_i >t}, \\
I^N(t) &= \sum_{j=1}^{I^N(0)} \bone_{\eta_j^0 > t} + \sum_{k=1}^{E^N(0)} \bone_{\xi_k^0 \le t} \bone_{\xi_k^0+ \eta_k >t}  + \sum_{i=1}^{A^N(t)} \bone_{\tau_i^N+ \xi_i \le t} \bone_{\tau_i^N+ \xi_i + \eta_i >t},\\
R^N(t) & = \sum_{j=1}^{I^N(0)} \bone_{\eta_j^0 \le t} +  \sum_{j=1}^{E^N(0)} \bone_{\xi_j^0+ \eta_j \le t}  + \sum_{i=1}^{A^N(t)} \bone_{\tau_i^N+ \xi_i + \eta_i \le t}. 
\end{align*}

Assume that $(\xi_i, \eta_i)$'s are i.i.d. bivariate random vectors with a joint distribution $H(du,dv)$, which has marginal c.d.f.'s $G$ and $F$ for $\xi_i$ and $\eta_i$, respectively, and a conditional c.d.f. of $\eta_i$, $F(\cdot |u)$ given that $\xi_i=u$. Assume that $(\xi^0_j, \eta_j)$'s are i.i.d. bivariate random vectors with a joint distribution $H_0(du,dv)$, which has marginal c.d.f.'s $G_0$ and $F$ for $\xi_j^0$ and $\eta_j$, respectively, and a conditional c.d.f. of $\eta_j$, $F_0(\cdot|u)$ given that  $\xi^0_j=u$.  (Note that the pair $(\xi^0_j, \eta_j)$ is the remaining exposing time and the subsequent infectious period for the $i^{\rm th}$ individual initially being exposed.)
In addition, we assume that the sequences $\{\xi_i, \eta_i,\ i\ge1\}$, $\{\xi^0_j, \eta_j,\ 1\le j\le E^N(0)\}$ and $\{\eta^0_j,\ 1\le j\le I^N(0)\}$ are mutually independent. 
We use the notation $G^c=1-G$, and similarly for $G_0^c$, $F^c$ and $F_0^c$. 
Define 
\begin{align*}
 \Phi_0(t)&:= \int_0^t \int_0^{t-u} H_0(du, dv) =\int_0^t \int_0^{t-u} F_0(dv|u) d G_0(u), \\
  \Psi_0(t) &:= \int_0^t \int_{t-u}^\infty H_0(du, dv)  = \int_0^t \int_{t-u}^\infty F_0(dv| u) d G_0(u) = G_0(t)- \Phi_0(t),
\end{align*}
and 
\begin{align*}
 \Phi(t)&:= \int_0^t \int_0^{t-u} H(du, dv) =\int_0^t \int_0^{t-u} F(dv|u) d G(u), \\
  \Psi(t) &:= \int_0^t \int_{t-u}^\infty H(du, dv)  = \int_0^t \int_{t-u}^\infty F(dv| u) d G(u) = G(t)- \Phi(t). 
  \end{align*}
Note that in the case of independent $\xi_i$ and $\eta_i$, e have $F(dv) = F(dv|u)$, and 
\begin{align*} 
\Phi(t) = \int_0^t F(t-u) d G(u), \quad \Psi(t) = \int_0^t F^c(t-u) d G(u) = G(t) - \Phi(t). 
\end{align*} 
Similarly, with independent $\xi_j^0$ and $\eta_j$, we have  $F_0(dv) = F_0(dv|u)=F(dv)$, and 
\begin{align*}  
\Phi_0(t) = \int_0^t F(t-u) d G_0(u), \quad \Psi_0(t) = \int_0^t F^c(t-u) d G_0(u) = G_0(t) - \Phi_0(t). 
\end{align*}

\begin{assumption} \label{AS-SEIR-1}
There exists a deterministic constant $(\bar{S}(0),\bar{E}(0),\bar{I}(0),\bar{R}(0))\in [0,1]^4$ such that 
$\bar{S}(0)+\bar{E}(0)+\bar{I}(0)+\bar{R}(0)=1$, and  as $N\to\infty$,
 $(\bar{S}^N(0),\bar{E}^N(0),\bar{I}^N(0),\bar{R}^N(0)) \to (\bar{S}(0),\bar{E}(0),\bar{I}(0),\bar{R}(0))$ in probability. 
\end{assumption}

Define the LLN-scaled processes as in the $\bS\bI\bR$ model. 

\begin{theorem}\label{thm-FLLN-SEIR}
{\bf (Functional Law of Large Numbers for the }$\bS\bE\bI\bR$ {\bf model}).
Under Assumption~\ref{AS-SEIR-1}, we have 
\begin{equation*} 
\left(\bar{S}^N, \bar{E}^N,  \bar{I}^N, \bar{R}^N \right) \to \left(\bar{S}, \bar{E},  \bar{I}, \bar{R} \right) \qinq \bD^4
\end{equation*}
in probability as $n\to\infty$, where the limit process $(\bar{S}, \bar{E}, \bar{I}, \bar{R}) $ is the unique solution to the system of deterministic equations: for each $t\ge 0$, 
\begin{align}
\bar{S}(t) &=  \bar{S}(0)  - \bar{A}(t)  =  \bar{S}(0) -   \int_0^t  \bar\Upsilon(s) ds,  \nonumber \\
\bar{E}(t) &= \bar{E}(0) G_0^c(t) + \int_0^t G^c(t-s) \bar\Upsilon(s)  ds,  \label{eqn-barE-SEIR} \\
\bar{I}(t) 
& =  \bar{I}(0) F^c_0(t) +\bar{E}(0) \Psi_0(t)  +    \int_0^t \Psi(t-s) \bar\Upsilon(s)  ds,    \label{eqn-barI-SEIR}\\
\bar{R}(t) 
& = \bar{R}(0)+ \bar{I}(0) F_0(t) +\bar{E}(0) \Phi_0(t)  +    \int_0^t \Phi(t-s)  \bar\Upsilon(s)  ds .  \nonumber
\end{align} 
with $\bar\Upsilon(t) = \lambda \bar{S}(t) \bar{I}(t)$, the same in \eqref{eqn-barUpsilon-constlambda} for the SIR model.  
The limit $\bar{S}$ is in $\bC$ and $\bar{E}$, $\bar{I}$ and $\bar{R}$ are in $\bD$.  If $G_0$ and $F_0$ are continuous, then they are in $\bC$. 
\end{theorem}

For the proof of Theorem~\ref{thm-FLLN-SEIR}, as well as for the associated Functional Central Limit Theorem, we refer the reader to \cite{PP-2020}.

\noindent{\bf An alternative initial condition.}
For the initially exposed individuals, let $\tau^N_{j,0}, j=1,\dots, E^N(0)$ be the time being exposed, that is, $\tilde{\tau}^N_{j,0} = -\tau^N_{j,0} $ is the corresponding elapsed duration of being exposed at time zero.
For the initially infectious individual $k=1,\dots, I^N(0)$, let $\varsigma^N_{k,0}$ be the time when an initially infectious individual becomes infectious, and thus $\tilde{\varsigma}^N_{k,0}= - \varsigma^N_{k,0}$ is the elapsed time at time 0 since being infectious. 
WLOG, assume that  $0>\tau^N_{1,0}> \cdots > \tau^N_{E^N(0),0}$ (equivalently,  $0<\tilde{\tau}^N_{1,0}<\cdots <\tilde{\tau}^N_{E^N(0),0}$). Set $\tilde{\tau}^N_{0,0}=0$. Then we can write $E^N(0,x) = \max\{j \ge 0: \tilde{\tau}^N_{j,0} \le x\}$. 
Similarly, assume that $0 > \varsigma^N_{1,0}>\cdots> \varsigma^N_{I^N(0),0}$ (equivalently, $0 > \tilde{\varsigma}^N_{1,0}>\cdots> \tilde{\varsigma}^N_{I^N(0),0}$) and set $\tilde{\varsigma}^N_{0,0}=0$. Then we can write $I^N(0,x) = \max\{k \ge 0: \tilde{\varsigma}^N_{k,0} \le x\}$. We also assume that there exist constants $0 \le \bar{x}^e <\infty$ and $0 \le \bar{x} <\infty$ such that $E^N(0) = E^N(0,\bar{x}_e)$ and 
$I^N(0) = I^N(0, \bar{x})$ a.s. 

For the initially  infectious individuals, given their elapsed infection times $\tilde{\varsigma}^N_{k,0}$, $k=1,\dots, I^N(0)$, we assume that their remaining infectious times are conditional independent and have distributions dependent on their own infection ages, that is, given that the $\tilde{\varsigma}^N_{k,0}=s$, 
$\P(\eta^0_k > t| \tilde{\varsigma}^N_{k,0}=s) = \frac{F^c(t+s)}{F^c(s)}, \, t, s \ge 0.$
For the initially exposed individuals, the pairs   $\{(\xi^0_j, \eta^{E,0}_j)\}$ are assumed to be conditionally independent given the elapsed exposed times $\{\tilde{\tau}^N_{j,0}\}$, and the distribution of $(\xi^0_j, \eta^{E,0}_j)$ depends on the $\tilde{\tau}^N_{j,0}$. In particular, given that $\tilde{\tau}^N_{j,0}=s$, the joint distribution of $(\xi^0_j, \eta^{E,0}_j)$ is given  by $H_0(du,dv|s)$. 
Assume that the marginal distribution of $\xi^0_j$ given that $\tilde{\tau}^N_{j,0} =s$ is given by
$G^c(t|s) =1- G(t|s)=\P(\xi^0_j >t| \tilde{\tau}^N_{j,0}=s) = \frac{G^c(t+s)}{G^c(s)}, \, t, s\ge 0.$
and the conditional distribution of $\eta^{E,0}_j$ given $\tilde{\tau}^N_{j,0}=s$ and $\xi^0_j=u$, and given by $F_0(\cdot|u,s)=F_0(\cdot|s+u)$. Thus, $H_0(du,dv|s) = G(du|s) F_0(dv|u,s) = G(du|s) F_0(dv|u+s) $. 

Let 
\begin{align*}
 \Phi_0(t|s)&:= \int_0^t \int_0^{t-u} H_0(du, dv|s) =\int_0^t \int_0^{t-u} F_0(dv|u,s) G(du|s), 
 \\
  \Psi_0(t|s) &:= \int_0^t \int_{t-u}^\infty H_0(du, dv|s)  = \int_0^t \int_{t-u}^\infty F_0(dv| u,s) G(du|s) = G(t|s)- \Phi_0(t|s). 
\end{align*}
Similarly, with independent $\xi_j^0$ and $\eta_j^{E,0}$ given that $\tilde{\tau}^N_{j,0}=s$, we have 
\begin{align*} 
\Phi_0(t|s) = \int_0^t F_0(t-u|s)  G(du|s), \quad \Psi_0(t|s) = \int_0^t F^c_0(t-u|s) G(du|s) = G(t|s) - \Phi_0(t|s). 
\end{align*}
If, in addition, $F_0(\cdot|s) =F(\cdot)$,
then
\begin{align*} 
\Phi_0(t|s) = \int_0^t F(t-u)  G(du|s), \quad \Psi_0(t|s) = \int_0^t F^c(t-u) G(du|s) = G(t|s) - \Phi_0(t|s). 
\end{align*}

Instead of Assumption \ref{AS-SEIR-1}, we assume that 
there exist deterministic continuous nondecreasing functions $\bar{E}(0,x)$  and $\bar{I}(0,x)$ for $x\ge 0$ with $\bar{E}(0,0)=0$ and $\bar{I}(0,0)=0$ such that 
$\big(\bar{E}^{N}(0,\cdot), \bar{I}^{N}(0,\cdot)\big)  \to \big(\bar{E}(0,\cdot), \bar{I}(0,\cdot)\big)$ in $\bD^2$  in probability as $N\to\infty$. 
Then $(\bar{S}^N(0), \bar{E}^N(0), \bar{I}^N(0), \bar{R}(0) ) \to ( \bar{S}(0),\bar{E}(0),  \\
\bar{I}(0),\bar{R}(0))$ in $\RR^4_+$ in probability as $N\to \infty$, where $\bar{E}(0) = \bar{E}(0, \bar{x}_e) \in (0,1)$ and $\bar{I}(0) = \bar{I}(0, \bar{x}) \in (0,1)$, and $\bar{S}(0) + \bar{R}(0) =1 - \bar{E}(0) - \bar{I}(0)\in (0,1)$.
We can then prove the FLLN with the limit $\bar{S}(t)$ given in Theorem \ref{thm-FLLN-SEIR}, 
and the limits
\begin{align*}
\bar{E}(t) &=   \int_0^{\bar{x}_e} \frac{G^c(t+y)}{G^c(y)}  \bar{E}(0, dy)  +   \int_0^t G^c(t-s) \bar\Upsilon(s)  ds,   
\\
\bar{I}(t)  & =    \int_0^{\bar{x}} \frac{F^c(t+y)}{F^c(y)}  \bar{I}(0, dy)   +   \int_0^{\bar{x}_e} \int_{0}^t \int_{t-u}^\infty H_0(du,dv|y) \bar{E}(0,dy)  \non \\
& \qquad +    \int_0^t \Psi(t-s) \bar\Upsilon(s)  ds,  
\\
\bar{R}(t) & = \bar{R}(0)  +    \int_0^{\bar{x}} \left(1- \frac{F^c(t+y)}{F^c(y)} \right) \bar{I}(0, dy) +   \int_0^{\bar{x}_e} \int_{0}^t \int_0^{t-u} H_0(du,dv|y) \bar{E}(0,dy)  \non \\
& \qquad   +   \int_0^t \Phi(t-s) \bar\Upsilon(s)  ds . 
\end{align*} 

An FCLT can be similarly established, which we omit for brevity. 

\subsection{Multipatch $\bS\bI\bR$ model}\label{multipatch}


In order to model geographic heterogeneity,  multi-patch models have been used to study various infectious diseases.  ODE models are often used to study the dynamics of these models, arising as scaling limits of Markovian stochastic model, with exponentially distributed exposed/infectious periods and Markovian migration processes 
\cite{andersson2012stochastic,brittonpardoux,bichara2018multi,iggidr2016dynamics}. See also the relevant  models in \cite{ball2008network,magal2016final,magal2018final}.
Multi--patch non--Marvovian $\bS\bE\bI\bR$ models have recently been studied in  \cite{PP-2020b}. 
We now present the $\bS\bI\bR$ version of that model. 

The patches may refer to populations in different locations, for example, a densely populated city and a less populated rural area. 
Individuals in each patch are infected locally and from distance.
The main reason for allowing infection at distance is the following. Infectious individuals may travel from one
patch to another for work or vacation, and then return home. These roundtrip travels are hard to model as such. 
We prefer to consider the infections while away from home as infections at distance. 
The rate of infection is different in the patches (because of the differences in the 
density of population or the use of public transportations), while the law of the infectious period is the same (same illness).

Let $N$ be the total population size and $L$ be the number of patches.
For each $i=1,\dots,L$, let $S^N_i(t)$, $I^N_i(t)$ and $R^N_i(t)$ denote the numbers of individuals in patch $i$ that are susceptible, infectious and recovered at time $t$, respectively. 
Then we have the balance equation:
$$
N\;=\; \sum_{i=1}^L B_i^N(t), \quad \text{where} \quad  B_i^N(t) := S^N_i(t) + I^N_i(t) + R^N_i(t)  \,, \quad t \ge 0\,.
$$
Assume that $S^N_i(0)>0$,  $\sum_{i=1}^L I^N_i(0)>0$ and $R^N_i(0) =0$, $i=1,\dots,L$. 

Let  $\lambda_i$ be the infection rate of patch $i$, $i=1,\dots,L$.  
The instantaneous infection rate process of patch $i$ is given by, 
\begin{align*} 
\Upsilon^N_i(t) &\;=\;\frac{ \lambda_i S^N_i(t)\sum_{j=1}^L \kappa_{ij}I^N_j(t)}{N^{1-\gamma}(S^N_i(t)+I^N_i(t)+R^N_i(t))^\gamma}  \\
& = \lambda_i \Big( \frac{B_i^N}{N} \Big)^{1-\gamma} \frac{S_i^N(t)}{B_i^N(t)} \sum_{j=1}^L \kappa_{ij}I^N_j(t)
\,,\quad i =1,\dots, L\,,
\end{align*}
where  $\kappa_{ii}=1$  and $0\le \kappa_{ij} < 1$ for $i\neq j$ represent the infectivity from distance, and $0\le\gamma\le1$. 
 In the case $\gamma=1$, the rate of encounters of individuals in patch $i$ by a given infectious is given as $\lambda_i$ for an infectious of the same patch, and equal to $\lambda_i\kappa_{ii'}$ for an infectious from patch $i'$, whatever the total population in patch $i$ at time $t$ may be. This factor gets multiplied by the probability that a randomly chosen individual in patch $i$ be susceptible, which equals $S^N_i(t)/B^N_i(t)$. 
 In the case 
$\gamma=0$, the same rate is proportional to $B_i^N(t)$, the total population of patch $i$ at time $t$. In the intermediate cases, the rate lies between those two extremes. The FLLN is proved for for any value of $\gamma\in[0,1]$, and the FCLT  is only for $\gamma\in[0,1)$ in the general case, and for all $\gamma\in[0,1]$ in the case that infections are only local, \textit{i.e.},
$\kappa_{ij} =0$ for $i\neq j$. 
By convention, we shall assume that $\Upsilon^N_i(t)=0$ whenever $S^N_i(t)+I^N_i(t)+R^N_i(t)=0$ if $\gamma<1$, and the same in case $\gamma=1$, \textit{i.e.}, $\frac{0}{0}=0$ (it is of course $0$ if patch $i$ is empty). 

Let $A^N_i(t)$ be the cumulative counting process of individuals in patch $i$ that become infectious during $(0,t]$. 
Then we can give a representation of the process $A^N_i(t)$ via the standard Poisson random measure $Q_{i}$ on $\R^2_+$ (with mean measure $dsdu$): 
\begin{align}\label{An-rep-1}
A^N_i(t)\;=\;\int_0^t\int_0^\infty{\bf1}_{u\le \Upsilon^N_i(s)}Q_{i}(ds,du)\,,\quad t\ge 0\,. 
\end{align}
Equivalently, we can write 
\begin{align} \label{An-rep-2}
A^N_i(t) \;=\; A_{i,*} \left( \int_0^t  \Upsilon^N_i(s) ds \right) \,,\quad t\ge 0\,,
\end{align}
where $\{A_{1,*},\ldots,A_{L,*}\}$ are i.i.d. unit-rate Poisson processes. 
We let $\{\tau^N_{j,i},\ j\ge1\}$ denote the successive jump times of the process
$A^N_i$, for $i=1,\dots,L$. 

For the initially infected individuals,  let $\eta^0_{k,i}$, $k=1,\dots, I^N_i(0)$, denote their remaining infectious periods.  Assume that $\{\eta^0_{k,i}\}$ are independent and identically distributed (i.i.d.) with a cumulative distribution function (c.d.f.) $F_0$, for all $i,k$. 
For the newly infected individuals $A^N_i(t)$, let $\eta_{k,i}$, $k\in \N$, denote their remaining infectious periods.  Assume that $\{\eta_{k,i}\}$ are i.i.d. with a c.d.f. $F$, for all $i,k$.  Let $F^c_0=1-F_0$ and $F^c=1-F$.  It is reasonable to assume the same distribution for the infectious periods of individuals of the different patches since it is the same illness. 

Susceptible (resp. infectious, resp. removed) individuals migrate from patch $i$ to patch $j$ at rate $\nu_{S,i,j}$ (resp. at rate $\nu_{I,i,j}$, resp. at rate $\nu_{R,i,j}$)).
Let $X(t)$ denote the location (\textit{i.e.}, the patch) 
at time $t\ge 0$ of an infected individual. It is clear that 
$X(t)$ is a Markov process which alternates between states $1,\dots, L$. 
Define $p_{i,j}(t) =\P(X(t)=j|X(0)=i)$ for $i,j=1,\dots,L$ and $t \ge 0$. 
(Note that the probability is the same for any starting time, for example, $p_{i,j}(t) =\P(X(r+t)=j|X(r)=i)$ for any $r, t\ge 0$.)
%

We will use $X^{0,k}_i$ and $X^{k}_i$ to indicate the associated process for individual $k$ in patch $i$, for the initially and newly infected ones, respectively. 
Note that they are all mutually independent and have the same law as described for the process $X(t)$ above. Also note that the processes  $X^{k}_i$ start from the time becoming infected $\tau^N_{k,i}$ while the processes $X^{0,k}_i$ start from time $0$.

We now provide a representation of the epidemic evolution dynamics: 
\begin{align}
S^N_i(t)&\;=\;S^N_i(0)-A^N_i(t)
- \sum_{\ell=1,\ell \neq i}^L P_{S,i,\ell}\left(\nu_{S,i,\ell}\int_0^t S^N_i (s)ds\right)
+\sum_{\ell=1,\ell\neq i}^L P_{S,\ell,i}\left( \nu_{S,\ell,i} \int_0^tS^N_\ell(s)ds\right)\,, 
\non\\
I^N_i(t)&\;=\; \sum_{\ell=1}^L \sum_{k=1}^{I^N_\ell(0)} {\bf1}_{\eta_{k,\ell}^0>t}{\bf1}_{X^{0,k}_\ell(t)=i}
+ \sum_{\ell=1}^L \sum_{j=1}^{A^N_\ell(t)} {\bf 1}_{\tau^N_{j,\ell} + \eta_{j,\ell} >t} {\bf1}_{X^j_\ell(t-\tau^N_{j,\ell})=i}\,\,, \label{In-1-rep}
\\
R^N_i(t)&\;=\; R^N_i(0)+ \sum_{\ell=1}^L \sum_{k=1}^{I^N_\ell(0)} {\bf1}_{\eta_{k,\ell}^0\le t}{\bf1}_{X^{0,k}_\ell(\eta_{k,\ell}^0)=i}
+\sum_{\ell=1}^L \sum_{j=1}^{A^N_\ell(t)}  {\bf 1}_{\tau^N_{j,\ell} + \eta_{j,\ell} \le t} {\bf1}_{X^j_\ell( \eta_{j,\ell})=i} \non
\\& \quad - \sum_{\ell=1,\ell \neq i}^L P_{R,i,\ell}\left(\nu_{R,i,\ell}\int_0^t R^N_i (s)ds\right)
+\sum_{\ell=1,\ell\neq i}^L P_{R,\ell,i}\left( \nu_{R,\ell,i} \int_0^t R^N_\ell(s)ds\right) \,,\non 
\end{align}
where $P_{S,i,\ell}, P_{R,i,\ell}$ , $i,\ell=1,\dots,L$, are all unit-rate Poisson processes, mutually independent, and also independent of $P_{A,i}$. 
Here, the first term in $I^N_i(t)$ represents the number of initially infected individuals from patch $\ell=1,\dots,L$ that remain infected and are in patch $i$ at time $t$, and the second term represents the number of newly infected individuals from patch $\ell=1,\dots,L$ that remain infected and are in patch $i$ at time $t$. The first term in $R^N_i(t)$ represents the 
number of initially infected individuals from patch $\ell=1,\dots,L$ that have recovered by time $t$ and were in patch $i$ at the time of recovery,  and the second term represents the number of newly infected individuals from patch $\ell=1,\dots,L$ that have recovered by time $t$, and were in patch $i$ at the time of recovery. 

It is not easy to take the limit as $N\to\infty$ in the formula \eqref{In-1-rep}. We give another representation of the process $I^N_i(t)$. 
\begin{lemma} \label{lem-I1-2-rep}
We have
\begin{align*} 
I^N_i(t)&\;=\; I^N_i(0)  + A^N_i(t)  -   \sum_{\ell=1}^L \sum_{k=1}^{I^N_\ell(0)} {\bf1}_{\eta_{k,\ell}^0\le t}{\bf1}_{X^{0,k}_\ell(\eta_{k,\ell}^0)=i} 
- \sum_{\ell=1}^L \sum_{j=1}^{A^N_\ell(t)}  {\bf 1}_{\tau^N_{j,\ell} + \eta_{j,\ell} \le t} {\bf1}_{X^j_\ell( \eta_{j,\ell})=i} \non
\\& \quad - \sum_{\ell \neq i} P_{I,i,\ell}\left(\nu_{I,i,\ell}\int_0^t I^N_i (s)ds\right)
+\sum_{\ell\neq i} P_{I,\ell,i}\left( \nu_{I,\ell,i} \int_0^t I^N_\ell(s)ds\right) \,,
\end{align*}
where $P_{I,i,j}$, $i,j=1,\dots,L$, are all unit-rate Poisson processes, mutually independent, and also independent of $P_{A,i}$, $P_{S,i,j}$ and  $P_{R,i,j}$. 
\end{lemma}
Let us comment on this formula. The first term counts the number of initially infectious individuals in patch $i$, and the second term adds the number of those who get infected in patch $i$ on the time interval $(0,t]$.
The last two terms describe the movements of infectious individuals out of $i$, and into $i$. The third terms subtracts  the number of individuals initially infected in any patch, who have recovered before time $t$ in patch $i$, and the fourth term subtracts the number of individuals infected on the time interval $(0,t]$ in any patch, who have recovered before time $t$ in patch $i$.

Define a PRM 
$\widetilde{Q}_{\ell}(ds,du, dv, d\theta)$ on $\R_+^3\times\{1,\dots,L\}$, which is the sum of the Dirac masses at the points  
$(\tau_{j,\ell}^N, U^N_{j,\ell}, \eta_{j,i}, X_\ell^j(\eta_{j,\ell}))$ with mean measure $ds\times du\times F(dv)\times \mu_\ell(v,d\theta)$, where
for each $v>0$, $\mu_\ell(v,\{\ell'\})=p_{\ell,\ell'}(v)$, and an infection occurs at time $\tau_{j,\ell}^N$ in case $U^N_{j,\ell}\le \Upsilon^N_\ell(\tau^N_{j,\ell})$. 

We can then write for $\ell, \ell' =1,\dots,L$,
\begin{align*} 
 \sum_{j=1}^{A^N_\ell(t)}  {\bf 1}_{\tau^N_{j,\ell} + \eta_{j,\ell} \le t} {\bf1}_{X^j_\ell( \eta_{j,\ell})=\ell'} 
 \;=\; \int_0^t \int_0^\infty \int_0^{t-s}\int_{\{\ell'\}} {\bf 1}_{u \le \Upsilon^N_\ell(s)} \widetilde{Q}_{\ell,inf}(ds,du, dv, d\theta). 
\end{align*}

\begin{remark}
 If we set some migration rates to zero, then the corresponding patches could be considered as sub-groups (like age groups) of the population, which  interact and infect one another. 
\end{remark}

For any process $Z^N = S^N_i, I^N_i$, or $R^N_i$, $i=1,\dots,L$, let $\bar{Z}^N:=N^{-1}Z^N$.   

\begin{assumption} \label{AS-FLLN}
There exist constants $0<\S_i(0)\le1$, $0\le\I_i(0)<1$, $0\le \rR_i(0)\le1$ with $\sum_{i=1}^L\I_i(0)>0$ such that  $\sum_{i=1}^L (\S_i(0)+ \I_i(0)+\rR_i(0)) =1$ and  $(\bar{S}^N_i(0), \bar{I}^N_i(0),\bar{R}^N_i(0)\, i=1,\dots, L)   \to (\bar{S}_i(0), \bar{I}_i(0),\bar{R}_i(0),\, i=1,\dots,L)$ in probability in $\R^{3L}$ as $N\to\infty$. In addition, assume that $F_0$ is continuous.
\end{assumption}

\begin{theorem} \label{thm-FLLN} 
Under Assumption~\ref{AS-FLLN}, 
 \begin{equation} \label{eqn-FLLN-conv}
 (\S^N_i,\I^N_i,\rR^N_i,\, i=1,\dots,L)\;\to\;(\S_i,\I_i,\rR_i,\, i=1,\dots,L) \qinq \bD^{3L} \qasq N \to \infty\,, 
 \end{equation}
in probability, locally uniformly  on $[0,T]$, where $(\S_i(t),\I_i(t),\rR_i(t),\, i=1,\dots,L) \in \bC^{3L}$ is the unique  solution to the following set of  deterministic integral equations: 
 \begin{align}
 \S_i(t) &\;=\; \S_i(0) -  \int_0^t \bar\Upsilon_i(s) ds  +   \sum_{\ell=1,\ell \neq i}^L\int_0^t \left( \nu_{S,\ell,i}  \S_\ell(s)-  \nu_{S,i,\ell}\S_i(s)  \right) ds\,, \non \\ 
\I_i(t) &\;=\;  \I_i(0) -  \int_0^t   \sum_{\ell=1}^L \I_\ell(0)  p_{\ell,i}(s) F_0(ds)  +   \int_0^t  \bar\Upsilon_i(s) ds  \non \\
 & \quad  -  \int_0^t  \sum_{\ell=1}^L  \left(  \int_0^{t-s} p_{\ell,i}(u) F(du)  \right)  \bar\Upsilon_\ell(s) ds +  \sum_{\ell \neq i} \int_0^t \left( \nu_{I,\ell,i}  \I_\ell(s)-  \nu_{I,i,\ell}\I_i(s)  \right) ds \,, \label{eqn-barIi}\\
\rR_i(t) &\;=\; \rR_i(0)+\int_0^t   \sum_{\ell=1}^L \I_\ell(0) p_{\ell,i}(s) F_0(ds)  +  \int_0^t  \sum_{\ell}  \left(  \int_0^{t-s} p_{\ell,i}(u) F(du)  \right)    \bar\Upsilon_\ell(s) ds \non\\ 
 & \quad +  \sum_{\ell=1,\ell\neq i}^L  \int_0^t \left( \nu_{R,\ell,i}  \rR_\ell(s)-  \nu_{R,i,\ell}\rR_i(s)  \right) ds\, , \non 
 \end{align}
with $\bar\Upsilon_{i}$ defined by
\begin{align*} 
\bar\Upsilon_i(t)\;=\;\frac{ \lambda_i\S_i(t)\sum_{j=1}^L\kappa_{ij}\I_j(t)}{(\S_i(t)+\I_i(t)+\rR_i(t))^\gamma} = \lambda_i \frac{\S_i(t)}{\bar{B}_i(t)^\gamma} \sum_{j=1}^L\kappa_{ij}\I_j(t)\,. 
\end{align*}
Here $\bar{B}_i(t):=\S_i(t)+\I_i(t)+\rR_i(t)$. 
\end{theorem}

For any process $Z^N$,  let $\hat{Z}^N:=\sqrt{N} ( \bar{Z}^N - \bar{Z})$ be the diffusion-scaled process where $ \bar{Z}^N$ is the fluid-scaled process and $\bar{Z}$ is its limit.   

\begin{assumption} \label{AS-FCLT}
There exist constants $0<\S_i(0)\le1$, $0\le\I_i(0)<1$ , $0\le \rR_i(0)<1$ with $\sum_{i=1}^L\I_i(0)>0$ such that  $\sum_{i=1}^L (\S_i(0)+ \I_i(0)+\rR_i(0)) =1$,   and random variables $\hat{S}_i(0)$, $\hat{I}_i(0)$ and $\hat{R}_i(0)$, $i=1,\dots,L$, such that 
 $(\hat{S}^N_i(0), \hat{I}^N_i(0), \hat{R}^N_i(0),\, i=1,\dots,L) \;\RA\; (\hat{S}_i(0), \hat{I}_i(0),\hat{R}_i(0),\, i=1,\dots,L) $ in $\R^{3L}$ as $N\to\infty$. In addition,  for $i=1,\ldots,L$,
 \[\sup_N\E\left[(\hat{S}^N_i(0))^2+(\hat{I}^N_i(0))^2+(\hat{R}^N_i(0))^2\right]<\infty\,.\]
\end{assumption}

\begin{theorem} \label{thm-FCLT} 
Under Assumption ~\ref{AS-FCLT},  in the two cases (i) $\gamma\in[0,1)$ or (ii)  $\gamma\in[0,1]$ and $\sum_{j\not=1}\kappa_{ij}=0$, 
 \begin{equation*} 
 (\hat{S}^N_i,\hat{I}^N_i,\hat{R}^N_i,\, i=1,\dots,L)\;\to\;(\hat{S}_i(t),\hat{I}_i(t),\hat{R}_i(t),\, i=1,\dots,L) \qinq \bD^{3L} \qasq N \to \infty, 
 \end{equation*}
where the limits are the unique solution to the following set of stochastic Volterra  integral equations driven by Gaussian processes: 
 \begin{align*} 
\hat{S}_i(t) &\;=\; \hat{S}_i(0) -   \int_0^t \hat{\Upsilon}_i(s) ds + \sum_{\ell=1, \ell\neq i}^L \int_0^t (\nu_{S,\ell,i}\hat{S}_\ell(s)- \nu_{S,i, \ell}\hat{S}_i(s)) ds \non\\
& \qquad - \hat{M}_{A,i}(t)  +\sum_{\ell=1, \ell\neq i}^L \big( \hat{M}_{S,\ell,i}(t) - \hat{M}_{S,i,\ell}(t) \big)\,,
\end{align*}
\begin{align} \label{eqn-hatI-1}
\hat{I}_i(t) &\;=\; \hat{I}_i(0)  \bigg( 1- \sum_{\ell=1}^L\int_0^t  p_{\ell,i}(s) F_0(ds) \bigg) + \int_0^t \hat{\Upsilon}_i(s) ds 
 - \sum_{\ell=1}^L \int_0^t \int_0^{t-s} p_{\ell,i}(u) F(du)  \hat\Upsilon_\ell(s) ds \non\\
 & \quad  + \sum_{\ell=1, \ell\neq i}^L \int_0^t (\nu_{I,\ell,i}\hat{I}_\ell(s)- \nu_{I,i, \ell}\hat{I}_i(s)) ds  -\sum_{\ell=1}^L \big( \hat{I}^{0}_{\ell,i}(t) + \hat{I}_{\ell,i}(t) \big)    \non\\
& \quad + \hat{M}_{A,i}(t)+  \sum_{\ell=1, \ell \neq i}^L \big( \hat{M}_{I,\ell,i}(t) - \hat{M}_{I,i,\ell}(t) \big) \,, 
\end{align}
\begin{align*} 
\hat{R}_i(t) &\;=\;  \hat{R}_i(0)+\hat{I}_i(0)  \sum_{\ell=1}^L\int_0^t  p_{\ell,i}(s) F_0(ds)  + 
\sum_{\ell=1}^L  \int_0^t \int_0^{t-s} p_{\ell,i}(u) F(du)  \hat\Upsilon_\ell(s) ds  \non\\
& \quad  +  \sum_{\ell=1, \ell\neq i}^L \int_0^t (\nu_{R,\ell,i}\hat{R}_\ell(s)- \nu_{R,i, \ell}\hat{R}_i(s)) ds  \non\\
& \quad + \sum_{\ell=1}^L \big( \hat{I}^{0}_{\ell,i}(t) + \hat{I}_{\ell,i}(t) \big)  +   \sum_{\ell=1, \ell \neq i}^L \big( \hat{M}_{R,\ell,i}(t) - \hat{M}_{R,i,\ell}(t) \big)\,. 
\end{align*}
Here, with the notation $\bar{I}_{(i)}(t)=\sum_{j=1}^L \kappa_{ij}\bar{I}_j(t)$ and $\bar{B}_i(t):=\S_i(t)+\I_i(t)+\rR_i(t)$, 
\begin{align*} 
\hat{\Upsilon}_i(t) &=\frac{\lambda_i}{\bar{B}_i(t)^{(1+\gamma)}} \Big( [(1-\gamma)\bar{S}_i(t)\!+\!\bar{I}_i(t)\!+\! \bar{R}_i(t)] \bar{I}_{(i)}(t) \hat{S}_i(t)\!  \non\\
& \quad+\!
\left[\bar{S}_i(t)(\S_i(t)\!+\!\I_i(t)\!+\!\rR_i(t))\!-\!\gamma\bar{S}_i(t)\bar{I}_{(i)}(t)\right]\hat{I}_i(t) 
\!-\! \gamma\bar{S}_i(t) \bar{I}_{(i)}(t) \hat{R}_i(t) \Big) +\frac{\bar{S}_i(t)\sum_{j\not=i}\kappa_{ij}\hat{I}_j(t)}{\bar{B}_i(t)^\gamma}\,, 
\end{align*}
\begin{align*}
&\hat{M}_{A,i}(t) \;=\; B_{A,i} \left(  \int_0^t   \bar\Upsilon_i(s) ds\right)\,,
\quad
\hat{M}_{S,i,j}(t) \;=\; B_{S,i,j}\left(\nu_{S,i,j} \int_0^t \bar{S}_i(s)ds\right)\,, \\
&
\hat{M}_{I,i,j}(t) \; =\; B_{I,i,j}\left(\nu_{I,i,j} \int_0^t \bar{I}_i(s)ds\right)\,, \quad 
\hat{M}_{R,i,j}(t)\; =\; B_{R,i,j}\left(\nu_{R,i,j} \int_0^t \bar{R}_i(s)ds\right)\,, \quad i \neq j\,,
\end{align*}
with $B_{A,i}$, $B_{S,i,j}$,  $B_{I,i,j}$, $B_{R,i,j}$ being mutually independent standard Brownian motions, and with the deterministic functions $\bar{S}_i, \bar{I}_i, \bar{R}_i$ being the limits in Theorem~\ref{thm-FLLN}. 
The processes $ \hat{I}^{0}_{i,j}$ and  $\hat{I}_{i,j}$ are continuous Gaussian processes with mean zero and covariance functions: 
\begin{align*}
Cov( \hat{I}^0_{i,j}(t),  \hat{I}^0_{i',j'}(t') ) &\;=\; \begin{cases}
\I_i(0)  \Big(\int_0^{t\wedge t'}  p_{i,j}(s) F_0(ds) -  \int_0^t  p_{i,j}(s) F_0(ds)  \int_0^{t'}  p_{i,j}(s) F_0(ds)\Big)\,,   &\text{if}\quad   i=i',j=j',\\
 0\,,  & otherwise,
\end{cases}
\end{align*}
\begin{align*}
Cov( \hat{I}_{i,j}(t),  \hat{I}_{i',j'}(t') ) &\;=\; \begin{cases}
 \int_0^{t\wedge t'} \int_0^{t\wedge t'-s} p_{i,j}(u) F(du)  \bar\Upsilon_i(s) ds\,,   &\text{if}\quad   i=i'\,, \, j=j'\,,\\
 0\,,  & otherwise.
\end{cases}
\end{align*}
In addition, $ \hat{I}^{0}_{i,j}$ and  $\hat{I}_{i,j}$ are independent, and also independent of the Brownian terms. 
\end{theorem}

%

\begin{remark} \label{rem-SIS}
The analysis can be easily extended to the multi-patch SIS model, where the population  in each patch has susceptible and infectious groups, and when infectious individuals recover, they become susceptible immediately. 
The epidemic evolution dynamics is described as
\begin{align*}
S^N_i(t)&\;=\;S^N_i(0)-A^N_i(t)  + \sum_{\ell=1}^L \sum_{k=1}^{I^N_\ell(0)} {\bf1}_{\eta_{k,\ell}^0\le t}{\bf1}_{X^{0,k}_\ell(\eta_{k,\ell}^0)=i}
+\sum_{\ell=1}^L \sum_{j=1}^{A^N_\ell(t)}  {\bf 1}_{\tau^N_{j,\ell} + \eta_{j,\ell} \le t} {\bf1}_{X^j_\ell( \eta_{j,\ell})=i} \non\\
& \quad - \sum_{\ell=1,\ell \neq i}^L P_{S,i,\ell}\left(\nu_{S,i,\ell}\int_0^t S^N_i (s)ds\right)
+\sum_{\ell=1,\ell\neq i}^L P_{S,\ell,i}\left( \nu_{S,\ell,i} \int_0^tS^N_i(s)ds\right) \,,\\ 
I^N_i(t)&\;=\; \sum_{\ell=1}^L \sum_{k=1}^{I^N_\ell(0)} {\bf1}_{t<\eta_{k,\ell}^0}{\bf1}_{X^{0,k}_\ell(t)=i}
+ \sum_{\ell=1}^L \sum_{j=1}^{A^N_\ell(t)} {\bf 1}_{\tau^N_{j,\ell} + \eta_{j,\ell} >t} {\bf1}_{X^j_\ell(t-\tau^N_{j,\ell})=i} \,\,, 
\end{align*}
where $A^n_i$ is given as in \eqref{An-rep-1} with $
\Upsilon^N_i(t)=\frac{ \lambda_i S^N_i(t)\sum_{j=1}^L\kappa_{ij}I^N_j(t)}{(S^N_i(t)+I^N_i(t))^\gamma}$, for $ i =1,\dots, L.$
Thus, in the FLLN, we obtain the same limit $\bar{I}_i$ in \eqref{eqn-barIi} as in the multi-patch SIR model, and the limit $\bar{S}_i(t)$:
 \begin{align}
 \S_i(t) &\;=\; \S_i(0) - \int_0^t \bar\Upsilon_i(s) ds  \int_0^t   \sum_{\ell} p_{\ell,i}(s) F_0(ds)  +  \int_0^t  \sum_{\ell}  \left(  \int_0^{t-s} p_{\ell,i}(u) F(du)  \right)   \bar\Upsilon_\ell(s) ds \non\\  
 & \qquad +   \sum_{\ell=1,\ell \neq i}^L\int_0^t \left( \nu_{S,\ell,i}  \S_j(s)-  \nu_{S,i,\ell}\S_i(s)  \right) ds\,,  \non
 \end{align}
 where $
\bar\Upsilon_i(t):=\frac{ \lambda_i \S_i(t)\sum_{j=1}\kappa_{ij}\I_j(t)}{(\S_i(t)+\I_i(t))^\gamma}. $ Similarly in the FCLT, we obtain  the same limit $\hat{I}_i$ as in \eqref{eqn-hatI-1} for the multi-patch SIR model, and the limit $\hat{S}_i(t)$: 
 \begin{align} 
\hat{S}_i(t) &\;=\; \hat{S}_i(0) -   \int_0^t \hat{\Upsilon}_i(s) ds   + \sum_{\ell=1}^L \int_0^t \int_0^{t-s} p_{\ell,i}(u) F(du)  \hat\Upsilon_\ell(s) ds  + \sum_{\ell=1}^L \big( \hat{I}^{0}_{\ell,i}(t) + \hat{I}_{\ell,i}(t) \big)  \non\\
& \qquad + \sum_{\ell=1, \ell\neq i}^L \int_0^t (\nu_{S,\ell,i}\hat{S}_\ell(s)- \nu_{S,i, \ell}\hat{S}_i(s)) ds - \hat{M}_{A,i}(t)  +\sum_{\ell=1, \ell\neq i}^L \big( \hat{M}_{S,\ell,i}(t) - \hat{M}_{S,i,\ell}(t) \big)\,, \non 
\end{align}
where 
\begin{align*}\hat{\Upsilon}_i(t) &=\frac{\lambda_i}{(\bar{S}_i(t)+\bar{I}_i(t))^{(1+\gamma)}} 
\left\{[(1-\gamma)\bar{S}_i(t)+\bar{I}_i(t)]\bar{I}_{(i)}(t) \hat{S}_i(t) + [\bar{S}_i(t)(\bar{S}_i(t)+\bar{I}_i(t))-\bar{S}_i(t) \hat{I}_{(i)}(t)]\hat{I}_i(t)\right\} \\
&\qquad +\frac{\lambda_i \bar{S}_i(t)\sum_{j\not=i}\hat{I}_j(t) }{ (\S_i(t)+\I_i(t))^\gamma}.
\end{align*} 
\end{remark}

\section{Models with varying infectivity and limiting integral equations models}\label{sec:varinfect}
\subsection{Stochastic model with varying infectivity, LLN and CLT}\label{sec:varinfect-eqs}
In this section, we shall consider the same model as in the original work of Kermack and McKendrick \cite{KMK}, except that 
we shall formulate a continuous time stochastic individual based model, which as the size $N$ of the population tends to $\infty$, converges to their model (but our model is slightly more general, since we do not assume that the law of the infectious period  is absolutely continuous).

As usual, the population consists of three groups of individuals, susceptible, infected and recovered. 
Let $N$ be the population size, and $S^N(t), I^N(t), R^N(t)$ denote the sizes of the three groups, respectively. 
We have the balance equation $N= S^N(t) + I^N(t) + R^N(t)$ for $t\ge 0$. Assume that $S^N(0)>0$, $I^N(0) >0$ and $R^N(0)$ are such that $S^N(0) +I^N(0)+R^N(0)=N$. Infections occur through interactions of infected individuals with the susceptibles, as in the standard models. 
	
Each initially infected individual is associated with an infectivity process $\lambda^0_j(t)$, $j=1,\dots, I^N(0)$, which are assumed to be i.i.d.  
Each newly infected individual is associated with an   infectivity process $\lambda_i(t)$, $i\in \NN$, which are also assumed to be i.i.d. We assume moreover that $(S^N(0),I^N(0),R^N(0))$, $\{\lambda^0_j\}_{j\ge1}$ and $\{\lambda_i\}_{i\ge1}$ are mutually independent. Assume that $\lambda^0_j(0) =0$ and $\lambda_i(0)=0$ with probability one. 
These processes are only taking effect during the infectious periods. 
Define
\begin{align*} 
\eta^0_j &:= \sup\{t>0: \lambda^0_j(t) > 0\}, \quad \forall \, j=1=1,\dots, I^N(0), \\
\eta_i &:= \sup\{t>0: \lambda_i(t) > 0\}, \quad \forall\, i\ge 1. 
\end{align*}
By the i.i.d. assumption of $\{\lambda^0_j(t)\}$, the variables $\eta^0_j, \, j=1,\dots, I^N(0)$, are also i.i.d., representing the remaining infectious durations of the initially infected individuals. 
Similarly,  $\eta_i,\, i \in \NN$  are i.i.d., also independent of $\{\lambda^0_j(t)\}$, and represent the  infectious periods of the newly infected individuals. 
Let $F_0$ and $F$ be the c.d.f.'s of the variables  $\eta^0_j$ and $\eta_i$, respectively, $F^c_0=1-F_0$ and $F^c=1-F$. 

The total force of infection which is exerted on the susceptibles at time $t$ can be written as 
\begin{align} \label{eqn-mfI}
\mathfrak{F}^N(t) = \sum_{j=1}^{I^N(0)}\lambda^0_j(t) 
+\sum_{i=1}^{A^N(t)}\lambda_i (t-\tau^N_i) \,, \quad t \ge 0.
\end{align}
	Thus, the instantaneous infectivity rate function at time $t$ is
\begin{align} \label{eqn-Upsilon-VI}
\Upsilon^N(t) =\mathfrak{F}^N(t) \times  \frac{S^N(t)}{N} , \quad t\ge 0. 
\end{align}
	Observe that in comparison with the $\Upsilon^N(t)$ in \eqref{Upsilon} of the standard model, we have replaced $\lambda I^N(t)$ by the total force of infection $\mathfrak{F}^N(t) $ in the generalized model. It is clear that  the standard SIR model a the particular case of the present model, where 
$\lambda(t)=\lambda{\bf1}_{t<\eta}$, $\eta$ being the random duration of the infectious period. 
	The cumulative infection process $A^N(t)$ is expressed exactly as in \eqref{eq:AN} , using the instantaneous infectivity rate function $\Upsilon^N(t)$ in \eqref{eqn-Upsilon-VI}.  
	
	The epidemic dynamics of the model can be described in the same way as the standard SIR models, in equations \eqref{eqn-SIR-balance}, \eqref{eqn-SIR-In} and \eqref{eqn-SIR-Rn}.
	

\begin{remark}{\bf The SEIR model}.
Suppose that $\lambda_i(t)=0$ for $t\in[0,\xi_i)$, where $\xi_i<\eta_i$, and denote $I$ as the compartment of infected (not necessarily infectious) individuals. An individual who gets infected at time $\tau^N_i$ is first exposed during the time interval $[\tau^N_i,\tau^N_i+\xi_i)$, and then infectious during the time interval
$(\tau^N_i+\xi_i,\tau^N_i+\eta_i)$. The individual is infected during the time interval $[\tau^N_i,\tau^N_i+\eta_i)$.
At time $\tau^N_i+\eta_i$, he recovers. All what follows covers perfectly this situation. In other words, our model acomodates perfectly an exposed period before the infectious period, which is important for many infectious diseases, including the Covid--19. However, we distinguish only three compartments, $\bS$ for susceptible, $\bI$ for infected (either exposed or infectious), $\bR$ for recovered. Note that we could also describe the evolution of the numbers of individuals in the four compartments  $\bS$, $\bE$, $\bI$ and $\bR$ as it is done in \cite{FPP2020b}.
\end{remark}

We make the following assumptions on $\lambda^0$ and $\lambda$.

\begin{assumption} \label{AS-lambda}
The random functions $\lambda(t)$ (resp. $\lambda^0(t)$), of which $\lambda_1(t), \lambda_2(t),\ldots$ (resp. $\lambda^0_1(t), \lambda^0_2(t), \\ \ldots$) are i.i.d. copies, satisfy the following assumptions.  
There exists a constant $\lambda^*<\infty$ such that $\sup_{t \in [0,T]} \max\{\lambda^0(t),  \lambda(t)\} \le \lambda^*$ almost surely, and in addition there exist a given number $k\ge1$, a random sequence $0=\xi^0\le\xi^1\le\cdots\le\xi^k=\eta$ and random functions $\lambda^j\in \bC$, $1\le j\le k$
such that 
\begin{equation} \label{eqn-lambda}
 \lambda(t)=\sum_{j=1}^k\lambda^j(t){\bf1}_{[\xi^{j-1},\xi^j)}(t)\,.
\end{equation}
We assume that for any $T>0$, there exists $\varphi_T\in \bC$ such that $\varphi_T(0)=0$ and for any $0\le s< t\le T$,
$\sup_{1\le j\le k}|\lambda^j(t)-\lambda^j(s)|\le\varphi_T(t-s)$.
\end{assumption}
Let $\bar{\lambda}^0(t) =\E[\lambda^0(t)]$ and $\bar{\lambda}(t) =\E[\lambda(t)]$ for $t\ge 0$. Also, let $v_{0}(t) = \Var(\lambda^0(t))$ and   $v(t) = \Var(\lambda(t))$ for $t\ge 0$. 

\begin{remark}
Recall that the basic reproduction number $R_0$ is the mean number of susceptible individuals whom an infectious individual infects in a large population otherwise fully susceptible. In this model, clearly
\[ R_0=\int_0^\infty \bar\lambda(t)dt\,.\]
 Suppose that $\lambda_i(t)=\tilde{\lambda}(t)\bone_{t<\eta_i}$, where $\tilde{\lambda}(t)$ is a deterministic function. Then
 \[ R_0=\int_0^\infty \tilde{\lambda}(t)F^c(t)dt\,.\]
In the standard SIR model with $\tilde{\lambda}(t) \equiv \lambda$ and $\E[\eta] = \int_0^\infty F^c(t)dt$, the formula above reduces to the well known $R_0 = \lambda \E[\eta]$. See, e.g., \cite{brittonpardoux}. 
\end{remark}

\begin{theorem} \label{thm-FLLN-VI}
 Under Assumptions ~\ref{AS-SIR-1} and~\ref{AS-lambda}, 
\begin{equation*} 
\big(\bar{S}^N, \bar{\mathfrak{F}}^N, \bar{I}^N, \bar{R}^N\big) 
\to \big(\bar{S}, \bar{\mathfrak{F}}, \bar{I}, \bar{R}\big) \qinq \bD^4 \qasq N \to \infty,
\end{equation*}
in probability, locally uniformly in $t$.  
The limits $\bar{S}$ and $\bar{\mathfrak{F}}(t)$ are the unique solution of the  system of Volterra integral equations
consisting of \eqref{SIR-barS}  and
\begin{align}
\bar{\mathfrak{F}}(t)&=\bar{I}(0)\bar{\lambda}^0(t)+\int_0^t\bar{\lambda}(t-s) \bar\Upsilon(s)ds\,,\label{eqn-barfrakI}
\end{align}
with
\begin{equation} \label{eqn-barUpsilon-VI}
\bar\Upsilon(t)=\bar{S}(t)\bar{\mathfrak{F}}(t)\,.
\end{equation}
Now the limits $\bar{I}$ and $\bar{R}$ are given by the formulas \eqref{SIR-barI}--\eqref{SIR-barR}.

\end{theorem}
\begin{remark}{\bf Comparison with the Kermack--McKendrick model}
If we assume that $F_0\equiv F$, then the last system of equations is exactly the system of equations (12), (13), (15) and (14) on page 704 of \cite{KMK}. Indeed, it follows from the computation at the start of Section~\ref{sec:SIRnonM} that the function $B_t$ of \cite{KMK} is our $F^c(t)$, while their $C_t$ is our $f(t)$. Moreover
their $A_t$ is our $\bar{\lambda}(t)$ (indeed, one can think of our $\lambda(t)$ as being the product of a deterministic function of $t$ (their $\phi_t$) multiplied by ${\bf1}_{\eta> t}$, so that $\bar{\lambda}(t)=\phi_t F^c(t)$.
\end{remark}

We now sketch the proof of Theorem~\ref{thm-FLLN-VI}.
\begin{proof}
Thanks to Assumption~\ref{AS-lambda}, it is clear that $\bar{\Upsilon}^N(t)\le\lambda^\ast$. Hence the first step of the proof of Theorem~\ref{thm-FLLN-SIR} remains valid here, \textit{i.e.,} we have the convergence, along a subsequence, of $(\bar{S}^N,\bar{A}^N)$.  We next consider the sequence 
\[\bar{\mathfrak{F}}^N(t)=\sum_{j=1}^{I^N(0)}\lambda^0_j(t) 
+\sum_{i=1}^{A^N(t)}\lambda_i (t-\tau^N_i)=\bar{\mathfrak{F}}_0^N(t)+\bar{\mathfrak{F}}_1^N(t)\,.\]
Concerning the first term $\bar{\mathfrak{F}}_0^N$, as in the proof of Theorem~\ref{thm-FLLN-SIR}, we first consider
\[ \breve{\mathfrak{F}}_0^N(t):=\sum_{j=1}^{N\bar{I}(0)}\lambda^0_j(t),\]
which converges thanks to a LLN for random elements in $\bD$, see Theorem 1 in \cite{rao1963law}. The difference
$\bar{\mathfrak{F}}_0^N(t)-\breve{\mathfrak{F}}_0^N(t)$ is treated as in the proof of Theorem~\ref{thm-FLLN-SIR}.
Concerning the term $\bar{\mathfrak{F}}_1^N$, we first consider
\[ \breve{\mathfrak{F}}_1^N(t):=N^{-1}\sum_{i=1}^{A^N(t)}\bar{\lambda}(t-\tau^N_i)=\int_0^t \bar{\lambda}(t-s)d\bar{A}^N(s)\,.\]
The argument for the weak convergence of that sequence towards 
$\bar{\mathfrak{F}}(t)=\int_0^t \bar{\lambda}(t-s)d\bar{A}(s)$, along any subsequence along which 
$\bar{A}^N\Rightarrow \bar{A}$ is similar to a similar result in the proof of Theorem~\ref{thm-FLLN-SIR}, 
with slightly more tricky arguments. For the details, as well as for the proof of the fact that
$\bar{\mathfrak{F}}_1^N(t)-\breve{\mathfrak{F}}_1^N(t)\to0$, we refer to Section 4 of \cite{FPP2020b}. 
It remains to prove that $(\bar{I}^N(t),\bar{R}^N(t))\Rightarrow(\bar{I}(t),\bar{R}(t))$, which requires similar arguments as in the first steps of the proof. Finally, one can show that the limiting equation has a unique deterministic solution, hence the whole sequence converges, and the convergence is in probability.
\end{proof}

For the FCLT, we need the following additional conditions on the random infectivity functions. 
\begin{assumption} \label{AS-lambda-CLT}
In addition to the conditions in Assumption~\ref{AS-lambda}, the random functions $\lambda(t)$ (resp. $\lambda^0(t)$) satisfy the following conditions.

\begin{itemize}

\item[(i)] There exist nondecreasing functions $\phi$ and $\psi$ in $\bC$ and $\alpha>1/2$ and $\beta>1$ such that for all $0 \le r\le s \le t $, denoting $\breve{\lambda}^0(t) = \lambda^0(t) - \bar{\lambda}^0(t)$, 
\begin{align*}
(a) \quad &  \E\big[ \big(\breve{\lambda}^0(t)  -\breve{\lambda}^0(s)\big)^2 \big] \le (\phi(t) - \phi(s) )^\alpha\,,\\
(b) \quad & \E\big[ \big(\breve{\lambda}^0(t)  -\breve{\lambda}^0(s)\big)^2 \big(\breve{\lambda}^0(s)  -\breve{\lambda}^0(r)\big)^2 \big] \le (\psi(t) - \psi(r) )^\beta. 
\end{align*}

\item[(ii)] Either $\lambda\in\bC$ and satisfies \eqref{varphiHolder}--\eqref{eqn-lambda-inc} below,  or else it satisfies \eqref{eqn-lambda} and the additional conditions below.
There exists a nondecreasing function 
$\varphi\in \bC$ satisfying 
\begin{equation}\label{varphiHolder}
\varphi(r)\le Cr^\alpha, \ \text{ with }\alpha>1/2\ \text{ and } C>0 \text{ arbitrary},
\end{equation} such that 
\begin{equation} \label{eqn-lambda-inc}
|\lambda^j(t)-\lambda^j(s)|\le\varphi(|t-s|), \quad \text{a.s.,}
\end{equation}
 for all $t,s\ge0$, $1\le j\le k$. 
Also, if $F_j$ denotes the c.d.f. of the r.v. $\xi_j$, then the exist $C'$ and  
 $\rho>0$ such that for any $0\le j\le k$, $0\le s<t$,
 \begin{equation*}
 F_j(t)-F_j(s)\le C'(t-s)^\rho\,,
 \end{equation*}
 and in addition, for any $1\le j\le k$, $r>0$,
 \begin{equation*}
 \P(\xi^{j}-\xi^{j-1}\le r|\xi^{j-1})\le C' r^\rho\,.
 \end{equation*}
\end{itemize}
 \end{assumption}

\begin{remark}
A simple example of a random function $\lambda^0$ (resp. $\lambda$) which satisfies the above condition (i)
(resp. (ii)) is as follows.
We take $\lambda^0$ (or $\lambda$) continuous and piecewise linear. It can first be $0$ for a random duration, then it starts from $0$ with a random positive slope, and finally decreases 
to zero with a random negative slope, after which it stays equal to $0$, both slopes being bounded in absolute value by a fixed constant. Not that with such a choice, $\lambda^0$ (resp. $\lambda$) is specified by a small number of parameters.
\end{remark}  
 
 \begin{theorem} \label{thm-FCLT-VI}
Under Assumptions~\ref{AS-SIR-2}, ~\ref{AS-lambda} and~\ref{AS-lambda-CLT}, 
\begin{equation*} 
\big(\hat{S}^N, \hat{\mathfrak{F}}^N, \hat{I}^N, \hat{R}^N\big) 
\RA \big(\hat{S}, \hat{\mathfrak{F}}, \hat{I}, \hat{R}\big) \qinq \bD^4 \qasq N \to \infty\,. 
\end{equation*}
The limit process $(\hat{S},\hat{\mathfrak{F}})$ is the unique solution of the following system of stochastic integral equations: 
\begin{align}
\hat{S}(t) &=  \hat{S}(0) - \hat{M}_A(t) + \int_0^t\hat{\Upsilon}(s)ds,  \label{eqn-hatS}\\
\hat{\mathfrak{F}}(t) &= \hat{I}(0) \bar{\lambda}^0(t)+  \hat{\mathfrak{F}}_{0}(t)   + \hat{\mathfrak{F}}_1(t) + \hat{\mathfrak{F}}_2(t) + \int_0^t \bar{\lambda}(t-s) \hat{\Upsilon}(s) ds,   \label{eqn-frakI}
\end{align}
where
\begin{align} \label{def-hatUpsilon}
\hat{\Upsilon}(t)  =  \hat{S}(t)  \bar{\mathfrak{F}}(t) + \bar{S}(t)  \hat{\mathfrak{F}}(t),
\end{align}
and $\bar{S}(t)$ and $ \bar{\mathfrak{F}}(t)$ are given in Theorem \ref{thm-FLLN-VI},
$\hat{M}_A$, $\hat{\mathfrak{F}}_{0}$, $\hat{\mathfrak{F}}_1$ and $\hat{\mathfrak{F}}_2$  are centered Gaussian processes which are  globally independent of $(\hat{E}(0),\hat{I}(0))$. Moreover, the processes in 
$(\hat{\mathfrak{F}}_{0}, \hat{\mathfrak{F}}_1, (\hat{\mathfrak{F}}_2, \hat{M}_A))$ are independent,  and the covariances of each of those four processes (the last one being $2$--dimensional) are given as follows:
 \begin{align*}
\Cov( \hat{\mathfrak{F}}_{0}(t),  \hat{\mathfrak{F}}_{0}(t'))
&= \bar{I}(0) \Cov(\lambda^0(t), \lambda^0(t')),\\
\Cov( \hat{\mathfrak{F}}_1(t),  \hat{\mathfrak{F}}_1(t'))
&= \int_0^{t\wedge t'} \Cov(\lambda(t-s), \lambda(t'-s)) \bar\Upsilon(s)  ds,\\
\Cov( \hat{\mathfrak{F}}_2(t),  \hat{\mathfrak{F}}_2(t')) &=  \int_0^{t\wedge t'}  \bar{\lambda}(t-s) \bar{\lambda}(t'-s)  \bar\Upsilon(s)  ds,\\
\Cov(\hat{M}_A(t),\hat{M}_A(t'))&=\int_0^{t\wedge t'}\bar\Upsilon(s)ds,\\
\Cov(\hat{M}_A(t),\hat{\mathfrak{F}}_2(t')) &=\int_0^{t\wedge t'} \bar{\lambda}(t'-s)  \bar\Upsilon(s)  ds\,.
\end{align*}
Concerning the pair $(\hat{M}_A,\hat{\mathfrak{F}}_2)$, $\hat{M}_A$ is a non--standard Brownian motion,
and $\hat{\mathfrak{F}}_2(t)=\int_0^t\bar\lambda(t-s)\hat{M}_A(ds)$.
$\hat{S}$ has continuous paths, and 
if $\bar{\lambda}^0$ and $\bar{\lambda}^{0,I}$ are in $\bC$, then $\hat{\mathfrak{F}}$ is also continuous. 

The limits $(\hat{I}, \hat{R})$ are the same as given in Theorem \ref{thm-FCLT-SIR}, with $\hat\Upsilon(t)$ being replaced by the expression in \eqref{def-hatUpsilon}. 
In addition, for $t, t'\ge 0$, 
\begin{align*}
\Cov(\hat{\mathfrak{F}}_{0}(t), \hat{I}_0(t')) 
&=  \bar{I}(0) \big(\E \big[ \lambda^{0}(t) \bone_{\eta^{0}> t'} \big] -\bar\lambda^{0}(t) F^c_{0}(t')   \big) \,, \\
\Cov(\hat{\mathfrak{F}}_{0}(t), \hat{R}_0(t')) 
&=  \bar{I}(0) \big(\E \big[ \lambda^{0}(t) \bone_{\eta^{0}\le t'} \big] -\bar\lambda^{0}(t) F_{0}(t')   \big) \,, \\
\Cov(\hat{\mathfrak{F}}_{1}(t),  \hat{I}_1(t')) &= \int_0^{t \wedge t'}  \Big( \E \big[\lambda(t-s) \bone_{\eta>t'-s}\big] - \bar\lambda(t-s) F^c(t'-s) \Big) \bar\Upsilon(s) ds\,, \\
\Cov(\hat{\mathfrak{F}}_{1}(t), \hat{R}_1(t')) &= \int_0^{t \wedge t'}  \Big( \E \big[\lambda(t-s) \bone_{\eta\le t' -s}\big] - \bar\lambda(t-s) F(t'-s) \Big) \bar\Upsilon(s) ds\,,\\
\Cov(\hat{\mathfrak{F}}_{2}(t),  \hat{I}_1(t')) &=  \int_0^{t \wedge t'}  \bar\lambda(t-s) F^c(t'-s) \bar\Upsilon(s) ds\,, \\
\Cov(\hat{\mathfrak{F}}_{2}(t), \hat{R}_1(t')) &=  \int_0^{t \wedge t'}  \bar\lambda(t-s) F(t'-s)  \bar\Upsilon(s) ds\,. 
\end{align*}

\end{theorem}

\noindent{\bf An alternative initial condition.}
In the above formulation, we have assumed that for the initially infected individuals, their infectivity functions $\lambda^0_j(\cdot)$ are i.i.d., and may follow a different law from those of the newly infected individuals $\lambda_i(\cdot)$, hence, the distribution $F_0$ of the remaining infected periods $\eta^0_j$ generated from $\lambda^0_j(\cdot)$, is different from $F$ of the infected periods $\eta_i$ generated from $\lambda_i(\cdot)$.  
However, we can assume that the random infectivity functions of all individuals, $\{\lambda^0_j(\cdot)\}_j$ and $\{\lambda_i(\cdot)\}_i$ are all i.i.d., while for the initially infected individuals, the time epochs of them becoming infected before time 0 are known, $\tau^N_{j,0}$, $j=1,\dots, I^N(0)$. 
Then $\tilde\tau^N_{j,0} = - \tau^N_{j,0}$, is the elapsed time at time 0 since infection. Set  $\tilde\tau^N_{0,0}=0$, and let $I^N(0,x) = \max\{j \ge 0: \tilde\tau^N_{j,0} \le x\}$. Assume that there exists $\bar{x} \in \RR_+$, such that $I^N(0) = I^N(0,\bar{x})$. Let $\bar\lambda(t) = \E[\lambda^0_j(t)] = \E[\lambda_i(t)]$ for $t\ge 0$. 

The remaining infected period is given by $\eta^0_j=\inf\{t>0: \lambda^0_j(\tilde\tau^N_{j,0} +r) =0, \forall r\ge t\}$. It depends on the elapsed infection time $\tilde\tau^N_{j,0}$, and independent from the remaining infected durations of the other individuals due to the i.i.d. assumption of $\{\lambda^0_j(\cdot)\}_j$.  Given that $\tilde\tau^N_{j,0}=s>0$, the distribution of $\eta^0_j$ is given as in \eqref{eqn-eta0-age}.

Instead of \eqref{eqn-mfI}, the total force of infectivity at time $t$ can be written as 
\begin{align*} 
\mathfrak{F}^N(t) = \sum_{j=1}^{I^N(0)}\lambda^0_j(\tilde{\tau}_{j,0}^N+t)\bone_{\tilde{\tau}_{j,0}^N \le \bar{x}} 
+\sum_{i=1}^{A^N(t)}\lambda_i (t-\tau^N_i) \,, \quad t \ge 0. 
\end{align*}
All the other processes have the same representations. Recall that the process $I^N(t)$ is given as in \eqref{eqn-SIR-In-I0x} with the variables $\eta^0_j$ implicitly depending on $\tilde{\tau}_{j,0}^N$. 

Under Assumption \ref{AS-SIR-initial2-LLN}, we can show that the FLLN holds with  $\bar{S}(t)$ in \eqref{SIR-barS}, and the limit $\bar{\mathfrak{F}}(t)$ is given as
\begin{align}\label{eqn-barfrakI-initial2}
\bar{\mathfrak{F}}(t)&=\int_0^{\bar{x}} \bar{\lambda}(y+t) \bar{I}(0,dy) +\int_0^t\bar{\lambda}(t-s)\bar\Upsilon(s)ds\,,
\end{align}
and the limits $\bar{I}$ and $\bar{R}$ are given by the same expressions in \eqref{SIR-barI-I0x} and \eqref{SIR-barR-I0x} with $\bar\Upsilon(t)$ in \eqref{eqn-barUpsilon-VI}. 

Under Assumption \ref{AS-SIR-initial2-CLT}, we can show that the FCLT holds with the limit $\hat{S}(t)$ given by \eqref{eqn-hatS} and the limit $\hat{\mathfrak{F}}(t)$ given by
 \begin{align*}
\hat{\mathfrak{F}}(t) &=\int_0^{\bar{x}}\bar{\lambda}(y+t) d \hat{I}(0,y)  +  \hat{\mathfrak{F}}_{0}(t)   + \hat{\mathfrak{F}}_1(t) + \hat{\mathfrak{F}}_2(t) + \int_0^t \bar{\lambda}(t-s) \hat{\Upsilon}(s) ds,  
\end{align*}
where $\hat{\Upsilon}(t)$ is given in \eqref{def-hatUpsilon}, 
and $\bar{S}(t)$ and $ \bar{\mathfrak{F}}(t)$ are given as above,
$\hat{\mathfrak{F}}_{0}$ is a continuous Gaussian process with mean zero and covariance function: for $t, t'\ge0$, 
 \begin{align*}
\Cov( \hat{\mathfrak{F}}_{0}(t),  \hat{\mathfrak{F}}_{0}(t'))
&= \int_0^{\bar{x}} \Cov(\lambda(y+t), \lambda(y+t'))\bar{I}(0,dy),
\end{align*}
and the other limits $\hat{M}_A$, $\hat{\mathfrak{F}}_1$ and $\hat{\mathfrak{F}}_2$  are centered Gaussian processes as given in Theorem \ref{thm-FCLT-VI}. 

The limits $(\hat{I}, \hat{R})$ are  given by the same expressions in \eqref{SIR-Ihat-I0x} and \eqref{SIR-Rhat-I0x} with $\hat\Upsilon(t)$ in \eqref{def-hatUpsilon}.
In addition, 
$\hat{\mathfrak{F}}_{0}$ and $\hat{I}_{0}(t)$, $\hat{R}_{0}(t)$ have covariance functions: for $t, t'\ge 0$,
\begin{align*}
\Cov(\hat{\mathfrak{F}}_{0}(t), \hat{I}_0(t')) 
&= \int_0^{\bar{x}} \E[ \lambda(y+t) \bone_{\eta^0|\tau_0=y >t'}]   \bar{I}(0,dy) - \int_0^{\bar{x}} \bar\lambda(y+t) \bar{I}(0,dy)   \int_0^{\bar{x}} \frac{F^c(t'+y)}{F^c(y)}  \bar{I}(0, dy)\,, \\
\Cov(\hat{\mathfrak{F}}_{0}(t), \hat{R}_0(t')) 
&=  \int_0^{\bar{x}} \E[ \lambda(y+t) \bone_{\eta^0|\tau_0=y \le t'}]   \bar{I}(0,dy)  \\
& \qquad - \int_0^{\bar{x}} \bar\lambda(y+t) \bar{I}(0,dy)   \int_0^{\bar{x}}\left(1- \frac{F^c(t'+y)}{F^c(y)} \right)  \bar{I}(0, dy)\,. 
\end{align*}

\subsection{The early phase of the epidemic}
In this subsection we follow again \cite{FPP2020b}, to which we refer the reader for the proofs.  
 Theorem~\ref{thm-FLLN-VI} shows that the deterministic system of equations \eqref{SIR-barS}-\eqref{eqn-barR} accurately describes the evolution of the stochastic process defined in the previous subsection when the initial number of infectious individuals is of the order of $ N $.
But epidemics typically start with only a handful of infectious individuals, and it takes some time before the epidemic enters the regime of Theorem~\ref{thm-FLLN-VI}.
Exactly how long this takes depends on the population size $ N $ and on the growth rate of the epidemic.
To determine this growth rate, we study the behavior of the stochastic process when the initial number of infectious individuals is kept fixed as $ N \to \infty $.

Recall that $R_0=\int_0^\infty \bar{\lambda}(t)dt$, and let $ \rho \in \R $ be the unique solution of
\begin{align} \label{def_rho}
\int_{0}^{\infty} \overline{\lambda}(t) e^{-\rho t} dt = 1.
\end{align}
If $R_0\le1$, the total number of infected individuals remains small as $N\to\infty$, while if $R_0>1$ (which we assume in what follows), with positive probability a major outbreak takes place, \textit{i.e.,} a positive fraction of the $N$ individuals is infected at some point during the course of the epidemic. 
It is well--known, see e.g. Section 1.2 in \cite{brittonpardoux} and also \cite{crump1968general, crump1969general} and the book  \cite{jagers1975branching}, that during its early stage, an epidemic can be well approximated by a continuous--time  branching process, often called Crump-Mode-Jagers branching process. Indeed, each infectious infects individuals in the population, and as long as almost  all the individuals in the population are susceptible, the probability that two distinct infectious individuals try to infect the same susceptible is close to $0$. As a result the ``progenies'' of the various infectious individuals are essentially independent, thus the branching property. Of course, that approximation breaks down as soon as a significant number of individuals have been hit by the disease. Using an approximation of the early phase by a (in our case non--Markov) branching process,
it has been shown in \cite{FPP2020b} that on the event that a major outbreak takes place, for any 
$\varepsilon<1-R_0^{-1}$, if $T^N_\varepsilon$ denotes the first time at which the proportion of infected individuals is at least $\varepsilon$, as $N\to\infty$,  $T^N_\varepsilon=\frac{1}{\rho}\log(N)+\mathcal{O}(1)$, which means an exponential growth with rate $\rho$. 

Next one can show that, still at the start of the epidemic, our LLN deterministic model also grows at the same rate 
$\rho$. More precisely, if we assume that we can replace $\bar{S}(t)$ by $1$, the LLN model becomes (after remultiplication by $N$) the following linear system:
\begin{align*}
{\mathfrak{F}}(t) &= {I}(0)  \bar{\lambda}^0(t) + \int_0^t \bar{\lambda}(t-s) {\mathfrak{F}}(s)ds \,,\\
{I}(t) &= {I}(0) F_{0}^c(t)   +\int_0^t F^c(t-s) {\mathfrak{F}}(s)ds\,,\\
{R}(t)&= R(0)+  {I}(0) F_{0}(t)   +\int_0^t F(t-s) {\mathfrak{F}}(s)ds\, . 
\end{align*}

In the next statement, which is part of Theorem 2.13 in \cite{FPP2020b}, $\rho$ is specified by \eqref{def_rho}.
\begin{theorem} \label{thm:early_phase_deterministic}
We assume that Assumption~\ref{AS-lambda} is valid, and that $R_0>1$, hence $\rho>0$.
	Define
	\begin{align*} 
	\bm{i} := \int_{0}^{\infty} F^c(s) \rho e^{-\rho s} ds, && \bm{r} := 1 - \bm{i},
	\end{align*}
	and
	\begin{align*}
	\overline{\lambda}_\rho(t) := \frac{\int_{0}^{\infty} \overline{\lambda}(t+s) e^{-\rho s} ds}{\int_{0}^{\infty} F^c(s) e^{-\rho s} ds}, && F_\rho^c(t) := \frac{\int_{0}^{\infty} F^c(t+s) e^{-\rho s} ds}{\int_{0}^{\infty} F^c(s) e^{-\rho s} ds}.
	\end{align*}
	If $ \overline{\lambda}^0 = \overline{\lambda}_\rho $ and $ F_0 = F_\rho $, the above linear system  admits the following solution
	\begin{align*} 
	\mathfrak{F}(t)=\rho\, e^{\rho t}, \quad I(t)=\bm{i}\, e^{\rho t}, \quad R(t) = \bm{r}\, e^{\rho t}\, \quad t \geq 0.
	\end{align*}
\end{theorem}

Let us suppose that $ \overline{\lambda} $ is only known up to a constant factor $ \mu > 0 $, \textit{\textit{i.e.}}, 
\begin{align*}
\overline{\lambda}(t) = \mu\, \overline{g}(t), \quad t \geq 0,
\end{align*}
where $ \mu $ is unknown but $ \overline{g} $ is known (for example from medical data on viral shedding).
We now assume w.l.o.g. that $ \overline{g} $ has been normalized in such a way that $\int_0^\infty\overline{g}(t)dt=1$.
We can then estimate $ \mu $ (and $ R_0 $) from the growth rate $ \rho $, which can be measured easily at the beginning of the epidemic ($ \rho = \log(2) / d $, where $ d $ is the doubling time of the daily number of newly infected individuals), using the relation \eqref{def_rho}.
The following is thus a corollary of Theorem~\ref{thm:early_phase_deterministic}.

\begin{coro} \label{coro-R0}
	Let $ \rho $ be the growth rate of the number of infected individuals.
	Then
	\begin{align}\label{eqn-R0}
	R_0=\mu = \left( \int_{0}^{\infty} \overline{g}(t) e^{-\rho t} ds \right)^{-1}\,.
	\end{align}
\end{coro}
Note that $R_0$ can be thought of as the growth rate of the epidemic from one generation to the next, while $\rho$ is the growth rate of the epidemic in real time. The formula \eqref{eqn-R0} is formula (2.7) in \cite{WallingaLipsitch}. 

\subsection{Application to the Covid--19 epidemic}

We now  explain how the type of model described in this section can be used to model the Covid--19 epidemic. As we have seen, the increase in realism with respect to the classical ``Markovian'' models (where the infectivity is constant and fixed across the population, and the Exposed and Infectious periods follow an exponential distribution) is paid by replacing a system of ODEs by  a system of Volterra integral equations. However, we have a small benefit in that the flexibility induced by the fact that the law of $\lambda$ is arbitrary allows us to reduce the number of compartments in the model, so that we can replace a system of ODEs 
by a system of Volterra type equations of smaller dimension. 

\begin{figure}[ht]
	\centering
	\includegraphics[width=.6\linewidth]{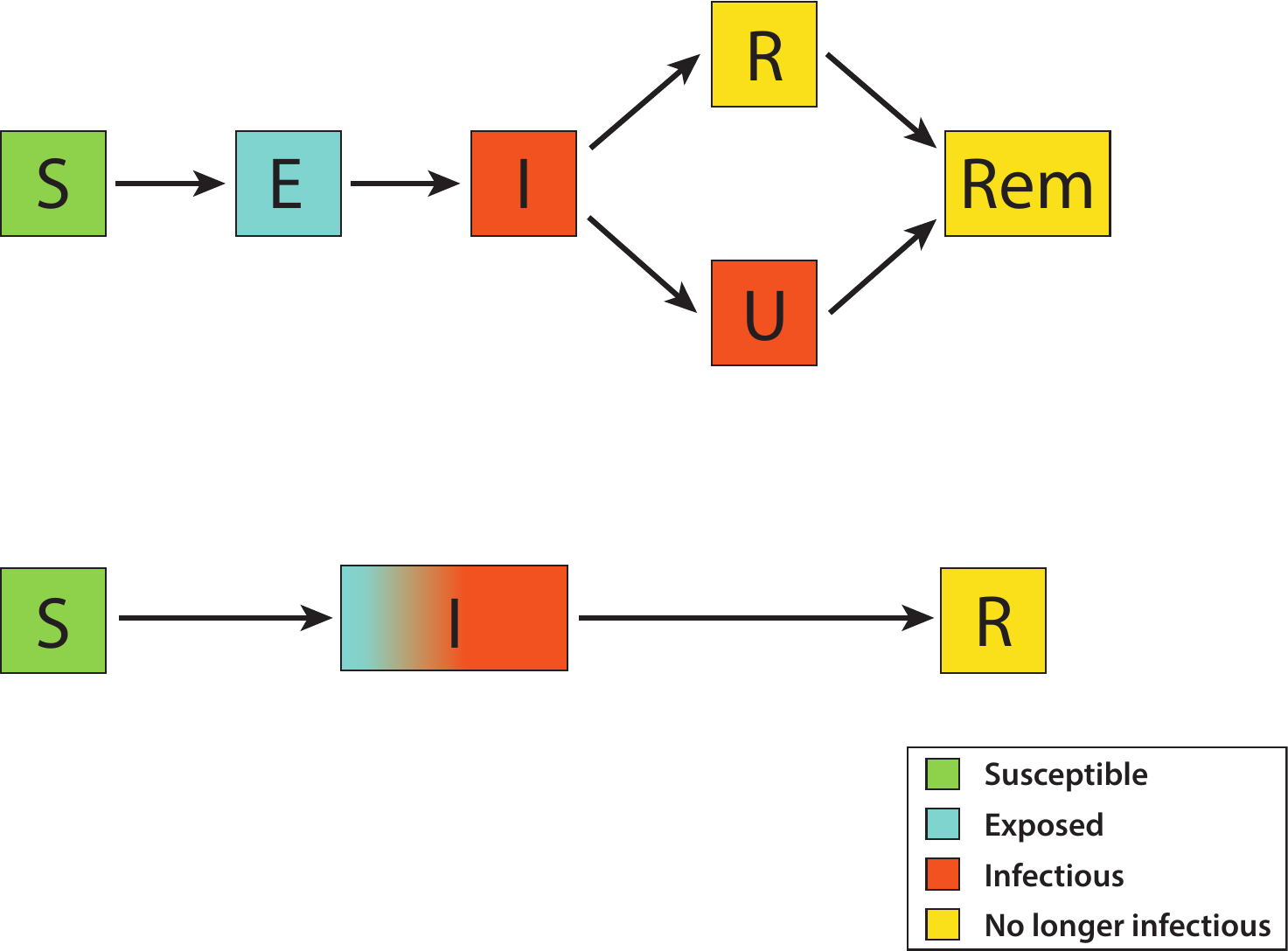}
	\caption{Flow chart of the SEIRU model of \cite{LMSW} and of our SIR model. We are able to replace the six compartments of the SEIRU model with only three compartments by using the equations described in Theorem~\ref{thm-FLLN-VI}.}\label{fig:SEIR}
\end{figure} 

 All the models which have been proposed for the Covid epidemic can be thought of as refinements of a Markovian SIR or SEIR model. Many of them include a bifurcation, which separates the infected individuals who are detected or not, who  have or not severe symptoms, who need to go to the hospital or not, to an intensive care unit or not, etc.
As one example of such Covid model, let us describe the $\bS\bE\bI\bR\bU$ model of \cite{LMSW}, see Figure \ref{fig:SEIR}. 
An individual who is infected is first ``Exposed'' $\bE$, then ``Infectious" $\bI$. Soon after, the infectious individual either develops significant symptoms, and then will be soon ``Reported'' $\bR$, and isolated so that he/she does not infect any more; while the alternative is that this infectious individual is asymptomatic: he/she develops no or very mild symptoms, so remains ``Unreported" $\bU$, and continues to infect susceptible individuals for a longer period. 
Both unreported and reported cases eventually enter the ``Removed" (${\bf Rem}$) compartment.
In this model, there are 6 compartments: $\bS$ like susceptible, $\bE$ like exposed, $\bI$ like infectious, 
$\bR$ like reported, $\bU$ like unreported, and ${\bf Rem}$ like removed.

Our approach allows us to have a more realistic version of this model with only 3 compartments (see Figure~\ref{fig:SEIR}): $\bS$ like susceptible, $\bI$ like infected (first exposed, then infectious), $\bR$ like removed (which includes the Reported individuals, since they do not infect any more, and will recover soon or later).  As already explained, we do not need to distinguish between the exposed and infectious, since the function $\lambda$ is allowed to remain equal to zero during a certain time interval starting from the time of infection. More importantly, since the law of $\lambda$ is allowed to be bimodal, we can accommodate in the same compartment $\bI$ individuals who remain infectious for a short duration of time, and others who will remain infectious much longer (but probably with a lower infectivity). Moreover, since we know, see \cite{he2020temporal}, that the infectivity decreases after a maximum which in the case of symptomatic individuals, seems  to take place shortly before symptom onset, our varying infectivity model allows us to use a model corresponding to what the medical science tells us about this illness. Note that our version of the SEIRU model from \cite{LMSW} is the same as the one which we have already used in \cite{FPP2020a} (except that there we had to distinguish the E and the I compartments). However, the main novelty here is that the infectivity decreases after a maximum near the beginning of the infectious period.

\begin{figure}[ht]
	\centering
	\includegraphics[width=0.5\textwidth]{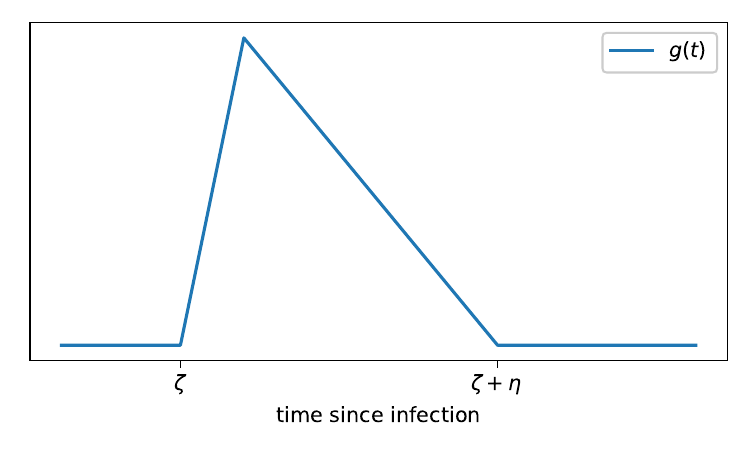}
	\caption{Profile of the function $g(t)$ used in our computation of $ R_0 $ as a function of $ \zeta $ and $ \eta $. The function increases linearly (up to a value 1 or $ \alpha $ depending on whether the individual is reported or unreported) on the interval $ [\zeta, \zeta + \eta/5] $ and then decreases linearly on $ [\zeta + \eta/5, \zeta + \eta] $.} \label{fig:g_t_profile}
\end{figure}

More precisely, we consider that $ t \mapsto g(t) $ increases linearly on the time interval $ [\zeta, \zeta + \eta / 5] $, from 0 to 1 for reported individuals, and from 0 to $ \alpha $ for unreported individuals, and that it then decreases linearly to 0 on the interval $ [\zeta + \eta / 5, \zeta + \eta] $, as shown on Figure~\ref{fig:g_t_profile}.
We then take $ (X_1, X_2) $ a pair of independent Beta random variables with parameters (2, 2) and we assume that
\begin{align*}
	\zeta = 2 + 2 X_1, && \eta = \begin{cases}
	3 + X_2 & \text{ for reported individuals,} \\
	8 + 4 X_2 & \text{ for unreported individuals.}
	\end{cases}
\end{align*}
This joint law of $ (\zeta, \eta) $ is the one that was used in \cite{FPP2020a}
to study the Covid--19 epidemic in France (where the infectivity was assumed to be constant and uniform among individuals in that work), and these values are compatible with the results described in \cite{he2020temporal}.

\begin{remark}
It is a general tendency, in order to make Markov models more realistic, to increase the number of compartments, compared to the general SEIR model. Our random varying infectivity approach allows us in the contrary to reduce the number of compartments. Not only do we  not need to distinguish between $\bE$ and $\bI$, but we also do not really need to introduce the $\bR$ compartment. When an infected individual creases to be infectious, he or she recovers, but formally we can keep him/her  in the $\bI$ compartment, at least as concerns the description of the dynamic of the epidemic. However, separating the $\bI$ and $\bR$ compartments is useful if we want to follow the number of infected (and infectious or soon infectious) individuals in the population. The same reduction of the number of compartments appears in the model of the Covid epidemic proposed by \cite{foutel2020individual}, which in that respect is similar to our model.
\end{remark}

\section{Models with infection-age dependent  infectivity and limiting PDEs}\label{sec:PDE}
Kermack and McKendrick pioneered the introduction of PDE models to describe infection-age dependent infectivity and recovery-age dependent susceptibility in their 1932 paper \cite{KMK32}. Here we shall describe the PDE model for an infection-age dependent infectivity, in the framework of a $\bS\bI\bR$/$\bS\bE\bI\bR$, \textit{i.e.,} where the recovered individuals do not lose their immunity. 
We shall obtain the deterministic (here PDE/integral equation) model as a LLN limit of individual based stochastic models. The underlying assumptions are the same as in the previous section, except for the initially infected individuals for whom the age of infection is given at time zero. But we shall give a different description of the model, as we shall see now.

We have the same compartments as in Section~\ref{sec:varinfect-eqs}, $S^N(t)$, $I^N(t)$ and $R^N(t)$ are as above, and again $S^N(t)+I^N(t)+R^N(t)\equiv N$. Let now $\mfI^N(t,x)$ be the number of infected individuals at time $t$ that have been infected for a duration less than or equal to $x$. 
Note that for each $t$, $\mfI^N(t,x)$ is nondecreasing in $x$, which is the distribution of $I^N(t)$ over the infection-ages.  
Let $A^N(t)$ be the cumulative number of newly infected individuals in $(0,t]$, with the jump times $\{\tau^N_i: i \in \NN\}$.  Each individual who has been infected after time $0$ has an infectivity process $\lambda_i(\cdot)$, and we assume that these random functions  are i.i.d.. Let $\eta_i= \sup\{t>0: \lambda_i(t) > 0\}$ be the infected period corresponding to the individual that gets infected at time $\tau^N_i$. The $\eta_i$'s are i.i.d., with a cumulative distribution function (c.d.f.) $F$.  Let $F^c=1-F$.

Let $\{\tau_{j,0}^N, j =1,\dots, I^N(0)\}$ be the times at which the initially infected individuals at time 0 were infected. 
Then $\tilde{\tau}_{j,0}^N = -\tau_{j,0}^N$, $ j =1,\dots, I^N(0)$, represents the age of infection of individual $j$ at time $0$. 
W.l.o.g., we assume that $0 > \tau_{1,0}^N> \tau_{2,0}^N> \cdots >\tau_{I^N(0),0}^N$ (or equivalently 
$0 < \tilde{\tau}_{1,0}^N< \tilde{\tau}_{2,0}^N< \cdots < \tilde{\tau}_{I^N(0),0}^N$). Set $\tilde{\tau}_{0,0}^N=0$. 
We define $\mathfrak{I}^N(0,x) = \max\{j \ge 0: \tilde{\tau}_{j,0}^N \le x\}$, the number of initially infected individuals that have been infected for a duration less than or equal to $x$ at time $0$. Assume that there exists $0 \le \bar{x}< \infty$  such that $I^N(0) = \mfI^{N}(0, \bar{x})$ a.s. 

To each initially infected individual $j=1,\dots, I^N(0)$, is associated an infectivity process $\lambda_j^0(\cdot)$, and we assume that they are also i.i.d., with the same law as $\lambda_i(\cdot)$.  
For each $j$, let  $\eta^0_j =\sup\{t>0: \lambda_j^0(\tilde{\tau}_{j,0}^N+t)>0\}$ be the remaining infectious period, which depends on the elapsed infection time $\tilde{\tau}_{j,0}^N$, but is independent of the elapsed infection times of other initially infected individuals.  
In particular, the conditional distribution of $\eta^0_j$ given that $\tilde{\tau}_{j,0}^N=s>0$ is given as in \eqref{eqn-eta0-age}.
Note that the $\eta^0_j$'s are independent but not identically distributed. 

For an initially infected individual $j=1,\dots,I^N(0)$, 
 the infection age is given by $\tilde{\tau}^N_{j,0}+ t$.
 For a newly infected individual $i$, the infection age is given by 
 $t- \tau^N_i$. 
 Note that $\lambda_i(\cdot)$ and $\lambda^0_j(\cdot)$ are equal to zero on $\RR_{-}$. 
 
 The aggregate force of infection at time $t$ is given by 
 \begin{align} \label{eqn-cI-n}
\mathfrak{F}^N(t) = \sum_{j=1}^{I^N(0)} \lambda_j^0 (\tilde{\tau}^N_{j,0}+t)   
+ \sum_{i=1}^{A^N(t)} \lambda_i(t-\tau^N_i) , \quad t \ge 0.
 \end{align}
We have again \eqref{eqn-Upsilon-VI} and \eqref{eq:AN}. Moreover the total number of individuals infected at time $t$ that have been infected for a duration which is less than or equal to $x$:
$$
\mfI^N(t,x) = \mfI^N_0(t,x) + \mfI^N_1(t,x), \quad t \ge 0, \, x \ge 0,
$$
where $\mfI^N_0(t,x)$ is the number of initially infected individuals who have been infected for a duration less than or equal to $x$ at time $t$, which is given as  
\begin{align} \label{eqn-In0-rep}
\mfI^{N}_0(t,x) = \sum_{j=1}^{I^N(0)} \bone_{\eta_j^0 > t} \bone_{\tilde{\tau}_{j,0}^N  \le (x-t)^+ } = \sum_{j=1}^{\mfI^N(0, (x-t)^+)} \bone_{\eta_j^0 > t}\,, \quad t, x\ge 0,
\end{align}
and $\mfI^{N}_1(t,x)$ is the number of newly infected individuals who have been infected for a duration  less than or equal to $x$ at time $t$, which equals
\begin{align} \label{eqn-In1-rep}
\mfI^{N}_1(t,x) &=  \sum_{i=1}^{A^N(t)} \bone_{(t-x)^+<\tau^N_i \le t} \bone_{\tau^N_i + \eta_i >t} = \sum_{i=1}^{A^N(t)}  \bone_{\tau^N_i + \eta_i >t}- \sum_{i=1}^{A^N((t-x)^+)}  \bone_{\tau^N_i + \eta_i >t}\non \\
&=  \sum_{i=A^N((t-x)^+)+1}^{A^N(t)}  \bone_{\tau^N_i + \eta_i >t} \,.
\end{align}

Note that for each $t$,  $\mfI^N_0(t, \cdot)$ has support over $[0, t+\bar{x}]$ and $\mfI^N_1(t, \cdot)$ has support over $[0,t]$.  Thus
$$
I^N(t) = \mfI^N_0(t,t+\bar{x}) + \mfI^N_1(t,t) =\mfI^N(t,\infty), \quad t \ge 0. 
$$

The sample paths of $\mfI^N(t,x)$ belong to the space $\bD_\bD$, denoting $ \bD(\RR_+;\bD(\RR_+;\RR))$, the $\bD$-valued $\bD$ space.

Define the fluid-scaled processes  $\bar{X}^N= N^{-1} X^N$ for any processes $X^N$. 
We make the following assumptions on the initial quantities. 

\begin{assumption} \label{AS-FLLN-Initial}
There exists a deterministic continuous nondecreasing function $\bar{\mfI}(0,x)$ for $x\ge 0$ with $\bar{\mfI}(0,0)=0$ such that 
$\bar{\mfI}^{N}(0,\cdot) \to \bar{\mfI}(0,\cdot)$ in $\bD$  in probability as $N\to\infty$. 
 Let $\bar{I}(0) = \bar{\mfI}(0, \bar{x})$.  Then $(\bar{I}^N(0), \bar{S}^N(0),\bar{R}^N(0)) \to (\bar{I}(0), \bar{S}(0),\bar{R}(0)) \in (0,1)^3$  in probability as $N\to \infty$ where  $ \bar{S}(0) + \bar{I}(0)+\bar{R}(0)=1$. 
\end{assumption}

\begin{remark} Suppose now that the r.v.'s $\{\tau_{j,0}^N\}_{1\le j\le N}$ are not ordered, but rather i.i.d., with a common distribution function $G$ which we assume to be  continuous. It then follows from the LLN  that Assumption~\ref{AS-FLLN-Initial} holds in this case.
\end{remark} 

 We have the following FLLN. 

\begin{theorem} \label{thm-FLLN-PDE}
Under Assumptions~\ref{AS-lambda} and~\ref{AS-FLLN-Initial}, as $N\to \infty$,
\begin{align*}
\big(\bar{S}^N, \overline{\mathfrak{F}}^N, \bar{\mfI}^{N},   \bar{R}^N\big) \to\big(\bar{S},\overline{\mathfrak{F}}, \bar{\mfI}, \bar{R}\big) \ \text{in probability, locally uniformly in $t$ and $x$}, 
\end{align*}
where the limits are the unique continuous solution to the following set of integral equations,  for $t, x\ge 0$, 
\begin{align}
\bar{S}(t) &= \bar{S}(0) -  \int_0^t\bar{\Upsilon}(s) ds, \\
 \overline{\mathfrak{F}}(t) 
  &= \int_0^{\bar{x}} \bar\lambda(y+t)  \bar{\mfI}(0, dy)  + \int_0^t  \bar\lambda(t-s)  \bar\Upsilon(s) ds\,,
  \label{eqn-overline-cal-I-2}\\
\bar{\mfI}(t,x) &=   \int_0^{(x-t)^+} \frac{F^c(t+y)}{F^c(y)}  \bar{\mfI}(0, dy)   +  \int_{(t-x)^+}^t F^c(t-s) \bar{\Upsilon}(s) ds, \label{eqn-barI}\\
\bar{R}(t) &=\bar{R}(0)+\int_0^{\bar{x}} \left(1- \frac{F^c(t+y)}{F^c(y)} \right)  \bar{\mfI}(0, dy) + \int_0^t F(t-s) \bar{\Upsilon}(s) ds,  \label{eqn-barR}
\end{align}
with 
\begin{equation} \label{eqn-bar-Upsilon}
\bar{\Upsilon}(t) =\bar{S}(t)  \overline{\mathfrak{F}}(t) =   \bar\mfI_x(t,0)\,. 
\end{equation}
 The function $\bar{\mfI}(t,x)$ is nondecreasing in $x$ for each $t$,  the integral w.r.t. $\bar{\mfI}(0,dy)$ is a Lebesgue-Stieltjes integral with respect to the measure which coincides with the distributional derivative $\partial_x\bar{\mfI}(0,\cdot)=\bar{\mfI}_x(0,\cdot)$. 
 As a consequence, $\bar{I}^N\to \bar{I}$ in $\bD$ in probability as $N\to\infty$ where 
 \begin{align}\label{eqn-barIt}
 \bar{I}(t)= \bar{\mfI}(t, t+\bar{x}) =  \int_0^{\bar{x}} \frac{F^c(t+y)}{F^c(y)}  \bar{\mfI}(0, d y)   +  \int_{0}^t F^c(t-s) \bar{\Upsilon}(s) ds, \quad t \ge 0. 
 \end{align}
 \end{theorem}
The proof of Theorem~\ref{thm-FLLN-PDE} is similar to the proofs of the FLLNs
 in the previous sections, with the additional complications that we have one function of two parameters.
 We refer the reader to Section 5 of \cite{PP-2021} for that proof.
 
 We now turn to deriving a PDE for the derivative with respect to $x$ of 
 $\bar{\mathfrak{I}}(t,x)$, when it exists. 
 The PDE models are linear equations with a nonlinear boundary condition, as studied in,  for example, \cite{inaba2001kermack,magal2013two,foutel2020individual}.

 In the next result, we shall assume that $F$ is absolutely continuous, $F(dx) = f(x) dx$ and denote by $\mu(x)=f(x)/F^c(x)$ the hazard function of the r.v. $\eta$. We refer to Section 3 of \cite{PP-2021} for its proof.

 \begin{prop} \label{prop-PDE}
Suppose that $F$ is absolutely continuous, with the density $f$, and that $\bar{\mfI}(0,x)$ is differentiable with respect to $x$, with the density function $\bar{\mfi}(0,x)$. Then for $t>0$, the increasing function $\bar{\mfI}(t,\cdot)$ is absolutely continuous, and $\bar{\mfi}(t,x):=\partial_x \bar{\mfI}(t,x)$  satisfies $(t,x)$ a.e. in $(0,+\infty)^2$, 
\begin{align} \label{eqn-barI-density-PDE}
 \frac{\partial \bar{\mfi}(t,x)}{\partial t} +  \frac{\partial \bar{\mfi}(t,x)}{\partial x}
 & = -\mu(x) \bar{\mfi}(t,x) \,,   
\end{align}
with the initial condition $\bar{\mfi}(0,x)= \bar{\mfI}_x(0,x)$ for $x \in [0,\bar{x}]$, and
the boundary condition
\begin{equation}\label{BC}
\bar{\mfi}(t,0)=\bar{S}(t)\int_0^{t+\bar{x}}\frac{\bar{\lambda}(x)}{\frac{F^c(x)}{F^c(x-t)}} \bar{\mfi}(t,x)dx\,,
\end{equation}
with the convention that $F^c=1$ on $\R_-$,  and that the integrand in \eqref{BC} is zero when $F^c(x)=0$. 

In addition, 
\begin{equation} \label{eqn-barS-der}
\bar{S}'(t) = -  \bar{\mfi}(t,0),\quad\text{and }\ \bar{S}(0)=1-\bar{I}(0)\,.
\end{equation}

Moreover, the PDE \eqref{eqn-barI-density-PDE} has a unique solution 
which is given as follows. For $x\ge t$, 
\begin{equation}\label{ident1}
 \bar{\mfi}(t,x)=\frac{F^c(x)}{F^c(x-t)}\bar{\mfi}(0,x-t) \, ,
 \end{equation}
 while for $t>x$,
 \begin{equation}\label{ident3}
\bar{\mfi}(t,x)=F^c(x)\bar{\mfi}(t-x,0)\,,
\end{equation}
and the boundary function is the unique solution of the integral equation
\begin{equation} \label{ident2}
\bar{\mfi}(t,0)= \left(  \bar{S}(0)-\int_0^t\bar{\mfi}(s,0)ds \right) \left( \int_0^{\bar{x}} \bar\lambda(y+t)  \bar{\mfi}(0,y)dy   + \int_0^t  \bar\lambda(t-s) \bar{\mfi}(s,0)ds \right)\,.
\end{equation}
\end{prop}

\begin{remark}
The reason why we can impose that the integrand in the right hand side of \eqref{BC} is zero whenever $F^c(x)=0$ is because $F^c(x)=0$ implies that $\bar{\mfi}(t,x)=0$ by \eqref{ident1} and \eqref{ident3}. 

In the special case  $\lambda_i(t)=\tilde\lambda(t)\bone_{t<\eta_i}$, where $\tilde\lambda(t)$ is a deterministic function, we obtain 
 $$\bar\lambda(t)=\tilde\lambda(t)F^c(t), \quad \E\big[\lambda^0(t)|\tilde\tau^N_{0}=y\big]=\tilde\lambda(t+y)\frac{F^c(t+y)}{F^c(y)}.$$ 
The boundary condition in \eqref{BC} then becomes 
 $$\bar\mfi(t,0) = \bar{S}(t) \int_0^{t+\bar{x}} \tilde\lambda(x)  \bar{\mfi}(t,x) dx$$
 This is usually how the boundary condition is imposed in the literature of PDE epidemic models (see, e.g., \cite[equation (2.5)]{inaba2001kermack}, \cite[equation (1.1)]{magal2013two} and \cite[equation (2)]{foutel2020individual}). 

\end{remark}

When  the distribution $F$ is not absolutely continuous, we denote below by $\nu$ the law of $\eta$, i.e., the measure whose distribution function is $F$, and let 
\[ G(t)=F(t^-),\ \ \ G^c(t)=1-G(t)=F^c(t^-)\,,\]
which are the left continuous versions of $F$ and $F^c$.
We will need to use $G^c(x)$ in the denominator, but not $F^c(x)$. The reason is that
if the support of $\nu$ is $[0,x_{max}]$, and $\nu(\{x_{max}\})>0$, then $F^c(x_{max})=0$, while $G^c(x_{max})>0$. 
We need a positive denominator at the point $x_{max}$, since $\nu(\{x_{max}\})>0$.

The above PDE result is generalized to the following. 

\begin{prop} \label{prop-PDE-measure}
Suppose that $\bar{\mfI}(0,x)$ is differentiable with respect to $x$, with the density function $\bar{\mfi}(0,x)$. Then for $t>0$, the increasing function $\bar{\mfI}(t,\cdot)$ is absolutely continuous, 
 and the following identity holds: 
\begin{align} \label{eqn-barI-density-PDE-mu-m}
 \frac{\partial \bar{\mfi}(t,x)}{\partial t} +  \frac{\partial \bar{\mfi}(t,x)}{\partial x}
 & = -\frac{\bar{\mfi}(t,x)}{G^c(x)}\nu(dx) \,,   
\end{align}
(i.e., the distribution which appears on the left hand side of \eqref{eqn-barI-density-PDE-mu-m} equals the measure
which has the density $-\frac{\bar{\mfi}(t,x)}{G^c(x)}$ with respect to the measure $\nu$) 
with the initial condition $\bar{\mfi}(0,x)= \bar{\mfI}_x(0,x)$ for $x \in [0,\bar{x}]$, and
the boundary condition
\begin{equation}\label{BC-m}
\bar{\mfi}(t,0)=\bar{S}(t)\int_0^{t+\bar{x}}\frac{\bar{\lambda}(x)}{\frac{G^c(x)}{G^c(x-t)}} \bar{\mfi}(t,x)dx\,,
\end{equation}
with the convention that $G^c=1$ on $\R_-$, and that the integrand in \eqref{BC-m} is zero whenever $G^c(x)=0$.

In addition, 
\begin{equation} \label{eqn-barS-der-m}
\bar{S}'(t) = -  \bar{\mfi}(t,0),\quad\text{and }\ \bar{S}(0)=1-\bar{I}(0)\,.
\end{equation}

Moreover, the PDE \eqref{eqn-barI-density-PDE-mu-m} has a unique solution 
which is given as follows. For $x\ge t$, 
\begin{equation}\label{ident1-m}
 \bar{\mfi}(t,x)=\frac{G^c(x)}{G^c(x-t)}\bar{\mfi}(0,x-t) \, ,
 \end{equation}
 while for $t>x$,
 \begin{equation}\label{ident3-m}
\bar{\mfi}(t,x)=G^c(x)\bar{\mfi}(t-x,0)\,,
\end{equation}
and the boundary function is the unique solution of the integral equation
\begin{equation} \label{ident2-m}
\bar{\mfi}(t,0)= \left(  \bar{S}(0)-\int_0^t\bar{\mfi}(s,0)ds \right) \left( \int_0^{\bar{x}} \bar\lambda(y+t)  \bar{\mfi}(0,y)dy   + \int_0^t  \bar\lambda(t-s) \bar{\mfi}(s,0)ds \right)\,.
\end{equation}
\end{prop}

\begin{remark}\label{re:leftcont}
The product $\frac{\bar{\mfi}(t,x)}{G^c(x)}\nu(dx)$ can also be rewritten as
\[ \bar{\mfi}(t,x)\times \frac{\nu(dx)}{G^c(x)}\, ,\]
where the second factor can be thought of as the ``hazard measure'', i.e., the generalization of the hazard function, of the r.v. $\eta$. 

\end{remark}

\subsection{The SIS model with infection-age dependent infectivity} \label{sec:SIS-PDE}

In the SIS model, the infectious individuals become susceptible once they recover. Since $S^N(t) + I^N(t) = N$ for each $t\ge 0$, the epidemic dynamics is determined by the process $I^N(t)$ alone, and we have the same representations of the processes $\mfI^N_0(t,x)$ and $\mfI^N_1(t,x)$ in \eqref{eqn-In0-rep} and \eqref{eqn-In1-rep}, respectively, while in the formula for $\Upsilon^N$ in \eqref{eqn-Upsilon-VI}, $S^N(t) = N - I^N(t)$.  The aggregate infectivity process $\mathfrak{F}^N(t)$ is still given by \eqref{eqn-cI-n}. 
The two processes $(\mathfrak{F}^N,\mfI^N)$ determine the dynamics of the SIS epidemic model. 
Under Assumption~\ref{AS-FLLN-Initial}, 
\begin{align*}
(\overline{\mathfrak{F}}^N, \bar{\mfI}^{N}) \to (\overline{\mathfrak{F}},\bar{\mfI}) \ \text{in probability, locally uniformly in $t$ and $x$}, \qasq N \to\infty,
\end{align*}
where 
\begin{align*} 
 \overline{\mathfrak{F}}(t) 
  &= \int_0^{\bar{x}} \bar\lambda(y+t)  \bar{\mfI}(0,dy)  + \int_0^t  \bar\lambda(t-s)  \big(1- \bar{\mfI}(s,\infty) \big) \bar{\mathfrak{F}}(s)ds\,, \\
\bar{\mfI}(t,x) 
&=  \int_0^{(x-t)^+} \frac{F^c(t+y)}{F^c(y)}  \bar{\mfI}(0, dy)   +  \int_{(t-x)^+}^t F^c(t-s) \big(1- \bar{\mfI}(s,\infty) \big) \bar{\mathfrak{F}}(s) ds\,,
\end{align*}
for $t,x \ge 0$. 
 If $\mfI(0,x)$ is differentiable and $F$ is absolutely continuous, then
the density function 
 $\bar{\mfi}(t,x)= \frac{\partial \bar{\mfI}(t,x)}{\partial x}$, if it exists, satisfies again \eqref{eqn-barI-density-PDE}.
The same calculations as in the case of the SIR model lead to \eqref{ident1}, \eqref{ident3} and
\begin{align*}
\bar{\mfi}(t,0)=\bar{S}(t)\left( \int_0^{\bar{x}} \bar\lambda(y+t)  \bar{\mfi}(0,y)dy   + \int_0^t  \bar\lambda(t-s) \bar{\mfi}(s,0)ds \right)\,. 
\end{align*}
However, the formula for $\bar{S}(t)$ is different in the case of the SIS model.  We have 
\begin{align*}
\bar{S}(t)&=1-\bar{I}(t)
=1-\int_0^{\bar{x}}\frac{F^c(t+y)}{F^c(y)}\bar{\mfi}(0,y)dy-\int_0^tF^c(t-s)\bar{\mfi}(s,0)ds\,.
\end{align*}
Thus,  the Volterra equation on the boundary reads
\begin{equation*}
\begin{split}
\bar{\mfi}(t,0)&=\left(\int_0^{\bar{x}}\bar{\lambda}(t+y)\bar{\mfi}(0,y)dy
+\int_0^t\bar{\lambda}(t-s)\bar{\mfi}(s,0)ds\right)\\ &\quad\times
\left(1-\int_0^{\bar{x}}\frac{F^c(t+y)}{F^c(y)}\bar{\mfi}(0,y)dy-\int_0^tF^c(t-s)\bar{\mfi}(s,0)ds\right)\,,
\end{split}
\end{equation*}
whose form is similar to the one for the SIR model.

%
%
%
%
%
%

Recall that the standard SIS model has a nontrivial equilibrium point $\bar{I}^*= 1-\mu/\lambda$ if $\mu<\lambda$, where  $\lambda$ is the infection rate (the bar over $\lambda$ is dropped for convenience),  and $1/\mu$ is the mean of the infectious periods. See Section 4.3 in \cite{PP-2020} for the account of the SIS model with general infectious periods. Here we consider the model in the generality of infection-age dependent infectivity.

\begin{prop}\label{prop:iast}
Suppose that $ \lim \bar{\mfI}(t,x) \to \bar{\mfI}^\ast(x)$ exists as $t\to\infty$ and $ \bar{I}^\ast =\bar{\mfI}^\ast(\infty)$.
If $R_0=\int_0^\infty \bar\lambda(y)dy\le1$,  $\bar{I}^\ast=0$ (the disease free equilibrium). In the complementary case, $R_0=\int_0^\infty \bar\lambda(y)dy>1$,  if $\bar{\mfI}(0,\bar{x})>0$,
\begin{equation} \label{eqn-barI*-formula}
 \bar{I}^\ast=1-\left(\int_0^\infty \bar\lambda(y)dy\right)^{-1} =1-\frac{1}{R_0}\,.
 \end{equation}
The density function $\bar{\mfi}(t,x)$ has an equilibrium $ \bar{\mfi}^*(x)$ in the age of infection $x$, given by
\begin{align} \label{eqn-barI*-x-formula}
 \bar{\mfi}^*(x) =  \frac{d \bar{\mfI}^*(x)}{d x} =  \bar{I}^*  \mu F^c(x),
\end{align}
where $\mu^{-1}=\int_0^\infty F^c(t)dt$ is the expectation of the duration of the infectious period. 
If $F$ has a density $f$, then the equilibrium density $ \bar{\mfi}^*(x)$ satisfies 
\begin{align*}
 \frac{d \, \bar{\mfi}^*(x)}{d x} = -  \bar{I}^*  \mu f(x), \quad  \bar{\mfi}^*(0)=  \bar{I}^*  \mu. 
\end{align*}
\end{prop}

\begin{proof}
 The fact that $ \bar{I}^\ast =0$ if $R_0\le1$ and $>0$ if $R_0>1$ follows from 
branching process arguments, and the fact that the start of the epidemic can be approximated by a branching process, see e.g.  section 1.3 in \cite{brittonpardoux}. 
Assume that the equilibrium $\bar{\mfI}^*(x) := \bar{\mfI}(\infty,x)$ exists. 
Then it must satisfy
\begin{align*}
\bar\mfI^\ast(x)&=(1-\bar\mfI^\ast(\infty))\int_0^xF^c(u)du\int_0^\infty\frac{\bar\lambda(y)}{F^c(y)}\bar\mfI^\ast(dy)
\\
&=(1-\bar{I}^\ast)\mu^{-1}F_e(x)\int_0^\infty\frac{\bar\lambda(y)}{F^c(y)} \bar\mfI^\ast(dy),
\end{align*}
where $F_e(x)=\mu \int_0^x F^c(s) ds$, the equilibrium (stationary excess) distribution.  Letting $x\to\infty$ in this formula, we deduce
\begin{equation*}
\bar{I}^\ast=(1-\bar{I}^\ast)\mu^{-1}\int_0^\infty\frac{\bar\lambda(y)}{F^c(y)}\bar\mfI^\ast(dy)\,.
\end{equation*}
Combining the last two equations, we obtain
\begin{equation}\label{eqn-barmfI*} 
\bar\mfI^\ast(x)=\bar{I}^\ast F_e(x)\,.
\end{equation}
Plugging this formula in the previous identity, we deduce that
\[ \bar{I}^\ast=(1-\bar{I}^\ast)\bar{I}^\ast\int_0^\infty \bar\lambda(y)dy\,.\]
Then the formula \eqref{eqn-barI*-formula} can be directly deduced from this equation. 
The formula \eqref{eqn-barI*-x-formula} follows by  taking the derivative with respect to $x$ in \eqref{eqn-barmfI*}.
\end{proof}

\subsection{On the comparison between the Markov/ODE and the non--Markov/integral equation--PDE}
The models -- both stochastic and deterministic -- which we have considered starting from Section~\ref{sec:nonmarkov} differ from the models which are mostly used in epidemic modeling, although, at least concerning the deterministic models, what we are doing is not new compared to the pioneering work of Kermack and McKendrick \cite{KMK} from 1927. 

We believe that the $\bS\bI\bR$ models considered  in Sections~\ref{sec:varinfect} and~\ref{sec:PDE} can be made, by a proper choice of the parameters, much more realistic that the Markov / ODE models of Section~\ref{sec:markov}. One may however ask the question whether those ``refined'' models make a real difference, as compared to Markov / ODE models. It is not clear that the large time behaviors of the two kind of models are significantly different. Indeed, we note that the formula for the endemic equilibrium in the $\bS\bI\bS$ model \eqref{eqn-barI*-formula} in Proposition~\ref{prop:iast} is exactly the same, when expressed in terms of the basic reproduction number $R_0$, as the formula obtained in Section~\ref{sec:sism}.

On the other hand, the transitory behavior can be drastically different in the two types of models. 
Indeed, we can implement in our ``refined'' model the memory of recent situations of the epidemic, so that when the rate of contact between individuals changes drastically (e.g. when a government enforces a lockdown), in the ``refined'' model the number of daily infections will take more time to go down than in the ODE model, see Figure~\ref{Fig:4}. 
Some authors who use ODE models correct that behavior by making the contact rate change gradually after the time of lockdown, which does not correspond to the behavior of the population. 
As observed several times and in various places during the Covid-19 pandemic, daily infections, as well as hospital admissions and hospital deaths, continue to grow for a relatively long time (up to a few weeks) after strict preventive measures are taken, a pattern that arises naturally in models with memory but is much harder to reproduce using ODE models. 
This behavior has consequences both for inference methods and for decision making, since ODE models can underestimate the inertia of the epidemic on the short term.

 
 \begin{figure}
     \includegraphics[width=.6\textwidth]{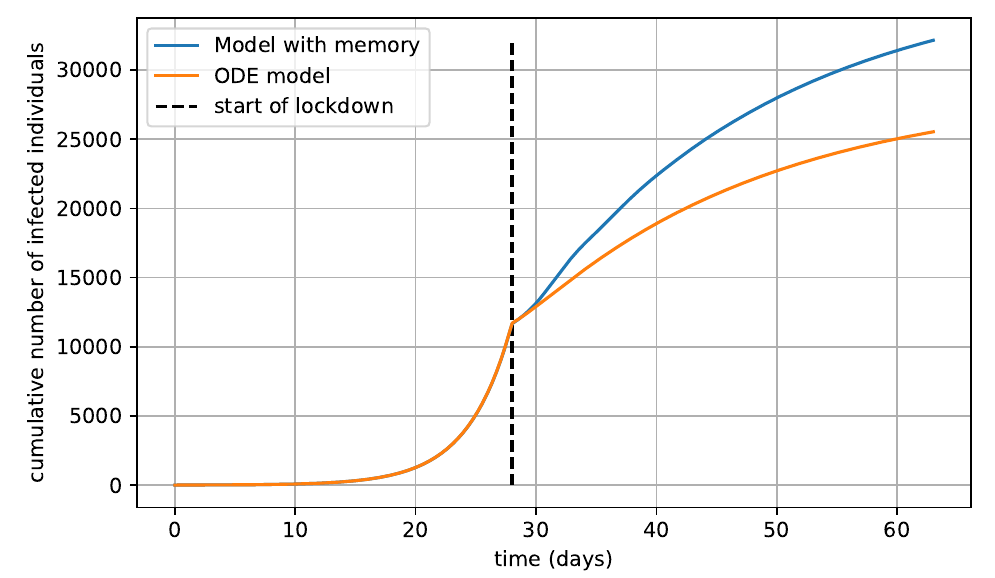}
     \caption{Cumulative number of infected individuals through time in two models: an ODE model obtained as the limit of a Markov stochastic SEIR model and a model with memory obtained as a limit of a non-Markov SEIR model. Both models have the same mean exposed and infectious period, and are chosen so that they have the same initial growth rate. After 28 days, the contact rate is reduced instantly, in such a way that both models should have the same rate of decay of newly infected individuals. We see that the epidemic in the model with memory ``slows down'' less rapidly than in the ODE model, due to its greater inertia, which causes a larger number of individuals to be infected.}
     \label{Fig:4}
\end{figure}

\section{Models with varying infectivity and immunity}\label{sec:varinfect-varimmun}

As in the previous sections, we start with a population of fixed size $N$, and we enumerate the individuals in the population with the parameter $k$, $1\le k\le N$.

The model in this section is a sort of $\bS\bI\bR\bS$ model, except that we do not really distinguish between the states $\bR$ and $\bS$. We shall only consider the compartments $\bS$ and $\bI$, and formally our model is a $\bS\bI\bS$ model, although individuals experience immunity after recovery, before being susceptible again.
Each individual who is infected first draws a random infectivity function, as in the previous section, and may infect other individuals as before.
At the end of the infectious period, the individual is first immune (\textit{i.e.,} its susceptibility is equal to zero), but this acquired immunity then wanes with time, and we assume that the susceptibility of the individual increases gradually according to some random function.
When such a partially susceptible individual is the target of an infectious contact, the probability that this individual becomes reinfected is given by its susceptibility (so this quantity evolves between 0 and 1). Whenever an individual in the compartment $\bS$ has susceptibility $0$, he/she is in fact immune.

Since each individual might get infected an arbitrary number of times, and the infectivity and susceptibility functions of the infection age are a priori different after each new infection, we attach to each individual a countable family of infectivity and susceptibility functions. More precisely, we consider two mutually independent families $(\lambda_{k,i},\gamma_{k,i})_{1\le k\le N,\ i\ge0}$ and $(\lambda_{k,0},\gamma_{k,0})_{1\le k\le N}$ of i.i.d. elements of $\bD^2$, which are such that  all $\lambda_{k,i}$ take values in $[0,\lambda^\ast]$ and all $\gamma_{k,i}$ take values in $[0,1]$. The last quantities represent the infectivity and susceptibility starting at time $t=0$, while for $i\ge1$, $(\lambda_{k,i},\gamma_{k,i})$ represents the infectivity and susceptibility of the $k$--th individual after his/her $i$--th infection (not counting a possible infection before time $0$). 

At time $0$,  individual $k$ can be susceptible (or ``naive''). In that case,  $\lambda_{k,0}(t)\equiv0$ and $\gamma_{k,0}(t)\equiv1$. A second possibility is that individual $k$ is infected. In that case $\lambda_{k,0}\ge0$ and 
$\gamma_{k,0}(0)=0$. A third possibility is that  individual $k$ has recovered at time $0$ from a past infection. In that case, $\lambda_{k,0}(t)\equiv0$ and the function $\gamma_{k,0}$ is arbitrary.

In addition to what has been explained above, we assume that  
\begin{equation*}
\sup\{t\ge0,\ \lambda_{k,i}(t)>0\}\le \inf\{t\ge0,\ \gamma_{k,i}(t)>0\}. 
\end{equation*}
See Figure \ref{fig:lg} for an example of a pair $(\lambda,\gamma)$.
\begin{figure}
     \includegraphics[width=.6\textwidth]{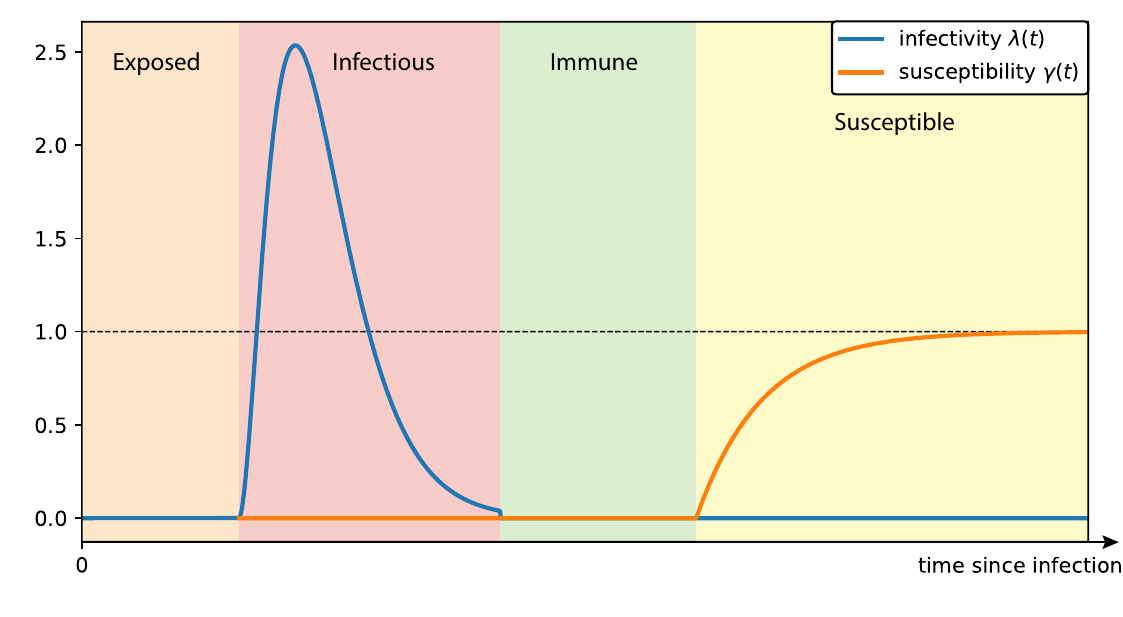}
     \caption{An example of a pair of functions $(\lambda,\gamma)$ which satisfies our assumptions.}
     \label{fig:lg}
\end{figure}

We introduce the following notations:
\begin{align*}
&\eta_{k,0}=\sup\{t\ge0,\ \lambda_{k,0}(t)>0\}, \\
&\bar{I}(0) = \P(\eta_{1,0} > 0)\\
&\bar{\lambda}^0(t)=\E\left[\lambda_{1,0}(t)|\eta_{1,0}>0\right],\\ 
&\bar{\lambda}(t)=\E\left[\lambda_{1,1}(t)\right]\,. 
\end{align*}

Let us now describe our individual based stochastic model. Contrary to what we did in the previous sections, we do not just count the number of infections in the population on the time interval $(0,t]$. We shall denote by $A^N_k(t)$ the number of times that the individual $k$ has been infected on the time interval $(0,t]$. Let $\sigma^N_k(t)$ denote the age of infection of the individual $k$ at time $t$, \textit{i.e.},
\[ \sigma^N_k(t):=t-\sup\{s\in[0,t],\ A^N_k(s)=A^N_k(s^-)+1\}\,,\]
with the convention that the sup of an empty set is $0$.
At time $t$, the infectivity of the individual $k$ is $\lambda_{k,A^N_k(t)}(\sigma^N_k(t))$, and its susceptibility is 
$\gamma_{k,A^N_k(t)}(\sigma^N_k(t))$. Note that in the case where $A^N_k(t)=0$, we recover the above description of the situation prior to the first (re)infection.

The total force of infection in the population at time $t$ is
\begin{equation*} 
\mathfrak{F}^N(t):=\sum_{k=1}^N\lambda_{k,A^N_k(t)}(\sigma^N_k(t))\, .
\end{equation*}
According to the above description, we expect that the rate at which the individual $k$ gets infected is
\begin{equation*}
\Upsilon_k^N(t)=N^{-1}\gamma_{k,A^N_k(t)}(\sigma^N_k(t))\mathfrak{F}^N(t)=\gamma_{k,A^N_k(t)}(\sigma^N_k(t))\bar{\mathfrak{F}}^N(t)\,, 
\end{equation*}
where $\bar{\mathfrak{F}}^N(t)=N^{-1}\mathfrak{F}^N(t)$.

Let now $\{Q_k,\ 1\le k\le N\}$ be a collection of mutually independent standard Poisson random measures 
on $\R_+^2$.
We assume that the number of infections endured by the individual $k$ on the interval $(0,t]$ is given by
\[ A^N_k(t)=\int_0^t\int_0^\infty{\bf1}_{u\le\Upsilon_k^N(s^-)}Q_k(du,ds)\,.\]
We finally define the average susceptibility in the population as
\begin{equation*} 
\bar{\mathfrak{S}}^N(t)=N^{-1}\sum_{k=1}^N\gamma_{k,A^N_k(t)}(\sigma^N_k(t))\,.
\end{equation*}

We will show that the pair $(\bar{\mathfrak{S}}^N(t),\bar{\mathfrak{F}}^N(t))$ converges to a deterministic pair 
$(\bar{\mathfrak{S}}(t),\bar{\mathfrak{F}}(t))$, locally uniformly in time. Before we state this convergence result, let us study the limiting equation. Let 
$\bD_+$ denote the subset of $\bD$ consisting of non negative functions, $(x,y)\in \bD_+^2$ be a solution to the following set of equations: 
\begin{equation}\label{eqlimit}
\left\{
\begin{aligned}
x(t)&=\E\left[\gamma^0(t)\exp\left(-\int_0^t\gamma^0(r)y(r)dr\right)\right]\\&\quad
+\int_0^t\E\left[\gamma(t-s)\exp\left(-\int_s^t\gamma(r-s)y(r)dr\right)\right]x(s)y(s)ds,\\
y(t)&=\bar{I}(0)\bar{\lambda}^0(t)+\int_0^t\bar{\lambda}(t-s)x(s)y(s)ds\, .
\end{aligned}
\right.
\end{equation}

\begin{prop}
Equation \eqref{eqlimit} has a unique solution in $\bD_+^2$.
\end{prop}

\begin{proof}
We first prove uniqueness. We need an a priori bound on the solutions.
Suppose $(x,y)$ is a non negative solution of \eqref{eqlimit}, \textit{i.e.,} a solution satisfying $x(t)\ge0, y(t)\ge0$ for all $t\ge0$. 
Since $\gamma^0(t)\le 1$ and $\gamma(t)\le1$, we deduce from the first equation that
\begin{align} \label{tot_size_det}
	x(t)\le\E\left[\exp\left(-\int_0^t\gamma^0(r)y(r)dr\right)\right]
	+\int_0^t\E\left[\exp\left(-\int_s^t\gamma(r-s)y(r)dr\right)\right]x(s)y(s)ds\,.
\end{align}
If we multiply the first equation in \eqref{eqlimit} by $y(t)$, we obtain an identity which shows that the derivative with respect to $t$ of the right hand side of the above inequality is zero, hence that upper bound equals its value at time $t=0$, which is $1$. We have proved that $x(t)\le1$. Next from the second equation and Gronwall's Lemma we deduce that $y(t)\le\exp(\lambda^\ast t)$. With the help of those bounds, it is not very hard to show that \eqref{eqlimit} has at most one negative solution. 

Existence can be shown using a Picard iteration procedure, thanks to the estimates which are used for uniqueness. 
Note that the solution starts with $x(0)>0$ and $y(0)>0$, and it is not hard to see that neither $x$ nor $y$ can hit $0$ in finite time.
\end{proof}

We can now state the main result of this section. 
\begin{theorem}\label{th-LLN-VI-VS}
Under the above assumptions, as $N\to\infty$, 
$(\bar{\mathfrak{S}}^N,\bar{\mathfrak{F}}^N)\to(\bar{\mathfrak{S}},\bar{\mathfrak{F}})$ in $\bD^2$ in probability,
where $(\bar{\mathfrak{S}},\bar{\mathfrak{F}})$ is the unique solution of \eqref{eqlimit}.
\end{theorem}

In other words, the pair $(\bar{\mathfrak{S}},\bar{\mathfrak{F}})$ solves the system of integral equations
\begin{equation}\label{eqlimit2}
\left\{
\begin{aligned}
\bar{\mathfrak{S}}(t)&=\E\left[\gamma^0(t)\exp\left(-\int_0^t\gamma^0(r)\bar{\mathfrak{F}}(r)dr\right)\right]\\&\quad
+\int_0^t\E\left[\gamma(t-s)\exp\left(-\int_s^t\gamma(r-s)\bar{\mathfrak{F}}(r)dr\right)\right]\bar{\mathfrak{S}}(s)\bar{\mathfrak{F}}(s)ds,\\
\bar{\mathfrak{F}}(t)&=\bar{I}(0)\bar{\lambda}^0(t)+\int_0^t\bar{\lambda}(t-s)\bar{\mathfrak{S}}(s)\bar{\mathfrak{F}}(s)ds\, .
\end{aligned}
\right.
\end{equation}
\begin{remark}
The second equation in \eqref{eqlimit2} is equation \eqref{eqn-barfrakI}. In the particular case of the 
$\bS\bI\bR$ model, $\gamma^0(t)\equiv1$ if the individual is susceptible at time $0$, and $\equiv0$ otherwise.
Moreover, $\gamma(t)\equiv0$. In that case, the first equation in \eqref{eqlimit2} reduces to
\[ \bar{\mathfrak{S}}(t)=\bar{S}(0)\exp\left(-\int_0^t\bar{\mathfrak{F}}(s)ds\right).\]
So our new result is consistent with Theorem \ref{thm-FLLN-VI}.
\end{remark}

We will sketch the proof of this Theorem, refering the reader to \cite{FPPZ2021} for the details.
The main idea of this proof  is to  replace the collection 
$\{A^N_k,\ 1\le k\le N\}$ by an i.i.d. sequence of random processes, which will be close to that collection in an appropriate sense. The idea of the construction of that sequence is the following.
$A^N_k$ depends upon $N$ only through $\bar{\mathfrak{F}}^N$, which is a mean field interaction. 
We shall compare that sequence which an i.i.d. sequence, obtained by replacing $\bar{\mathfrak{F}}^N$ by its limit $\bar{\mathfrak{F}}$. 

Let $Q$ be a standard Poisson Random Measure on $\R_+^2$, and $(\lambda_i,\gamma_i)_{i\ge1}$ an i.i.d. sequence, each one having the law of $(\lambda_{1,1},\gamma_{1,1})$, and which is globally independent of $Q$.
Also take $ (\lambda_0, \gamma_0) $ independent of the previous sequence and distributed as $ (\lambda_{1,0}, \gamma_{1,0}) $.
To each deterministic $m\in \bD_+$, we associate the solution $A^{(m)}(t)$ of the following SDE:
\begin{equation}\label{eq:m}
\left\{
\begin{aligned}
A^{(m)}(t)&=\int_0^t\int_0^\infty{\bf1}_{u\le\Upsilon^{(m)}(s^-)}Q(ds,du),\\
\Upsilon^{(m)}(t)&=\gamma_{A^{(m)}(t)}(\sigma^{(m)}(t))\times m(t)\,,
\end{aligned}
\right.
\end{equation}
where 
\[ \sigma^{(m)}(t):=t-\sup\{s\in[0,t],\ A^{(m)}(s)=A^{(m)}(s^-)+1\}\vee0\,.\]
Let now 
\[\Psi^{(m)}(t):=\E\left[\lambda_{A^{(m)}(t)}(\sigma^{(m)}(t))\right], \quad \Theta^{(m)}(t):=\E\left[\gamma_{A^{(m)}(t)}(\sigma^{(m)}(t))\right]\,.\]
 The next Lemma is crucial for our proof.
 \begin{lemma}\label{le:m}
 The exists a unique $m^\ast\in \bD_+$ such that 
$\Psi^{(m^\ast)}=m^\ast$. Moreover, $(\Theta^{(m)},m)$ is the unique solution of \eqref{eqlimit} iff 
$\Psi^{(m)}=m$.
 \end{lemma}
\begin{proof}
Let us denote by $\{\tau^{(m)}_i, i\ge1\}$ the successive jump times of $A^{(m)}$. We have
\begin{align*}
\Psi^{(m)}(t)&=\E\left[\lambda_{A^{(m)}(t)}(\sigma^{(m)}(t))\right]\\
&=\E\left[\lambda_0(t)+\sum_{i=1}^{A^{(m)}(t)}\lambda_i(t-\tau^{(m)}_i)\right]\\
&=\bar{I}(0)+\bar{\lambda}^0(t)+\E\left[\int_0^t\bar{\lambda}(t-s)dA^{(m)}(s)\right]\\
&=\bar{I}(0)+\bar{\lambda}^0(t)+\int_0^t\bar{\lambda}(t-s)m(s)\Theta^{(m)}(s)ds
\end{align*}
Moreover 
\begin{align*}
\Theta^{(m)}(t)&=\E\left[\gamma_0(t){\bf1}_{A^{(m)}(t)=0}\right]+
\sum_{i\ge1}\E\left[\gamma_i(t-\tau^{(m)}_i){\bf1}_{\tau^{(m)}_i\le t}{\bf1}_{A^{(m)}(t)=i}\right]\\
&=\E\left[\gamma_0(t)\exp\left(-\int_0^t\gamma_0(r)m(r)dr\right)\right] \\
& \qquad +\sum_{i\ge1}\E\left[\gamma_i(t-\tau^{(m)}_i){\bf1}_{\tau^{(m)}_i\le t}
\exp\left(-\int_{\tau^{(m)}_i}^t\gamma_i(r-\tau^{(m)}_i)m(r)dr\right)\right]\\
&=\E\left[\gamma_0(t)\exp\left(-\int_0^t\gamma_0(r)m(r)dr\right)\right] \\
& \qquad +\int_0^t\int_D\gamma(t-s)\exp\left(-\int_s^t\gamma(r-s)m(r)dr\right)\mu(d\gamma)m(s)\Theta^{(m)}(s)ds\,.
\end{align*}
Let $(x,y)$ denote the unique solution of  \eqref{eqlimit}. Let us choose $m=y$. Then comparing the last identity to the first equation in \eqref{eqlimit}, we deduce that $\Theta^{(y)}=x$, and comparing the previous identity to the second equation of \eqref{eqlimit}, we deduce that $\Psi^{(y)}=y$. Conversely, if $\Psi^{(m)}=m$, we have that 
$(\Theta^{(m)},m)$ solves \eqref{eqlimit}, hence the result.
\end{proof}

We now make again use of the same sequence of independent PRMs $\{Q_k,\ k\ge1\}$, and for each $k\ge1$, we let $A_k$  be the $A^{(m^\ast)}$ associated to $Q_k$. More explicitly, we define for each $k\ge1$, 
\begin{equation}\label{eq:k}
\left\{
\begin{aligned}
A_k(t)&=\int_0^t\int_0^\infty{\bf1}_{u\le\Upsilon_k(s^-)}Q_k(ds,du),\\
\Upsilon_k(t)&=\gamma_{A_k(t)}(\sigma_k(t))\times \bar{\mathfrak{F}}(t)\,,
\end{aligned}
\right.
\end{equation}
where 
\[ \sigma_k(t):=t-\sup\{s\in[0,t],\ A_k(s)=A_k(s^-)+1\}\vee0\,.\]
Note that it follows from Lemma~\ref{le:m} that for each $k\ge1$, $\bar{\mathfrak{F}}(t)=\E\left[\lambda_{A_k(t)}(\sigma_k(t))\right]$, hence $\bar{\mathfrak{F}}(t)\le\lambda^\ast$.

The next step in the proof is the following Lemma.

\begin{lemma}\label{le:Nk-k}
For any $k\ge1$ and $T>0$, we have
\begin{equation*}
 \E\left[\sup_{0\le t\le T}|A^N_k(t)-A_k(t)|\right]\le \E\int_0^T|\Upsilon^N_k(t)-\Upsilon_k(t)|dt\le\frac{\lambda^\ast}{\sqrt{N}}T\exp(2\lambda^\ast T)\,.
 \end{equation*}
\end{lemma}
\begin{proof}
The first inequality is rather obvious. We now establish the second inequality. We will use repeatedly the fact that the r.v. $\sup_{0\le r\le t}|A^N_k(r)-A_k(r)|$ is either $0$ or else $\ge1$.
First note that
\begin{align*}
\E\left[|\Upsilon^N_k(t)-\Upsilon_k(t)|\right]&\le
\E\left[|\Upsilon^N_k(t)-\Upsilon_k(t)|{\bf1}_{A^N_k(t)=A_k(t), \sigma^N_k(t)=\sigma_k(t)}\right]
+\lambda^\ast\E\left(\sup_{0\le r\le t}|A^N_k(r)-A_k(r)|\right). 
\end{align*}
The first term on the right is bounded by
\begin{align*}
\E&\left[\left|\frac{1}{N}\sum_{j=1}^N(\lambda_{j,A^N_j(t)}(\sigma^N_j(t))-\E[\lambda_{1,A_1(t)}(\sigma_1(t))])\right|\right]\\
&\le\E\left[\left|\frac{1}{N}\sum_{j=1}^N(\lambda_{j,A^N_j(t)}(\sigma^N_j(t))-\lambda_{j,A_j(t)}(\sigma_j(t))])\right|\right] \\
& \qquad +\E
\left[\left|\frac{1}{N}\sum_{j=1}^N(\lambda_{j,A_j(t)}(\sigma_j(t))-\E[\lambda_{1,A_1(t)}(\sigma_1(t))])\right|\right]
\end{align*}
By a standard computation, the second term on the right hand side is bounded by $\lambda^\ast/\sqrt{N}$, and the first term by $\lambda^\ast\E\left(\sup_{0\le r\le t}|A^N_k(r)-A_k(r)|\right)$. 
If we define $\delta_N(t):=\E\left(\sup_{0\le r\le t}|A^N_k(r)-A_k(r)|\right)$, combining the above computations yields
\[ \delta_N(T)\le\frac{\lambda^\ast}{\sqrt{N}}T+2\lambda^\ast\int_0^T \delta_N(t)dt\,.\]
The result now follows from Gronwall's Lemma.
\end{proof}

\begin{proof}[Completing the proof of Theorem~\ref{th-LLN-VI-VS}]
The remainder of the proof of Theorem~\ref{th-LLN-VI-VS} can be sketched as follows.
 The r.v. whose expectation is close to $0$ by  Lemma~\ref{le:Nk-k} is $0$ with probability close to $1$ for large $N$. It is then not hard to deduce that
both
\[ \sup_{0\le t\le T}\left|\lambda_{k,A^N_k(t)}(\sigma^N_k(t))- \lambda_{k,A_k(t)}(\sigma_k(t))\right| \
\text{ and }\sup_{0\le t\le T}\left|\gamma_{k,A^N_k(t)}(\sigma^N_k(t))- \gamma_{k,A_k(t)}(\sigma_k(t))\right|\]
tend to $0$ in probability. Since the sequence $(A_k,\sigma_k)_{k\ge1}$  is i.i.d., the result essentially follows from the law of large numbers in $D$, see \cite{rao1963law}.
\end{proof}

\begin{remark}
While the random infectivity appears in the limiting LLN deterministic equations only through its mean function $\bar{\lambda}(t)$, a complicated mixed moment--exponential moment of the trajectory of the random susceptibility $\gamma(t)$ appears in the deterministic version of our varying infectivity / varying susceptibility model. This is due to the possibility of reinfection of the individuals who are experiencing a graduate loss of their immunity / gain of their susceptibility. 
\end{remark}

Using the same techniques, we can also obtain the limiting equations for the proportion of susceptible and infectious individuals.
As before, we let
\begin{align*}
	\eta_{k,i} = \sup\lbrace t \geq 0, \lambda_{k,i}(t) > 0 \rbrace,
\end{align*}
and
\begin{align*}
	F_0^c(t) = \P(\eta_{1,0} > t), && F^c(t) = \P(\eta_{1,1} > t).
\end{align*}
Then define
\begin{align*}
	I^N(t) = \sum_{k=1}^{N} {\bf 1}_{\sigma^N_k(t) < \eta_{k,A^N_k(t)}},
\end{align*}
\textit{i.e.,} the number of infectious individuals at time $ t $ (recall that $ \eta_{k,0} = 0 $ if the $ k $-th individual is initially susceptible).
Also set
\begin{align*}
	S^N(t) = \sum_{k=1}^{N} {\bf 1}_{\sigma^N_k(t) \geq \eta_{k,A^N_k(t)}} = N-I^N(t).
\end{align*}
Then, setting $ \bar{I}^N(t) = \frac{1}{N} I^N(t) $ and $ \bar{S}^N(t) = \frac{1}{N} S^N(t) $, we have the following convergence.

\begin{coro}\label{cor:proportions_VIVS}
	Under the assumptions of Theorem~\ref{th-LLN-VI-VS}, as $ N \to \infty $, $ (\bar{S}^N, \bar{I}^N) \to (\bar{S}, \bar{I}) $ in $ \bD^2 $ in probability, where
	\begin{align} \label{eqn-VIVS-IS}
		&\bar{I}(t) = \bar{I}(0) F^c_0(t) + \int_{0}^{t} F^c(t-s) \bar{\mathfrak S}(s) \bar{\mathfrak F}(s) ds, \\
		&\bar{S}(t) = \E\left[ {\bf 1}_{t \geq \eta_{0}} \exp\left( - \int_{0}^{t} \gamma_0(r) \bar{\mathfrak F}(r) dr \right) \right] + \int_{0}^{t} \E\left[ {\bf 1}_{t \geq \eta} \exp\left( - \int_{s}^{t} \gamma(r-s) \bar{\mathfrak F}(r) dr \right) \right] \bar{\mathfrak S}(s) \bar{\mathfrak F}(s) ds, \nonumber
	\end{align}
	where $ (\eta_0, \gamma_0) $ is distributed as $ (\eta_{1,0}, \gamma_{1,0}) $ and $ (\eta, \gamma) $ as $ (\eta_{1,1}, \gamma_{1,1}) $.
\end{coro}

Note that we can recover the fact that $ \bar{I}(t) + \bar{S}(t) = 1 $ for all $ t \geq 0 $ from the fact that $ \gamma(s) = 0 $ for all $ s < \eta $ (resp. $ \gamma_0(s) = 0 $ for all $ s < \eta_0 $) and the fact that the right hand side of \eqref{tot_size_det} equals 1.

\begin{proof}[Proof of Corollary~\ref{cor:proportions_VIVS}]
	This convergence follows from Lemma~\ref{le:Nk-k} in much the same way as the convergence of $ \bar{\mathfrak S} $ and $ \bar{\mathfrak F} $.
	Lemma~\ref{le:Nk-k} implies that, for any fixed $ k $,
	\begin{align*}
		\sup_{t \in [0,T]} \left| {\bf 1}_{\sigma^N_k(t) \geq \eta_{k,A^N_k(t)}} - {\bf 1}_{\sigma_k(t) \geq \eta_{k,A_k(t)}} \right|
	\end{align*}
	tends to zero in probability as $ N \to \infty $, and the convergence of $ \bar{I}^N $ follows from arguments similar to those in the proof of Theorem \ref{th-LLN-VI-VS}.
	The convergence of $ \bar{S}^N $ then follows from the fact that $ \bar{S}^N(t) = 1-\bar{I}^N(t) $ and $ \bar{S}(t) = 1-\bar{I}(t) $.
\end{proof}

\begin{remark}
The novelty in this approach is the construction of a sequence of i.i.d. processes to invoke the law of large numbers for processes in $\bD$, by using the solution of a MacKean-Vlasov type Poisson-driven stochastic equation (as in the propagation of chaos theory).
This approach has been adapted to prove the LLN for the homogeneous model with varying infectivity in \cite{FPP2021-MPMG}, which requires much weaker conditions on the random infectivity functions than that in \cite{FPP2020b}. It is then used to prove a LLN for the multi-patch multi-group epidemic model with varying infectivity in \cite{FPP2021-MPMG}. 
\end{remark}

\begin{remark}
The integral equations in \eqref{eqlimit2} and  \eqref{eqn-VIVS-IS} can be related to the PDE models as first proposed in
Kermack and McKendrick \cite{KMK32}, when a particular set of initial conditions is used and
the varying infectivity and susceptibility random functions are deterministic functions of the time since infection and the random duration of the infectious period.  
We refer the readers to Section~5 of \cite{FPPZ2021} for this detailed discussion for the SIRS model. 
\end{remark}

The paper \cite{FPPZ2021} also contains results on the endemic equilibrium of the model \eqref{eqlimit2}. Those results hold under additional assumptions on the pairs $(\lambda_{k,i},\gamma_{k,i})$, which we do not detail. Define
$\gamma_\ast:=\lim_{t\to\infty}\gamma(t)$. If $R_0\le \E\left[\frac{1}{\gamma_\ast}\right]$, then the disease free equilibrium is the only equilibrium and it is stable.
If however $R_0> \E\left[\frac{1}{\gamma_\ast}\right]$, then there is a unique endemic equilibrium $(\bar{\mathfrak{S}}_\ast, \bar{\mathfrak{F}}_\ast)$, which can be specified as follows: $\bar{\mathfrak{S}}_\ast=1/R_0$, and $\bar{\mathfrak{F}}_\ast$ is the only positive solution $x$ of the equation
\[ \int_0^\infty\E\left[\exp\left(-\int_0^s\gamma\left(\frac{r}{x}\right)dr\right)\right]ds=R_0\,.\]
We have not been able to show  that this endemic equilibrium is stable, but we have shown that under some severe assumptions the disease free equilibrium is unstable whenever $R_0> \E\left[\frac{1}{\gamma_\ast}\right]$.

%
%
%

\section{Other types of models and open problems}\label{sec:open}

\subsection{Non--homogeneous models}
All the models presented so far are homogeneous, in the sense that whenever one infectious individual meets someone else, anyone in the population has the same chance to be met, and to be possibly infected if he/she was susceptible (with the exception of  the multipatch model of Section~\ref{multipatch} and of the last section, where the susceptibility of the various individuals plays a role in that choice, but it does not contradict the homogeneity). There are many reasons why this is not realistic, and we shall indicate several attempts to correct the homogeneous model, and make the epidemic models more realistic. However, the reader should realize that too complicated models may not be really useful, among other reasons because they involve too many parameters, which might not be easy to estimate.

A first complexification of the above models is to distribute the population into age groups. This is quite reasonable concerning the Covid -- 19 epidemic, since the proportion of severe cases and deaths among those who catch the disease depends very much upon the age. This is not too difficult to implement, provided one can exploit informations about the contact rates between those age groups,  which might be found in the sociology literature. 
Also, in several countries data concerning the numbers of hospitalized patients, those in intensive care units, and those who die, are available by age class. Note however that the health condition of the patients (and their weight) is almost as important as their age.

It has been recently pointed out, see \cite{BBT-Science}, that the heterogeneity in levels of social activity between individuals has an impact on the herd immunity. Indeed, those who have a higher rate of social contacts have more chances to get infected towards the beginning of the epidemic, and they will infect more people than those with a lower rate of social contacts. Once most of the socially very active individuals have been infected and become immune, one may think that the progress of the epidemic might slow down. For that reason, an epidemic which infects a certain percentage of the population will be more efficient towards building herd immunity than vaccinating the same proportion of individuals. 
This raises also questions concerning a vaccination campaign. Should one vaccinate first those more at risk, but who have few social contacts and do not contribute much to the propagation of the epidemic, or rather those who have much social contacts?

The reason why  reality is not homogeneous has to do with the fact that each individual has frequent contact with those  in his close environment (those who share the same household and workplace), and much less frequent contact with people who are met in public transportation, shops, various social activities. There has been a lot of effort to adapt the epidemic models to such situation with various levels of contact rates. 
\cite{ballsirl} contains a recent review of the works in that direction. See also \cite{FP2021} for a mean field model approach to household epidemic models.

%
%

\subsection{Spatial models}
A good reason for non homogeneity is the spatial dispersion of the population. There is of course a strong motivation for studying epidemics models for population distributed in discrete or continuous space. There is a  quite significant body of literature on deterministic models in those two situations, see in particular \cite{ABLN2007}, \cite{ABLN2008}, \cite{rass2003spatial} and  reviews in Chapter 15 of \cite{martcheva2015} and Chapter 14 of \cite{BCF-2019}. 
Recently, some authors have proved law of large numbers and central limit theorems for Markov epidemic models  in continuous space, see in particular \cite{nzipardouxyeo} and \cite{bowongemakouapardoux}. Concerning non--Markov models, we have recently studied such a model in discrete space, see Section \ref{multipatch}. 

There is also an extensive literature on epidemic models on random graphs, including various limit theorems and asymptotic results with large population, large graphs, dense graphs/graphons, hypergraphs and various network topologies,  see, e.g., \cite{barbour1990epidemics,andersson1998limit,decreusefond2012large,ball2014epidemics,janson2014law,miller2014epidemics,janson2017near,fransson2019sir,van2010critical,keliger2022local,aurell2022finite,bodo2016sis,higham2021epidemics}, and also the recent survey \cite{tran}.  As far as we know, most of these works are for Markovian models. It would be interesting to investigate non-Markovian epidemic transmissions on random graphs.

\subsection{Control problems}

Optimal control problems in the Markovian and limiting ODE epidemic models have been studied extensively in the literature. Isolation, vaccination and immunization strategies have been developed to minimize the epidemic size and costs associated with the implementation of them. 
For example, optimal isolation strategy to minimize the total number of infected individuals \cite{abakuks1973optimal} or to minimize the total infectious  burden over an outbreak together with a cost for implementing the control in \cite{wickwire1975optimal,morton1974optimal},  optimal vaccination strategy \cite{abakuks1974optimal,wickwire1975optimal,morton1974optimal}, optimal immunization strategy \cite{wickwire1979optimal},   optimal combined isolation-vaccination strategy with resource constraints \cite{hansen2011optimal}, and optimal control to minimize the total number of infectious and the time needed for the infection to go extinct \cite{bolzoni2014react, bolzoni2019optimal}. To cope with Covid-19 pandemic, various lockdown, social distancing, testing and vaccination strategies have been implemented by governments. Some studies have been conducted of their effect, see, e.g., \cite{djidjou2020optimal,elie2020contact,xu2021control}. Incentives for individuals to participate in the mitigation process are also studied from the game theory perspective, see, e.g., \cite{hubert2020incentives, aurell2020optimal,kordonis2021dynamic}. 
It would be interesting to study how robust these control strategies are to the Markovian assumption, in particular, when the infectious periods are assumed to have a general distribution rather than exponential.

\subsection{Open problems}

There is clearly a need for more work on non--Markovian spatial epidemic models, both in discrete and continuous space. Also, endemic situations should be studied in the varying infectivity / varying susceptibility situation which we have exposed in section \ref{sec:varinfect-varimmun}. Then the study of large and moderate deviations from the limiting deterministic LLN model opens new questions. We expect to address these questions in future work. 

\section{Appendix}\label{sec:appendix}
\subsection{Poisson processes and Poisson Random Measures}\label{subsec:Poisson}
A standard Poisson process $P(t)$ is a counting process (a process which counts a number of events which has happened during the interval $[0,t]$), which is such that $P(0)=0$, $P$ has independent increments\footnote{This means that for any $n\ge1$, any $0=t_0<t_1<\cdots<t_n$, 
$P(t_1),P(t_2)-P(t_1),\ldots,P(t_n)-P(t_{n-1})$ are independent. } and for any $0\le s<t$, 
the law of $P(t)-P(s)$ is Poi$(t-s)$. Equivalently, for any $t\ge0$, the time after $t$ until the next event is independent of what happened before $t$ and its law is $\mathcal{E}\text{xp}(1)$. If $P(t)$ is a standard Poisson process and $\lambda>0$, $P(\lambda t)$ is a rate $\lambda$ Poisson process (\textit{i.e.,} for $s<t$, $P_{inf}(t)-P_{inf}(s)\simeq\text{Poi}(\lambda(t-s))$, the waiting time until the next event after $t$ is $\mathcal{E}\text{xp}(\lambda)$. More generally, for a deterministic function $\lambda(t)$, $P\left(\int_0^t\lambda(s)ds\right)$ is a rate $\lambda(t)$ Poisson process. 

A Poisson Random Measure (abbreviated PRM) $Q$ on a measurable set $E$ with mean measure $\mu$ is a  sum of Dirac measures at random points, which is such that the number of those points in disjoint subsets are independent, and for any measurable set $A$, $Q(A)\simeq\text{Poi}(\mu(A))$. A PRM $Q$ on a subset of $\R^d$ will be called standard if its mean measure is the Lebesgue measure. Note that what we have called above a standard Poisson process is the distribution function of a standard PRM on $\R_+$.

It is not very hard to show that if $\lambda(t)$ is a measurable locally bounded $\R_+$--valued function, then the two processes
\[ P\left(\int_0^t\lambda(s)ds\right)\quad\text{and }\int_0^t\int_0^\infty{\bf1}_{u\le\lambda(s)}Q(ds,du),\]
where $P$ is a standard Poisson process and $Q$ a standard Poisson random measure on $\R_+^2$, have the same law (i.e., the same finite dimensional distributions).

\subsection{Brownian motion and space--time white noise}\label{sec:BM-WN} 
A standard Brownian motion $\{B(t),\ t\ge0\}$  is a Gaussian process with continuous paths and independent increments, and such that for any $t\ge0$, $B(t)\sim N(0,t)$, i.e., $B(t)$ is a Gaussian r.v. with mean $0$ and variance $t$. A non standard Brownian motion could have a non zero mean, and a different variance.

We use in the statement of Theorem \ref{thm-FCLT-SIR} the notion of a white noise on $\R^2$. A standard white noise $W$ on $\R^2$ is a generalized Gaussian process $\{W(f),\ f\in L^2(\R^2)\}$ whose law is specified by the fact that $f\mapsto W(f)$ is linear, and $W(f)\sim N(0,\| f\|^2_{L^2(\R^2)}$. Equivalently, for any Borel subset 
$A\subset\R^2$ with finite Lebesgue measure, $W(A):=W({\bf1}_A)\sim N(0,\text{Leb}(A))$. A non standard white noise on $\R^2$ is associated with a measure $\mu$ on $\R^2_+$, such that $W(A)\sim N(0,\mu(A))$. The law of $W(A)$ is specified, provided $\mu(A)<\infty$.

\subsection{The space $\bD$}\label{subsec:D}
In this paper, we denote by $\bD:=\bD([0,+\infty))$ the space of functions from $[0,+\infty)$ into 
$\R$ which are right continuous and possess a left limit at any time $t>0$. Such a function is said to be c\`al\`ag, an acronym for {\it continu \`a droite et limit\'e \`a gauche}. If $x\in \bD$, whenever $t_n\to t$, with 
$t_n\ge t$ for any $n\ge1$, $x(t_n)\to x(t)$, and we shall write $x(t^-)$ for the value of $\lim_n x(t_n)$, whenever $t_n<t$ for all $n\ge1$. It is not convenient to equip $\bD$ with the supnorm topology, since we want that after a small modification of the time of a jump, the resulting function be close to the original one.  

A sequence converges in $\bD$ iff it converges in $\bD([0,T])$ for all $T>0$. It then suffices to discuss the convergence in $\bD([0,T])$. A distance on $\bD([0,T])$ can be defined as follows. Let $\Lambda_T$ denote the set of continuous strictly increasing functions from $[0,T]$ into itself, which map $0$ into $0$ and $T$ into $T$. If $\|\cdot\|_T$ denotes the supnorm on $[0,T]$ and $I$ the identity mapping, a possible choice for the distance is
\[ d(x,y)=\inf_{\lambda\in\Lambda}\{\|\lambda-I\|_T\vee\|x-y\circ\lambda\|_T\}\,.\]
The associated topology is sometimes called the Skorokhod $J_1$ topology. That distance makes $\bD([0,T])$ separable. If we want $\bD([0,T])$ to be complete, we better use a slightly different distance, whose definition
is given by replacing $\|\lambda-I\|_T$ by $\sup_{0\le s<t\le T}\left|\log\frac{\lambda(t)-\lambda(s)}{t-s}\right|$.

It is crucial for us to have conditions under which a sequence of stochastic processes with  trajectories in $\bD$ is tight, which implies that such a sequence has a subsequence which converges weakly for the topology of $\bD$.
Let us formulate the celebrated Aldous tightness criterion. A sufficient condition for a sequence $X^n$ of random elements of $\bD$ to be tight is that the two following conditions are satisfied:
\begin{enumerate}
\item For any $T>0$, $\limsup_{n}\P\left(\|X^n\|_T\ge a\right)\to0$, as $a\to\infty$.
\item For any $\varepsilon$, $\eta$, $T>0$, there exists $\delta_0>0$ and $n_0$ such that if $\delta\le\delta_0$ and $n\ge n_0$, for any discrete $X^n$--stopping time $\tau\le T$, 
$ \P\left(|X^n(\tau+\delta)-X^n(\tau)|\ge\varepsilon\right)\le\eta$.
\end{enumerate} 
A proof of this criterion can be found e.g. on pages 178-179 of \cite{billingsley}.

In the case a-of a semi--martingale, we have a very simple criterion to verify Aldous's condition.   The following is Proposition 37 in \cite{pardoux2016probabilistic}:
\begin{prop}\label{pro:tightD}
Let $X^n$ be a sequence of semimartingales of the form
\begin{align*}
X^n(t)&= X^N_0+\int_0^t\varphi^n(s)ds+M^n(t),\ \text{ and}\\
\langle M^n\rangle_t&=\int_0^t \psi^n(s)ds,
\end{align*}
where $M^n(t)$ is a martingale, and $\langle M^n\rangle_t$ its associated predictable increasing process (i.e.,
$\langle M^n\rangle_t$ is predictable and $|M^n(t)|^2-\langle M^n\rangle_t$ is a martingale). 

If both $\{X^n_0\}$, and $\{\sup_{0\le t\le T}(|\varphi^n(t)|+\psi^n(t))\}$ are tight for all $T>0$, then $X^n$ is tight in $\bD$.
\end{prop}

Tightness of semimartingales is usually not too hard to establish. With Markov processes are always associated martingales. However, with our non--Markov processes, we do not necessary have martingales to help us. This is 
why more delicate techniques are involved in the proofs for the non--Markov processes.
 
	\bibliographystyle{plain}
	\bibliography{Epidemic-Survey}

\end{document}